\documentclass[12pt,reqno]{amsart}

\usepackage[utf8]{inputenc}
\usepackage[T1]{fontenc}
\usepackage{amsmath,amssymb,amsthm,mathtools,mathrsfs}
\usepackage{geometry}
\usepackage[colorlinks=true,citecolor=blue,linkcolor=blue,urlcolor=blue]{hyperref}
\usepackage{enumitem}
\usepackage[expansion=false]{microtype}
\usepackage{booktabs}
\usepackage{tabularx}
\usepackage{aliascnt}
\usepackage{cleveref}
\crefname{assumption}{Assumption}{Assumptions}
\Crefname{assumption}{Assumption}{Assumptions}

\theoremstyle{plain}
\newtheorem{theorem}{Theorem}[section]
\newaliascnt{proposition}{theorem}
\newtheorem{proposition}[proposition]{Proposition}
\aliascntresetthe{proposition}
\crefname{proposition}{proposition}{propositions}
\Crefname{proposition}{Proposition}{Propositions}
\newaliascnt{lemma}{theorem}
\newtheorem{lemma}[lemma]{Lemma}
\aliascntresetthe{lemma}
\crefname{lemma}{lemma}{lemmas}
\Crefname{lemma}{Lemma}{Lemmas}
\newaliascnt{corollary}{theorem}
\newtheorem{corollary}[corollary]{Corollary}
\aliascntresetthe{corollary}
\crefname{corollary}{corollary}{corollaries}
\Crefname{corollary}{Corollary}{Corollaries}

\theoremstyle{definition}
\newaliascnt{definition}{theorem}
\newtheorem{definition}[definition]{Definition}
\aliascntresetthe{definition}
\crefname{definition}{definition}{definitions}
\Crefname{definition}{Definition}{Definitions}
\newaliascnt{example}{theorem}
\newtheorem{example}[example]{Example}
\aliascntresetthe{example}
\crefname{example}{example}{examples}
\Crefname{example}{Example}{Examples}
\newaliascnt{openquestion}{theorem}
\newtheorem{openquestion}[openquestion]{Open Question}
\aliascntresetthe{openquestion}
\crefname{openquestion}{open question}{open questions}
\Crefname{openquestion}{Open Question}{Open Questions}
\newtheorem{assumption}{Assumption}
\renewcommand{\theassumption}{H\arabic{assumption}}

\theoremstyle{remark}
\newaliascnt{remark}{theorem}
\newtheorem{remark}[remark]{Remark}
\aliascntresetthe{remark}
\crefname{remark}{remark}{remarks}
\Crefname{remark}{Remark}{Remarks}

\newcommand{\R}{\mathbb{R}}

\newcommand{\Lf}{L_f}
\newcommand{\grad}{\nabla}
\newcommand{\Lap}{\Delta}
\newcommand{\vol}{\operatorname{Vol}}
\newcommand{\dV}{\,dV}

\newcommand{\ef}{e^{-f}}

\newcommand{\bdy}{\partial_\nu}

\newcommand{\abs}[1]{\left|#1\right|}
\newcommand{\ip}[2]{\langle #1,\, #2\rangle}
\newcommand{\fdom}{L^2(M,e^{-f}dV)}
\newcommand{\fint}{\mathop{\,\rlap{-}\!\!\int}}

\begin{document}

\title[Compatibility of Spectral and Growth Filtrations for Drift Laplacians]{%
Two Growth Filtrations for Drift Laplacians:\\
Compatibility, Rigidity, and an Inverse Hermite/Laguerre Theorem}

\author{Elham Matinpour}
\address{Department of Mathematics and Statistics, Auburn University,
Auburn, AL 36849, USA}
\email{elm0135@auburn.edu}

\begin{abstract}
Let $(M^n,g)$ be a complete, connected, noncompact Riemannian manifold,
and let $f\ge0$ be a proper $C^2$ weight. We assume polynomial bounds on $f$, $|\nabla f|$, and $\Delta f$
(\crefrange{ass:H1}{ass:H3}), with growth exponent $\alpha>0$.
We study the drift Laplacian
$L_f=\Delta-\langle\nabla f,\nabla\cdot\rangle$ on
$L^2(M,e^{-f}dV)$. We assume that its eigenfunctions have finite polynomial growth orders
$\gamma_k$, measured on the level sets of $f^{1/\alpha}$.
We also assume that $\gamma_k\to\infty$ in spectral order (\Cref{ass:H4}).
We compare this growth filtration with the filtration by eigenvalue
$\lambda_k$. For $\alpha>1$, we study the relation
$\lambda_k\asymp\gamma_k^{(2\alpha-2)/\alpha}$, called \emph{compatibility}.
An explicit rotationally symmetric example shows that this relation
can fail for the full spectrum, even when $\alpha=2$.

We prove discreteness and weighted Agmon estimates for $\alpha>1$.
A moment estimate bounds the concentration radius in terms of the
growth order and a finite-scale prefactor. A lower localization
condition gives one direction of the spectral comparison.
For exact warped products with a one-dimensional base, we impose the
radial growth condition and a regularity assumption on the radial drift.
These give $\alpha=2$ and $\lambda_k\asymp\gamma_k$ in the radial sector.
For general $(M,g)$, we assume a transitive isometric symmetry of the
level sets instead of a warped-product structure. The invariant-sector
growth condition and radial-drift regularity then give $\alpha=2$ and
$\lambda_k\asymp\gamma_k\asymp k$.
Finally, exact polynomiality of the invariant eigenfunctions implies
$\alpha=2$ without the drift regularity assumptions. The reduced
equation is then Hermite or generalized Laguerre after normalization.
The arguments use neither curvature bounds nor soliton equations.
\end{abstract}

\maketitle
\tableofcontents

\section{Introduction}\label{sec:intro}

\subsection{Background and motivation}

Let $(M^n,g)$ be a complete, connected, noncompact Riemannian manifold and
$f:M\to[0,\infty)$ a $C^2$ proper function. The drift Laplacian, also called the weighted or Bakry--\'Emery
Laplacian, is the operator
\[
  L_fu \;=\; \Delta u - \langle\nabla f,\nabla u\rangle .
\]
It is formally symmetric on $L^2(M,e^{-f}dV)$, with inner product
$\langle u,v\rangle_f=\int_Muv\,e^{-f}dV$. It arises in diffusion
theory and Bakry--\'Emery geometry \cite{BakryEmery1985}, and in the
study of self-shrinkers for mean curvature flow.

The basic example is the Ornstein--Uhlenbeck operator with
$f(x)=|x|^2/4$. On $\R$, its degree-$k$ Hermite eigenfunction grows
like $|x|^k$ and has eigenvalue $\lambda_k=k/2$. On $\R^n$, tensor products of the one-dimensional Hermite
eigenfunctions form a complete eigenbasis. For $n\ge2$, the $SO(n)$-invariant sector is described
by Laguerre polynomials in $|x|^2$. In either one-dimensional indexing, the eigenvalue and polynomial-growth
order are linear in the mode number. The degree in $|x|$ is twice the
Laguerre degree.

Colding and Minicozzi \cite{ColdingMinicozzi1997} proved that, on a
complete manifold with nonnegative Ricci curvature, harmonic functions
of each fixed polynomial growth degree form a finite-dimensional
space. Wu and Wu \cite{WuWu2023} proved the corresponding finiteness
for solutions of $(\Delta-aR)u=0$, $a>0$, on gradient Ricci shrinkers.
Here $R$ is the scalar curvature. They also proved that every
$f$-harmonic function of polynomial growth on such a shrinker is constant.

Colding and Minicozzi \cite{ColdingMinicozzi2021} obtained sharp
growth bounds for drift eigenfunctions under the identities
$\Delta f+S=n/2$ and $|\nabla f|^2+S=f$, with $f,S\ge0$ and $f$ proper.
For an eigenfunction of eigenvalue $\lambda$, their normalized
level-set $L^2$ quantity is $O(r^{4\lambda})$. The exhaustion
$b=2\sqrt f$ is $1$-Lipschitz, so
$f(x)\le\tfrac14(d(x,p)+2\sqrt{f(p)})^2$.
In this setting properness is an additional assumption rather than
a consequence of the two identities. For gradient Ricci shrinkers, two-sided quadratic
growth of $f$ follows from Cao and Zhou \cite{CaoZhou2010}.

This paper studies the converse direction. We assume no soliton or
shrinker structure on $(M,g)$ or $f$, and no curvature bound. The only restrictions on the volume growth of $(M,g)$ come from
Assumptions~\ref{ass:H1}--\ref{ass:H3} on the weight and its derivatives.
They give $\vol(\Sigma_r)\le e^{\eta r^\alpha+O(\log r)}$.
Exponential growth is allowed once $\alpha\ge1$. The question is the
following:

\begin{center}
If the eigenfunctions of $L_f$ grow polynomially in a suitable
intrinsic sense and their growth orders increase without bound, what
does this imply about the asymptotic shape of $f$?
\end{center}

\subsection{Two filtrations and their compatibility}

Let $\lambda_k$ be the eigenvalues of $-L_f$, counted with multiplicity.
Let $\gamma_k$ be the polynomial growth orders in Assumption~\ref{ass:H4}.
We call the two filtrations compatible if
\[
  \lambda_k\asymp\gamma_k^{(2\alpha-2)/\alpha}
\]
for all sufficiently large $k$. See \Cref{def:compatibility}.
This does not follow from \ref{ass:H1}--\ref{ass:H4}. In \Cref{prop:incompatibility},
both $f$ and the negative logarithm of the weighted radial density
are quadratic at infinity,
but angular energy makes $\lambda_k/\gamma_k$ unbounded.
The radial subsequence is compatible. We ask whether every \ref{ass:H1}--\ref{ass:H4} manifold admits a
compatible subsequence in \Cref{q:compatiblesubsequence}.

We compare growth and spectrum through two radii. The Agmon safe radius
$r_{\mathrm{Ag}}(k)$ is a scale beyond which the conjugated potential
exceeds $\lambda_k$. The concentration radius $r_C(k)$ is the median
of the weighted mass of $u_k$. We prove
$r_C(k)\lesssim r_{\mathrm{Ag}}(k)$.
If the Lower Localization Property (LLP) also holds, then
$r_C(k)\gtrsim\gamma_k^{1/\alpha}$, leading to
$\lambda_k\gtrsim\gamma_k^{(2\alpha-2)/\alpha}$.
The moment estimate provides another upper bound on $r_C(k)$.
This bound also depends on a finite-scale prefactor not specified by $\gamma_k$.

Compatibility and linear growth of $\lambda_k$ and $\gamma_k$ in the
index of a common subsequence force $\alpha=2$
(\Cref{cor:alpha2conditional}). Full linear indexing holds for the
one-dimensional Ornstein--Uhlenbeck model and for its radial sectors.
The full Gaussian spectrum, counted with multiplicity, instead has
$\lambda_k\asymp\gamma_k\asymp k^{1/n}$ when $n\ge2$.
Choosing one eigenfunction from each degree produces a subsequence
satisfying \Cref{ass:linearcorr}.

\subsection{Main results}

We state the main theorems here. The hypotheses are made precise in
\S\ref{sec:ass}, \S\ref{sec:compatibility}, and \S\S\ref{sec:warped}--\ref{sec:sturm}.

\begin{theorem}\label{thm:informal1}
Under \Cref{ass:H1,ass:H2,ass:H3,ass:H4} with $\alpha>1$, the following hold:
\begin{enumerate}[label=\rm(\alph*)]
\item (\S\ref{sec:conjugate}) $L_f$ has purely discrete
  spectrum, and its Schr\"odinger-conjugated eigenfunctions satisfy Agmon estimates.
  The decay rates depend on the constants in \Cref{ass:H1,ass:H2,ass:H3}.
  The prefactors may also depend on the fixed metric and weight.
\item (\S\ref{sec:twofiltrations}) the concentration radius satisfies
  the prefactor-dependent moment bound \eqref{eq:concUBmoment}, and
  $r_C(k)\lesssim r_{\mathrm{Ag}}(k)$ for all sufficiently large $k$.
\item (\S\ref{sec:localization}) the Lower Localization Property
  (LLP, \Cref{def:LLP}), together with
  $r_C(k)\lesssim r_{\mathrm{Ag}}(k)$, gives
  $\lambda_k\gtrsim\gamma_k^{(2\alpha-2)/\alpha}$
  (\Cref{thm:reduction}).
\end{enumerate}
\end{theorem}

\begin{remark}
LLP gives the lower spectral inequality. \ref{ass:H1}--\ref{ass:H4}
alone do not imply the upper inequality for the full spectrum, as
\Cref{prop:incompatibility} shows.
Whether \ref{ass:H1}--\ref{ass:H4} imply LLP is \Cref{prob:LLPunconditional}.
\end{remark}

\begin{theorem}\label{thm:introwarped}
Let $(M,g)$ be a complete exact warped product with one-dimensional
base and closed connected fiber, as in \S\ref{sec:warped}.
Let $f$ be radial and satisfy \Cref{ass:H1,ass:H2,ass:H3}.
On each noncompact end, assume $\Theta'(r)=O(r^{\alpha-2})$, where
$r$ is outward arclength and $\Theta=f'-(n-1)\varphi'/\varphi$.
If the radial eigenfunction sector satisfies \Cref{ass:H4}, then
$\alpha=2$ and $\lambda_k\asymp\gamma_k$ for every $k$ with
$\lambda_k>0$. In the one-ended case, $\gamma_k=c_*^{-1}\lambda_k$
for every $k$,
with $c_*>0$ independent of $k$ (\Cref{thm:wprigidity}).
\end{theorem}

\ref{ass:H3} cannot be dropped here. For every $n\ge4$,
\Cref{thm:H4notalpha2} provides a complete smooth rotationally symmetric
metric on $\mathbb R^n$ with a radial weight $f\asymp r^{n-1}$
satisfying \ref{ass:H1}, \ref{ass:H2} and \ref{ass:H4} in full. Its eigenvalues are proportional to
the growth orders of the eigenfunctions. There \ref{ass:H3} fails and
$\alpha=n-1\ne2$.

A sharper conditional statement applies to this example. On an exact
warped product, suppose the radial drift $\Theta=f'-V_r$ has a regular
effective power-law exponent in the sense of \Cref{def:beta}.
Then radial \ref{ass:H4} implies that the exponent is $2$ (\Cref{thm:betarigidity}). Equivalently, the radial weight
$W(r)=\varphi(r)^{n-1}e^{-f(r)}$ then satisfies $-\log W\asymp r^2$.
Under \ref{ass:H3} and the drift regularity above,
\Cref{lem:wpanticancel} shows that the effective exponent agrees with
$\alpha$, so $\alpha=2$.

The uniqueness statement uses the same effective-exponent hypothesis.
Consider an exact warped product with $\varphi$ bounded below on every
end. Suppose two radial weights satisfy \ref{ass:H1}, \ref{ass:H2} and
radial \ref{ass:H4}, and their drifts have regular effective exponents.
Each weight differs from the logarithm of the warped-slice volume by
a quadratic term. These terms are comparable on each end, so the weights
are comparable. They are asymptotic to each other if the slice volume grows faster than
$e^{cr^2}$ for every $c>0$ on every end (\Cref{cor:H4uniqueness}).

\begin{theorem}\label{thm:informal2}
Let $(M^n,g)$ be a complete, connected, noncompact Riemannian manifold,
and let $f$ satisfy \Cref{ass:H1,ass:H2,ass:H3,ass:H4}.
No warped-product structure is assumed. Then:
\begin{enumerate}[label=\rm(\alph*)]
\item \Cref{ass:H5} implies LLP, hence
  $r_C(k)\gtrsim\gamma_k^{1/\alpha}$ and, for $\alpha>1$,
  $\lambda_k\gtrsim\gamma_k^{(2\alpha-2)/\alpha}$
  (\Cref{prop:H5consequence} and \Cref{thm:reduction}).
\item for $\alpha>1$, under a transitive isometric symmetry
  (\Cref{ass:transitive}) and the regularity \eqref{eq:Dregularity}
  of the radial drift, $\alpha=2$, and the $G$-invariant sector
  satisfies $\lambda_k\asymp\gamma_k\asymp k$
  (\Cref{thm:osccount}, \Cref{cor:SL-compat}).
\end{enumerate}
\end{theorem}

\begin{remark}
Part (b) does not assume \ref{ass:H5} or LLP.
None of the arguments above establishes full compatibility in general. Part
(a), together with the unconditional
$r_C(k)\lesssim r_{\mathrm{Ag}}(k)$ of \Cref{thm:informal1}(b),
gives the lower spectral bound
$\lambda_k\gtrsim\gamma_k^{(2\alpha-2)/\alpha}$, but not the reverse
inequality. Part (b) establishes full compatibility, but for the invariant sector
only, and under a regularity hypothesis on the radial drift.
If compatibility holds and both sequences grow linearly in the
index of a common subsequence (\Cref{ass:linearcorr}), then $\alpha=2$
(\Cref{cor:alpha2conditional}).
\end{remark}

\begin{theorem}\label{thm:informal3}
Suppose $M^n$ satisfies \Cref{ass:H1,ass:H2,ass:H3,ass:H4} and
\Cref{ass:transitive}, for any $\alpha>0$. Assume that the $G$-invariant
sector $\{u_k\}$ of the eigenbasis satisfies exact radial polynomiality.
Each $u_k$ is a polynomial of degree $k$ in $b=f^{1/\alpha}$, or in $b^2$.
Then $\alpha=2$, so the radial weight is quadratic. After the normalization
in \Cref{thm:inversehermite}, the invariant-sector equation is Hermite
or generalized Laguerre, according to the polynomiality type.
\end{theorem}

Theorem~\ref{thm:informal3} is logically independent of
Theorem~\ref{thm:informal2}. Its proof uses a Bochner-type analysis of the polynomial coefficients
of the radial equation, rather than an analytic non-degeneracy estimate.
The transitive isometric symmetry in \Cref{ass:transitive} is needed
only to make exact polynomiality of the $G$-invariant sector well-posed.
The result concerns the invariant sector alone, uses no curvature
hypothesis on $(M,g)$, and holds for every $\alpha>0$, including
$0<\alpha\le1$. See \S\ref{sec:small}.
Theorem~\ref{thm:informal2}(a), by contrast, relies on confinement
through Schr\"odinger conjugation, which needs $\alpha>1$.
Theorem~\ref{thm:informal2}(b) uses a different radial argument.
The transitive symmetry reduces the invariant sector to a
one-dimensional Sturm--Liouville equation. The result also assumes
\eqref{eq:Dregularity}. Whether that regularity can be removed is
\Cref{prob:oscgrowth}.

\subsection{Scope of the hypotheses}

\ref{ass:H1}--\ref{ass:H3} impose bounds on the weight rather than a soliton equation.
There is no separate curvature, isoperimetric, or Sobolev assumption.
The general estimates apply to complete manifolds without symmetry.
The rigidity results require an invariant radial sector and either
drift regularity or exact polynomiality.
Two examples distinguish these assumptions. Persistent transverse
pinching produces zero-growth separated modes
(\Cref{thm:neckviolatesH4}). A different rotationally symmetric
example satisfies \ref{ass:H1}--\ref{ass:H4} but fails full-spectrum compatibility
(\Cref{prop:incompatibility}).

\subsection{Organization}

\S\S\ref{sec:setup}--\ref{sec:freq} give the hypotheses, preliminary
estimates, and frequency function. \S\ref{sec:conjugate} treats
Schr\"odinger conjugation and Agmon decay.
\S\ref{sec:twofiltrations} defines compatibility and contains the moment
bound and counterexample. \S\ref{sec:localization} studies lower
localization. \S\S\ref{sec:warped}--\ref{sec:sturm} prove the radial
rigidity results. \S\ref{sec:alphatwo} states the consequence of
compatibility with linear growth along a subsequence.
\S\ref{sec:bochner} proves the
inverse Hermite/Laguerre theorem. The last two sections discuss
$0<\alpha\le1$, examples, and further questions.

\section{Geometric Setup}\label{sec:setup}

Throughout, $(M^n,g)$ denotes a complete, connected, noncompact Riemannian
manifold of dimension $n\ge1$, without boundary, with any volume
growth. We fix a basepoint $x_0\in M$ and write $r(x)=d_g(x,x_0)$ for
the Riemannian distance. We write $dV$ for the Riemannian volume
element, $\nabla$ for the Riemannian gradient, and $\Delta=\Delta_g$
for the Laplace--Beltrami operator, with the sign convention
$\Delta u=\mathrm{div}(\nabla u)$.

Let $f\in C^2(M)$ be nonnegative and proper. Define
\[
  b \;:=\; f^{1/\alpha} \;:\; M\to[0,\infty),
\]
where $\alpha>0$ is the growth exponent of Assumption~\ref{ass:H1}
below. For $r>0$, the value $r$ is regular for $b$ exactly when
$r^\alpha$ is regular for $f$. At such a value the level set
\[
  \Sigma_r \;:=\; \{x\in M: b(x)=r\} \;=\; \{f=r^\alpha\}
\]
is a compact $C^2$ hypersurface. We write $d\sigma$ for the induced
measure on $\Sigma_r$, $\nu=\nabla b/|\nabla b|$ for the outward unit
normal, and $\partial_\nu u=\langle\nabla u,\nu\rangle$.

\begin{definition}\label{def:volgrowth}
$(M,g)$ has volume growth controlled by $\omega$ if there exist
$C_{\mathrm{vol}}>0$ and a nondecreasing $\omega:[1,\infty)\to[0,\infty)$ with
$\vol(B_r(x_0))\le C_{\mathrm{vol}} e^{\omega(r)}$ for all $r\ge1$.
\end{definition}

Every complete Riemannian manifold satisfies this for some $\omega$.
Polynomial growth corresponds to $\omega(r)=n\log r$. Sub-exponential
growth corresponds to $\omega(r)=\gamma r^p$ with $0<p<1$.
Exponential growth corresponds to $\omega(r)=\beta r$, as for
$\mathbb H^n$. The only constraint on $\omega$ comes from the weight.
H1--H3 give $\omega(r)=O(r^\alpha)$, \Cref{lem:volmoment}. Thus polynomial growth
is compatible with every $\alpha>0$, and exponential growth with
every $\alpha\ge1$.

\begin{definition}\label{def:weightedlap}
The drift Laplacian with weight $f$ is
$L_fu:=\Delta u-\langle
\nabla f,\nabla u\rangle$.
It is formally symmetric on $L^2(M,e^{-f}dV)$ with respect to
$\langle u,v\rangle_f:=\int_Muve^{-f}dV$.
Its quadratic form is $\mathcal E(u,u)=\int_M|\nabla u|^2e^{-f}dV$.
\end{definition}

\begin{remark}\label{rem:asymp}
We write $A\lesssim B$ if $A\le CB$ with $C$ independent of the
varying eigenfunction index and radii. Unless a more specific dependence
is stated, $C$ may depend on the fixed manifold, weight, and constants in the hypotheses. We write $A\asymp B$ when both inequalities hold.
\end{remark}

\section{Standing Hypotheses}\label{sec:ass}

We formulate four standing hypotheses. They constrain only the
large-scale behavior of $f$, and they leave $f$ unconstrained on the
compact set $\{r\le R_0\}$. H1 permits any $\alpha>0$. Some results
below need $\alpha>1$. These results use Schr\"odinger conjugation, \S\ref{sec:conjugate}--\S\ref{sec:compatibility},
where confinement is needed. Others hold for every $\alpha>0$, among
them the concentration radius of \S\ref{sec:twofiltrations} and the
Inverse Hermite/Laguerre Theorem of \S\ref{sec:bochner}. See
\S\ref{sec:small}.

\begin{assumption}[Asymptotic polynomial bound]\label{ass:H1}
There exist constants $c_1,c_2>0$, $\alpha>0$, and $R_0>0$ such that
\[
  c_1\,r(x)^\alpha \;\le\; f(x) \;\le\; c_2\,r(x)^\alpha \qquad\forall\,r(x)\ge R_0.
\]
\end{assumption}

\begingroup
\renewcommand{\theassumption}{H2}
\renewcommand{\theHassumption}{H2}
\begin{assumption}[Two-sided gradient control]\label{ass:H2}
There exist constants $c_3\ge c_5>0$ such that
\[
  c_5\,f(x)^{2-2/\alpha} \;\le\; |\nabla f(x)|^2 \;\le\; c_3\,f(x)^{2-2/\alpha}
  \qquad\forall\,r(x)\ge R_0.
\]
\end{assumption}
\endgroup
\addtocounter{assumption}{-1}

\begin{remark}\label{rem:gradb}
On $\Sigma_r$ we have $f=r^\alpha$, so for $r\ge R_0$ H2 reads
$c_5r^{2\alpha-2}\le|\nabla f|^2\le c_3r^{2\alpha-2}$. Since $b=f^{1/\alpha}$, the chain
rule gives $|\nabla b|^2=\alpha^{-2}f^{2/\alpha-2}|\nabla f|^2$, so H2
implies
\[
  c_0 \;\le\; |\nabla b| \;\le\; C_0 \qquad\text{on } \{r\ge R_0\},
  \qquad c_0=\sqrt{c_5}/\alpha,\ \ C_0=\sqrt{c_3}/\alpha.
\]
These two-sided bounds on $|\nabla b|$ are used throughout without
further reference.
\end{remark}

\begingroup
\renewcommand{\theassumption}{H3}
\renewcommand{\theHassumption}{H3}
\begin{assumption}[Two-sided Laplacian control]\label{ass:H3}
There exist $c_4>0$ and $0\le\eta<\tfrac12$ such that
\[
  \Delta f(x) \;\ge\; -c_4\,r(x)^{\alpha-1}
  \quad\text{and}\quad
  \Delta f(x) \;\le\; \eta\,|\nabla f(x)|^2, \qquad\forall\,r(x)\ge R_0.
\]
\end{assumption}
\endgroup
\addtocounter{assumption}{-1}

\begin{remark}\label{rem:H3}
The lower bound in H3 provides a lower bound for $\Delta b$
(\Cref{lem:deltab}). The upper bound provides
$V=\tfrac14|\nabla f|^2-\tfrac12\Delta f
\ge(\tfrac14-\tfrac\eta2)|\nabla f|^2$.
Hence $V\to\infty$ when $\alpha>1$.
For a gradient Ricci shrinker, $\Delta f=n/2-R$ motivates the upper
condition. The shrinker identities alone do not give H2 and the lower
condition in H3 in the forms used here.
\end{remark}

\medskip
\noindent\textbf{Convention on $r$ and $R_0$.}
H1--H3 are stated on $\{r(x)\ge R_0\}$, where $r(x)=d_g(x,x_0)$ is
the Riemannian distance of \S\ref{sec:setup}. The level sets
$\Sigma_r=\{b=r\}$, and every radius attached to them, are indexed by
values of $b$. We relate the two parameters by enlarging $R_0$ once. The
closed ball $\{r(x)\le R_0\}$ is compact and $b$ is continuous, so
$B_0:=\max_{\{r\le R_0\}}b$ is finite. Put $R_1:=\max(R_0,B_0,1)+1$.
Every $x$ with $b(x)\ge R_1$ has $r(x)>R_0$, so H1--H3 hold at every
such $x$. We assume this enlargement has been made and write $R_0$
for $R_1$ from here on. Hence $R_0\ge1$, and H1--H3 are available on
$\{b\ge R_0\}$. By H1, $b\asymp r(x)$ there. After changing
$c_4$ by a fixed factor, the lower bound in H3 may be written
$\Delta f\ge-c_4 b^{\alpha-1}$ on the tail. This is valid also when
$0<\alpha<1$. From here on a bare $r$, as in $\Sigma_r$, $r\ge R_0$,
or $V(r)$, always denotes a value of $b$. The Riemannian distance is
written $r(x)$, with its argument, and enters only where H1 compares
$f$ to $r(x)$. This convention also applies in Remark~\ref{rem:gradb}.

\begingroup
\renewcommand{\theassumption}{H4}
\renewcommand{\theHassumption}{H4}
\begin{assumption}[Intrinsic spectral growth]\label{ass:H4}
$-L_f$ has eigenvalues $0=\lambda_0<\lambda_1\le\lambda_2\le\cdots\nearrow
+\infty$ in $L^2(M,e^{-f}dV)$, with $L^2(M,e^{-f}dV)$-normalized
eigenfunctions $u_k$ satisfying $L_fu_k=-\lambda_ku_k$. The family $\{u_k\}_{k\ge0}$ is a complete
orthonormal system in $L^2(M,e^{-f}dV)$. The eigenvalue $\lambda_0=0$, with constant
eigenfunction $u_0$, is always present, since
$\int_Me^{-f}dV<\infty$, \Cref{lem:volmoment}. H4 requires the
remaining spectrum to be positive and discrete. For each $k$, define
\begin{equation}\label{eq:gammak}
  \gamma_k:=\inf\Bigl\{N\ge0:
  \frac1{\vol(\Sigma_r)}\int_{\Sigma_r}u_k^2\,d\sigma
  \le C_{k,N}r^{2N}\text{ for some }C_{k,N}>0\text{ and all }r\ge R_0\Bigr\}.
\end{equation}
We require $\gamma_k<+\infty$ for every $k$ and
$\gamma_k\to+\infty$ as $k\to\infty$.
\end{assumption}
\endgroup
\addtocounter{assumption}{-1}

\begin{remark}\label{rem:H4delicate}
The constants in \eqref{eq:gammak} may depend on the eigenfunction
and on the exponent $N$. Here $\gamma_k$ measures the asymptotic polynomial degree, not the
size at a fixed radius. This distinction is
needed even in elementary radial examples. In the
Ornstein--Uhlenbeck model, $H_k(r)^2\sim r^{2k}$, but
$\int H_k^2e^{-r^2/4}<\infty$. If
$\psi_k:=u_ke^{-f/2}$, then the growth of $u_k$ reflects the balance
between the decay of $\psi_k$ and the factor $e^{f/2}$.
Theorem~\ref{thm:neckviolatesH4} shows that, even when H1--H3 hold,
the natural separated eigenbasis can contain zero-growth modes at
unbounded energies. H4 also restricts the geometry seen by the
chosen eigenbasis.
\end{remark}

\begin{remark}\label{rem:completeness}
For $\alpha>1$, completeness of $\{u_k\}$ is not an independent
requirement. By \Cref{lem:discrete}, H1--H3 and Schr\"odinger conjugation give a
purely discrete spectrum for $L_f$. Completeness then follows from
the spectral theorem. For $0<\alpha\le1$
the argument of \Cref{lem:discrete} does not apply,
\Cref{rem:whyalphagt1}. In that range completeness is part of the
content of H4, and does not follow from H1--H3.
\end{remark}

\begin{remark}
H1--H4 impose no separate bound on the tangential spectrum of the
level sets. This distinction matters in \Cref{prop:incompatibility}.
\end{remark}

\section{Preliminary Lemmas}\label{sec:prelim}

\begin{lemma}\label{lem:properness}
Suppose $(M,g)$ is complete and $f$ satisfies H1. Then $b=f^{1/\alpha}$
is proper on $M$. Every sublevel set $\{b\le R\}$ is compact,
and every level set $\Sigma_r=b^{-1}(r)$ is compact.
\end{lemma}

\begin{proof}
By H1, $f(x)\ge c_1r(x)^\alpha$ for $r(x)\ge R_0$, where $r(x)$ is
the Riemannian distance from the base point $x_0$ used in H1. We have
$b(x)=f(x)^{1/\alpha}\ge c_1^{1/\alpha}r(x)$ for $r(x)\ge R_0$. Hence
\[
  \{b\le R\} \;\subset\; \{r(x)\le R_0\} \;\cup\; \{r(x)\le R/c_1^{1/\alpha}\}
  \;\subset\; \{r(x)\le \max(R_0,\,R/c_1^{1/\alpha})\},
\]
a closed metric ball of finite radius in $(M,g)$. Since $(M,g)$ is
complete, closed and bounded subsets of $M$ are compact, by
Hopf--Rinow. Hence $\{r(x)\le\max(R_0,R/c_1^{1/\alpha})\}$ is compact.
As $b$ is continuous, $\{b\le R\}$ is a closed subset of this compact
set, so it is compact.

For every $r$, $\Sigma_r=b^{-1}(r)$ is closed in the compact set
$\{b\le r\}$, so it is compact.
\end{proof}

\begin{corollary}\label{cor:weightfinite}
Under H1--H3, for every $r\ge R_0$ at which $\Sigma_r$ is a regular
level set, $0<\vol(\Sigma_r)<\infty$. On this interval
$(R_0,\infty)$, the induced measure
\[
  d\nu(r) \;=\; e^{-f(r)}\left(\int_{\Sigma_r}\frac{d\mathcal H^{n-1}}{|\nabla b|}\right)dr
\]
is a positive measure on $(R_0,\infty)$, and it is locally bounded by
the two-sided bound $c_0\le|\nabla b|\le C_0$ of \Cref{rem:gradb}.
\end{corollary}

\begin{proof}
By Lemma~\ref{lem:properness}, $\Sigma_r$ is a compact hypersurface
for regular $r$. A compact smooth submanifold of dimension $n-1$ has
finite and positive $(n-1)$-dimensional Hausdorff measure, so
$0<\vol(\Sigma_r)<\infty$. By H1, $f$ is continuous and finite. By
\Cref{rem:gradb}, $|\nabla b|\in[c_0,C_0]$ on $\{r\ge R_0\}$. Hence the
density $e^{-f(r)}\int_{\Sigma_r}d\mathcal H^{n-1}/|\nabla b|$ is
finite and positive at every regular $r\ge R_0$. This makes $d\nu$ a
positive measure. It is bounded on compact subintervals of
$(R_0,\infty)$. Indeed, the $C^1$ flow of
$\nabla b/|\nabla b|^2$ identifies nearby level sets, and the corresponding
Jacobian and $|\nabla b|^{-1}$ vary continuously on each compact tail
interval.
\end{proof}

\begin{remark}
Under a transitive isometric symmetry (\Cref{ass:transitive},
\S\ref{sec:sturm}), $|\nabla b|$ is constant on each $\Sigma_r$
(\Cref{ass:radial}). In this case
$d\nu(r)=e^{-f(r)}\vol(\Sigma_r)/|\nabla b|\,dr$,
the form used in that setting.
Without such symmetry, $|\nabla b|$ satisfies only the two-sided
bound of H2, and the integral form above is the general statement.
\end{remark}

\begin{remark}
Properness makes the level sets compact. H2 makes all sufficiently
large levels regular.
\end{remark}

\begin{lemma}\label{lem:coarea}
On the tail $\{b\ge R_0\}$, the function $b=f^{1/\alpha}$ is $C^2$ and
has no critical points. Every $r\ge R_0$ is a regular value of $b$ on the tail,
and $\Sigma_r$ is a compact $C^2$ hypersurface. For every measurable
$\phi\ge0$ and every $R_0\le a<c<\infty$,
\[
  \int_{\{a<b<c\}} \phi\,dV
  = \int_a^c\int_{\Sigma_s}\frac{\phi}{|\nabla b|}\,d\sigma\,ds.
\]
In particular, for $r\ge R_0$,
\[
  \int_{\{R_0<b<r\}} \phi\,dV
  = \int_{R_0}^r\int_{\Sigma_s}\frac{\phi}{|\nabla b|}\,d\sigma\,ds.
\]
\end{lemma}

\begin{proof}
On $\{b\ge R_0\}$ one has $f=b^\alpha\ge R_0^\alpha>0$, so
$b=f^{1/\alpha}$ is $C^2$ there. By \Cref{rem:gradb},
$|\nabla b|\ge c_0>0$ on the tail, so there are no critical points.
The level sets are smooth and compact by
\Cref{lem:properness}. The stated identities are the coarea formula
applied to the tail region. See~\cite{Chavel2006}.
\end{proof}

\begin{lemma}\label{lem:deltab}
Under H1--H3, there exists $C_b=C_b(c_3,c_4,\alpha)>0$ such that
\[
  \Delta b \ge -C_b
  \quad\text{on }\Sigma_r,\text{ for all }r\ge R_0.
\]
\end{lemma}
\begin{proof}
Since $b=f^{1/\alpha}$, the chain rule gives
\[
  \Delta b
  = \frac{1-\alpha}{\alpha^2}\,f^{1/\alpha-2}|\nabla f|^2
  + \frac{1}{\alpha}\,f^{1/\alpha-1}\Delta f.
\]
On $\Sigma_r$ we have $f=r^\alpha$, $f^{1/\alpha-2}=r^{1-2\alpha}$,
and $f^{1/\alpha-1}=r^{1-\alpha}$.

For the first term, H2 on $\Sigma_r$ implies
$|\nabla f|^2\le c_3 r^{2\alpha-2}$, so
\[
  \frac{1-\alpha}{\alpha^2}\,r^{1-2\alpha}\cdot c_3\,r^{2\alpha-2}
  = -\frac{(\alpha-1)c_3}{\alpha^2 r}.
\]
For the second term, the lower bound in H3 gives
$\Delta f\ge -c_4 r^{\alpha-1}$, so
\[
  \frac{1}{\alpha}\,r^{1-\alpha}\cdot(-c_4 r^{\alpha-1})
  = -\frac{c_4}{\alpha}.
\]
If $0<\alpha\le1$, then $1-\alpha\ge0$ and the first term is
non-negative, so the two terms together give
$\Lap b\ge -\dfrac{c_4}{\alpha}$ on $\Sigma_r$ for all $r\ge R_0$.

If $\alpha>1$, the two terms give, for all $r\ge R_0$,
\[
  \Lap b \;\ge\; -\frac{(\alpha-1)c_3}{\alpha^2\,r} - \frac{c_4}{\alpha}
  \;\ge\; -\frac{(\alpha-1)c_3}{\alpha^2} - \frac{c_4}{\alpha},
\]
since $r\ge R_0\ge1$.

In both cases $\Lap b\ge -C_b$ on $\Sigma_r$ for all $r\ge R_0$, with
\[
  C_b := \frac{(\alpha-1)_+\,c_3}{\alpha^2} + \frac{c_4}{\alpha},
  \qquad (\alpha-1)_+ := \max(\alpha-1,0).
\]
\end{proof}

\begin{lemma}\label{lem:deltabUpper}
Under H1--H3, for all $r\ge R_0$:
\[
  \Lap b \;\le\;
  \begin{cases}
  \dfrac{(1-\alpha)c_3}{\alpha^2 R_0} + \dfrac{\eta c_3}{\alpha}\,R_0^{\alpha-1},
    & 0<\alpha\le1,\\[3mm]
  \dfrac{\eta c_3}{\alpha}\, r^{\alpha-1},
    & \alpha>1.
  \end{cases}
\]
In particular, $\Lap b$ is bounded above by an $r$-independent constant
$C_b^+=C_b^+(c_3,\eta,\alpha,R_0)$ for $0<\alpha\le1$.
For $\alpha>1$, no such constant bound holds in general under H1--H3 alone. The sharp
upper bound grows like $r^{\alpha-1}$.
\end{lemma}
\begin{proof}
On $\Sigma_r$, $f=r^\alpha$, so as in Lemma~\ref{lem:deltab},
\[
  \Lap b = \frac{1-\alpha}{\alpha^2}r^{1-2\alpha}|\nabla f|^2
  + \frac1\alpha r^{1-\alpha}\Lap f.
\]
For the second term, H3 gives $\Lap f\le\eta|\nabla f|^2$ and H2
gives $|\nabla f|^2\le c_3 r^{2\alpha-2}$. We obtain
$\frac1\alpha r^{1-\alpha}\Lap f \le \frac{\eta c_3}{\alpha}r^{\alpha-1}$.

As for the first term, if $0<\alpha\le1$ then the coefficient
$(1-\alpha)/\alpha^2$ is non-negative. Maximizing over $|\nabla f|^2$
with the upper bound in H2 implies $\frac{1-\alpha}{\alpha^2}r^{1-2\alpha}|\nabla f|^2\le
\frac{(1-\alpha)c_3}{\alpha^2 r}$, which is non-negative and
non-increasing in $r$, so it is bounded by its value at $r=R_0$. Likewise $r^{\alpha-1}$ is non-increasing for $\alpha\le1$, so
the second term is bounded by its value at $r=R_0$. Combining the two estimates gives the
stated bound for $0<\alpha\le1$.

If $\alpha>1$ the coefficient $(1-\alpha)/\alpha^2$ is negative.
Maximizing over $|\nabla f|^2$ with the lower bound in H2 implies
$\frac{1-\alpha}{\alpha^2}r^{1-2\alpha}|\nabla f|^2 \le
-\frac{(\alpha-1)c_5}{\alpha^2 r}\le0$. Dropping this non-positive term
and keeping only the second term's bound $\frac{\eta c_3}{\alpha}r^{\alpha-1}$
proves the stated bound for $\alpha>1$.
\end{proof}

The upper bound in H3 leads to the following estimate for the volume of
the level sets. The resulting bound is the only place the geometry of
$(M,g)$ enters the moment estimates of \S\ref{sec:twofiltrations}.

\begin{lemma}\label{lem:volmoment}
Under \crefrange{ass:H1}{ass:H3}, for any $\alpha>0$, there is a
constant $C_\Sigma$ such that, for $r\ge R_0$,
\begin{equation}\label{eq:sigmagrowth}
  \log\vol(\Sigma_r) \;\le\; \eta\,r^\alpha+(1-\alpha)\log r+C_\Sigma .
\end{equation}
There is a constant $C_M>0$, independent of $q$, such that
\begin{equation}\label{eq:volmoment}
  \int_{R_0}^\infty r^q\,e^{-f(r)}\vol(\Sigma_r)\,dr \;\le\; C_M^{\,q+1}\,\Gamma\Bigl(\frac{q+1}{\alpha}\Bigr)
  \qquad\text{for every }q\ge0,
\end{equation}
where $f(r)=r^\alpha$ is the constant value of $f$ on $\Sigma_r$. In
particular $\int_Me^{-f}\,dV<\infty$. Through \Cref{ass:H1}, one may
take
\[
  \omega(R)=\eta c_2R^\alpha+O(\log R)
\]
in \Cref{def:volgrowth}.
\end{lemma}
\begin{proof}
Write
\[
 A(r):=\int_{\Sigma_r}|\nabla b|\,d\sigma,
 \qquad
 G_+(r):=\int_{R_0}^r s^{2\alpha-2}A(s)\,ds.
\]
Since $f=b^\alpha$, on $\Sigma_r$ we have
$\partial_\nu f=\alpha r^{\alpha-1}|\nabla b|$. Green's formula on
$B_r:=\{b\le r\}$ gives
\[
 \alpha r^{\alpha-1}A(r)=\int_{B_r}\Delta f\,dV.
\]
Separate the compact core $B_{R_0}$ from the tail. By H3 and the tail
coarea formula,
\begin{align*}
 \alpha r^{\alpha-1}A(r)
 &\le \int_{B_{R_0}}\Delta f\,dV
      +\eta\int_{\{R_0<b<r\}}|\nabla f|^2\,dV\\
 &= \int_{B_{R_0}}\Delta f\,dV+\eta\alpha^2G_+(r).
\end{align*}
Let $C_+:=\max\{0,\int_{B_{R_0}}\Delta f\,dV\}$. If $\eta=0$, the
boundary estimate directly gives
\[
  \alpha r^{\alpha-1}A(r)\le C_+,
\]
so $A(r)\le C r^{1-\alpha}$. Suppose now that $\eta>0$. Since
$G_+'(r)=r^{2\alpha-2}A(r)$, for a.e. $r\ge R_0$,
\[
 G_+'(r)\le \eta\alpha r^{\alpha-1}G_+(r)
       +\frac{C_+}{\alpha}r^{\alpha-1}.
\]
Gr\"onwall's inequality gives $G_+(r)\le C_2e^{\eta r^\alpha}$.
Substitution in the boundary estimate then gives
\[
 A(r)\le C_3 r^{1-\alpha}e^{\eta r^\alpha}.
\]
The same bound holds for every $0\le\eta<\tfrac12$. Since
$A(r)\ge c_0\vol(\Sigma_r)$ by \Cref{rem:gradb}, this proves
\eqref{eq:sigmagrowth}.

For \eqref{eq:volmoment}, $f=r^\alpha$ on $\Sigma_r$, so
\[
 e^{-f(r)}\vol(\Sigma_r)
 \le C_4 r^{1-\alpha}e^{-(1-\eta)r^\alpha}.
\]
Because the integral starts at the fixed positive number $R_0$, this
implies, after enlarging the constant if necessary,
\[
 \int_{R_0}^\infty r^q e^{-f(r)}\vol(\Sigma_r)\,dr
 \le C_M^{q+1}\Gamma\!\left(\frac{q+1}{\alpha}\right),
 \qquad q\ge0.
\]
For all sufficiently large $q$, comparison with the full Gamma
integral gives the same estimate with a fixed shift in the Gamma
argument. The ratio of the shifted and unshifted Gamma functions grows
only polynomially in $q$, and the factor $(1-\eta)^{-q/\alpha}$ is
absorbed by $C_M^{q+1}$. The remaining bounded interval of $q$ is
covered by enlarging $C_M$.

The tail coarea formula and $|\nabla b|\ge c_0$ now give
\[
 \int_{\{b>R_0\}}e^{-f}\,dV
 \le \frac1{c_0}\int_{R_0}^\infty e^{-f(r)}\vol(\Sigma_r)\,dr<\infty.
\]
The integral over $B_{R_0}$ is finite by compactness, so
$\int_Me^{-f}\,dV<\infty$.

Finally, H1 implies $b(x)\le c_2^{1/\alpha}r(x)$ outside a compact set.
For large $R$, the Riemannian ball $B_R(x_0)$ is contained in
$B_{R_0}\cup\{R_0<b\le c_2^{1/\alpha}R\}$. Using the tail coarea
formula, $|\nabla b|\ge c_0$, and \eqref{eq:sigmagrowth},
\[
 \vol B_R(x_0)
 \le C+\frac1{c_0}\int_{R_0}^{c_2^{1/\alpha}R}\vol(\Sigma_s)\,ds
 \le C'R^2e^{\eta c_2R^\alpha}.
\]
This proves $\log\vol B_R(x_0)\le\eta c_2R^\alpha+O(\log R)$.
\end{proof}

The coefficient $\eta$ in \eqref{eq:sigmagrowth} cannot be decreased uniformly.

\begin{example}\label{ex:volgrowthsharp}
Fix $n\ge2$, $\alpha>1$, $\eta\in(0,\tfrac12)$, and
$\eta'\in(0,\eta)$. On $M=\mathbb R^n$ take a smooth rotationally
symmetric metric $g=dr^2+\varphi(r)^2g_{\mathbb S^{n-1}}$ with
$\varphi(r)=r$ near the origin and
\[
  \varphi(r):=\exp\Bigl(\frac{\eta'}{n-1}\,r^\alpha\Bigr)\qquad(r\ge R_0),
\]
with a smooth positive interpolation in between. Choose a smooth
nonnegative radial weight $f$ with $f(r)=r^\alpha$ for all sufficiently
large $r$. Then $b=r$ on the tail, so H1 holds with $c_1=c_2=1$.
Also $|\nabla f|^2=\alpha^2r^{2\alpha-2}$ on the tail,
so H2 holds with $c_3=c_5=\alpha^2$. And H3 holds with the given
$\eta$ and $c_4=1$ once $R_0$ is large. Yet $\log\vol(\Sigma_r)=\eta'r^\alpha+\log\vol(\mathbb S^{n-1})$.
\end{example}
\begin{proof}
On the tail, $\Delta f=f''+(n-1)\frac{\varphi'}{\varphi}f'
=\alpha(\alpha-1)r^{\alpha-2}+\eta'\alpha^2r^{2\alpha-2}
=\eta'|\nabla f|^2+\alpha(\alpha-1)r^{\alpha-2}$. The last term is nonnegative
and of lower order, so $-c_4r^{\alpha-1}\le\Delta f\le\eta|\nabla f|^2$ as
soon as $r^\alpha\ge\frac{\alpha-1}{\alpha(\eta-\eta')}$. Since
$\vol(\Sigma_r)=\varphi(r)^{n-1}\vol(\mathbb S^{n-1})$, the growth claim follows.
Since $\eta'$ can be chosen arbitrarily close to $\eta$, the coefficient
$\eta$ in \eqref{eq:sigmagrowth} cannot be decreased uniformly.
\end{proof}

\section{The Frequency Function}\label{sec:freq}

Fix $k$ throughout, and write $u=u_k$, $\lambda=\lambda_k$, $\gamma=\gamma_k$.

\begin{definition}\label{def:freq}
For $r>R_0$ set
\begin{align}
  I(r) &:= \frac{1}{\vol(\Sigma_r)}\int_{\Sigma_r}u^2|\nabla b|\,d\sigma,
  \label{eq:I}\\
  D(r) &:= \frac{1}{\vol(\Sigma_r)}\int_{\Sigma_r}u\,\partial_\nu u\,d\sigma.
  \label{eq:D}
\end{align}
Whenever $I(r)>0$, define
\begin{equation}\label{eq:U}
  U(r):=\frac{rD(r)}{I(r)}.
\end{equation}
\end{definition}

By H2, the vector field $X:=\nabla b/|\nabla b|^2$ is $C^1$ and
nonvanishing on the tail. Its local flow identifies neighboring level
sets of $b$. The first-variation formula for integrals over these level
sets gives local absolute continuity of the quantities in $I$ and $D$.
At regular points, it also gives the derivative formulas used below. Both $I$ and $D$ are
differentiable for a.e. $r>R_0$, and so is $U$ on every interval on
which $I>0$. Possible zero levels of $I$ are treated separately in
\Cref{lem:Ivanish}.

\begin{definition}\label{def:FEPhi}
For $u=u_k$ an $L_f$-eigenfunction of eigenvalue $\lambda_k$, and
$r>R_0$:
\begin{align}
  F(r) &:= \int_{\{b<r\}}u^2\,e^{-f}\,dV, \label{eq:Fdef}\\
  E(r) &:= \int_{\{b<r\}}|\nabla u|^2\,e^{-f}\,dV, \label{eq:Edef}\\
  \Phi(r) &:= \int_{\Sigma_r}u\,\partial_\nu u\,e^{-f}\,d\sigma. \label{eq:Phidef}
\end{align}
By the weighted divergence theorem, $\Phi(r)=e^{-f(r)}\vol(\Sigma_r)D(r)$.
Testing $L_fu_k=-\lambda_ku_k$ against $u_k$ over all of $M$ gives
$F(\infty)=1$ and $E(\infty)=\lambda_k$. The relation of $\Phi$ to the mass density $m(r)$ introduced in
\S\ref{sec:twofiltrations} is given in Lemma~\ref{lem:PhiUm}.
\end{definition}

\begin{remark}\label{rem:CM}\label{rem:CMcomparison}
Colding and Minicozzi \cite[(0.5)--(0.6), Theorem~0.7]{ColdingMinicozzi2021}
use $b_{\rm CM}=2\sqrt f$ and normalize the boundary mass by $r^{1-n}$.
Under their hypotheses, including properness of $f$ and the identities
$\Delta f+S=n/2$ and $|\nabla f|^2+S=f$ with $f,S\ge0$,
the frequency of a nonzero $L^2(e^{-f}dV)$ eigenfunction with
eigenvalue $\lambda>0$ satisfies
\[
  U_{\rm CM}(r)\le2\lambda
  \left(1+\frac{4\lambda+2n-4+\epsilon}{r^2}\right)
\]
for every $\epsilon>0$ and all sufficiently large $r$.
Our quantities \eqref{eq:I} and \eqref{eq:D} use the actual
$\vol(\Sigma_r)$ instead. For a fixed exhaustion, this common volume
normalization cancels in $U=rD/I$. The identity for
$J=\vol(\Sigma_r)I$ in \Cref{cor:Jidentity} avoids differentiating
the level-set volume. The argument uses no shrinker identity.
\end{remark}

\begin{lemma}\label{lem:gammaequiv}
Under H1--H4,
\begin{equation}\label{eq:gammaequiv}
  \gamma_k
  = \inf\Bigl\{N\ge0 : I_k(r)\le C_{k,N}r^{2N}
    \text{ for some }C_{k,N}>0\text{ and all }r\ge R_0\Bigr\}.
\end{equation}
Moreover, with $\log0:=-\infty$,
\begin{equation}\label{eq:gammalimsup}
  \gamma_k=\max\left\{0,\limsup_{r\to\infty}
  \frac{\log I_k(r)}{2\log r}\right\}.
\end{equation}
\end{lemma}
\begin{proof}
By \Cref{rem:gradb},
\[
  \frac{c_0}{\vol(\Sigma_r)}\int_{\Sigma_r}u_k^2\,d\sigma
  \le I_k(r)\le
  \frac{C_0}{\vol(\Sigma_r)}\int_{\Sigma_r}u_k^2\,d\sigma .
\]
Multiplying a growth bound by the fixed constants $c_0$ or $C_0$
does not change its infimal exponent. This proves
\eqref{eq:gammaequiv}.

Let $L_k$ denote the limsup in \eqref{eq:gammalimsup}. If $N>L_k$,
then $I_k(r)\le r^{2N}$ for all sufficiently large $r$. On the
remaining compact interval $[R_0,R]$, continuity gives
$I_k(r)\le C_{k,N}r^{2N}$ after increasing $C_{k,N}$. Hence
$\gamma_k\le\max\{0,L_k\}$. Conversely, if
$I_k(r)\le C_{k,N}r^{2N}$ on the tail, then
$L_k\le N$. Taking the infimum over admissible $N\ge0$ gives the
reverse inequality.
\end{proof}

\begin{lemma}\label{lem:Ivanish}
Under H1--H2, if $I(r_0)=0$ for some regular $r_0>R_0$, then
$u\equiv0$ on $\Sigma_{r_0}$.
\end{lemma}
\begin{proof}
Since $\vol(\Sigma_{r_0})>0$ and $|\nabla b|\ge c_0>0$ on the tail,
$I(r_0)=0$ implies
\[
  0=\int_{\Sigma_{r_0}}u^2|\nabla b|\,d\sigma
  \ge c_0\int_{\Sigma_{r_0}}u^2\,d\sigma.
\]
Hence $u=0$ almost everywhere on $\Sigma_{r_0}$, and continuity gives $u\equiv0$ on $\Sigma_{r_0}$.
\end{proof}

\section{Schrödinger Conjugation and Agmon Decay}\label{sec:conjugate}

\subsection{Conjugation and discrete spectrum}

\begin{lemma}\label{lem:conjugation}
Define the unitary map $\Phi:L^2(M,e^{-f}\dV)\to L^2(M,\dV)$ by
$\Phi(u)=e^{-f/2}u$. Setting $\psi=\Phi(u)=e^{-f/2}u$, the equation
$L_fu=-\lambda u$ is equivalent to $H\psi=\lambda\psi$, where
\begin{equation}\label{eq:Hdef}
  H \;=\; -\Delta+V, \qquad V \;=\; \tfrac14|\nabla f|^2-\tfrac12\Delta f.
\end{equation}
$-L_f$ and $H$ are unitarily equivalent and have
identical spectra.
\end{lemma}
\begin{proof}
Using $\nabla u=e^{f/2}(\nabla\psi+\tfrac12\psi\nabla f)$,
\[
  \Delta u = e^{f/2}\Bigl[\Delta\psi+\langle\nabla f,\nabla\psi\rangle
  +\tfrac12(\Delta f)\psi+\tfrac14|\nabla f|^2\psi\Bigr],
\]
so
\[
  L_fu = \Delta u-\langle\nabla f,\nabla u\rangle
  = e^{f/2}\Bigl[\Delta\psi+\tfrac12(\Delta f)\psi-\tfrac14|\nabla f|^2\psi\Bigr]
  = e^{f/2}\bigl(\Delta\psi-V\psi\bigr).
\]
This is equivalent to $L_fu=-\lambda u\iff\Delta\psi-V\psi=-\lambda\psi\iff H\psi=\lambda\psi$.
\end{proof}

\begin{lemma}\label{lem:Vbounds2}
Under \crefrange{ass:H1}{ass:H3} with $\alpha>1$, set
$\kappa^2:=\tfrac14-\tfrac\eta2>0$ (\Cref{ass:H3}).
\begin{enumerate}[label=\rm(\alph*)]
\item In terms of $b$, with
  $c_B^2:=\kappa^2c_5$ and
  $C_B^2:=\tfrac14c_3+\tfrac12c_4c_1^{-(\alpha-1)/\alpha}$,
  \begin{equation}\label{eq:Vboundsb}
    c_B^2\,b(x)^{2\alpha-2} \;\le\; V(x) \;\le\; C_B^2\,b(x)^{2\alpha-2}
    \qquad\text{on }\{b\ge R_0\}.
  \end{equation}
\item In terms of Riemannian distance, with
  $c_V^2:=\kappa^2c_5c_1^{2-2/\alpha}$ and
  $C_V^2:=\tfrac14c_3c_2^{2-2/\alpha}+\tfrac12c_4$,
  \begin{equation}\label{eq:Vbounds2}
    c_V^2\,r(x)^{2\alpha-2} \;\le\; V(x) \;\le\; C_V^2\,r(x)^{2\alpha-2}
    \qquad\text{on }\{r(x)\ge R_0\}.
  \end{equation}
\end{enumerate}
All four constants depend only on $c_1,c_2,c_3,c_5,c_4,\eta,\alpha$. In
particular $V(x)\to+\infty$ as $b(x)\to\infty$.
\end{lemma}
The two versions differ through H1's constants $c_1,c_2$.
These constants compare $f$ with $r(x)$. The two versions coincide
only when $c_1=c_2=1$. The level sets and the radii in this paper are
indexed by $b$. Part (a) is the form used below, except where a geodesic
path is parametrized by Riemannian arc length.
\begin{proof}
By H3, $\Delta f\le\eta|\nabla f|^2$, so
$V=\tfrac14|\nabla f|^2-\tfrac12\Delta f\ge\kappa^2|\nabla f|^2$.

For (a), by H2, $|\nabla f|^2\ge c_5f^{2-2/\alpha}$, and $f=b^\alpha$
identically, so $f^{2-2/\alpha}=b^{2\alpha-2}$. This gives the
lower bound with no use of H1. For the upper bound, H2 implies
$|\nabla f|^2\le c_3f^{2-2/\alpha}=c_3b^{2\alpha-2}$, and H3 gives
$-\Delta f\le c_4r(x)^{\alpha-1}$. By H1's lower bound,
$c_1r(x)^\alpha\le f=b^\alpha$, so $r(x)\le c_1^{-1/\alpha}b$ and
$r(x)^{\alpha-1}\le c_1^{-(\alpha-1)/\alpha}b^{\alpha-1}$, since
$\alpha-1>0$. Since $b\ge R_0\ge1$, $b^{\alpha-1}\le b^{2\alpha-2}$, so
$V\le\tfrac14c_3b^{2\alpha-2}+\tfrac12c_4c_1^{-(\alpha-1)/\alpha}b^{2\alpha-2}
=C_B^2b^{2\alpha-2}$.

For (b), by H1, $f\ge c_1r(x)^\alpha$, so (exponent $2-2/\alpha>0$ for
$\alpha>1$) $f^{2-2/\alpha}\ge c_1^{2-2/\alpha}r(x)^{2\alpha-2}$, and
combining with $V\ge\kappa^2c_5f^{2-2/\alpha}$ gives the lower bound.
In the other direction, $f\le c_2r(x)^\alpha$ implies
$|\nabla f|^2\le c_3c_2^{2-2/\alpha}r(x)^{2\alpha-2}$, and
$-\Delta f\le c_4r(x)^{\alpha-1}\le c_4r(x)^{2\alpha-2}$ for
$r(x)\ge R_0\ge1$, so $V\le C_V^2r(x)^{2\alpha-2}$.
\end{proof}

\begin{lemma}\label{lem:discrete}
Under H1--H3 with $\alpha>1$, $L_f$ has purely discrete spectrum in
$L^2(M,e^{-f}\dV)$.
\end{lemma}
\begin{proof}
By Lemma~\ref{lem:conjugation}, it suffices to show $H=-\Delta+V$ has
purely discrete spectrum in $L^2(M,\dV)$. By
Lemma~\ref{lem:Vbounds2}, $V\to+\infty$. Since $f\in C^2(M)$, $V$ is
continuous, and is locally bounded below. Compactness of the quadratic-form embedding then proves the claim.
After adding a constant to $V$, the form norm controls $H^1$ on each
compact domain. Local Rellich compactness gives compactness there.
Since $V\to\infty$, a form-bounded family has uniformly small
$L^2$ mass outside a sufficiently large compact set. The global
embedding is compact. Unitary equivalence gives the same
conclusion for $L_f$.
\end{proof}

\subsection{\texorpdfstring{Agmon decay on $(M,g)$}{Agmon decay on (M,g)}}

\begin{definition}\label{def:agmon}
For eigenvalue $\lambda_k$, the classically allowed region is
\[
  \mathcal{A}_k := \{x\in M : V(x)\le\lambda_k\},
\]
which is compact by
Lemma~\ref{lem:Vbounds2}.

The Agmon metric associated with $\lambda_k$ is the (possibly degenerate)
Riemannian metric
\[
  ds^2_{\mathcal{A}} = (V(x)-\lambda_k)_+\,ds^2_g,
\]
where $(V-\lambda_k)_+=\max(V-\lambda_k,0)$.

The Agmon distance from $\mathcal{A}_k$ to $x$ is
\begin{equation}\label{eq:rho}
  \rho_{\lambda_k}(x)
  := \inf_\gamma \int_0^1 \sqrt{(V(\gamma(t))-\lambda_k)_+}\;|\dot\gamma(t)|_g\,dt,
\end{equation}
where the infimum is over piecewise $C^1$ curves $\gamma:[0,1]\to M$ with
$\gamma(0)\in\mathcal{A}_k$ and $\gamma(1)=x$, and $|\dot\gamma|_g$ denotes
the length of $\dot\gamma$ in the metric $g$.

It is standard that $\rho_{\lambda_k}$ is a $1$-Lipschitz function with respect to
$ds^2_{\mathcal{A}}$ and satisfies the eikonal inequality
$|\nabla_g\rho_{\lambda_k}|\le\sqrt{(V-\lambda_k)_+}$ a.e.~on $M$.
\end{definition}

\begin{proposition}\label{prop:agmon}
Under H1--H3 with $\alpha>1$, let $\psi_k=e^{-f/2}u_k\in L^2(M,\dV)$
satisfy $H\psi_k=\lambda_k\psi_k$. Then for every $\varepsilon\in(0,1)$,
\begin{equation}\label{eq:agmonL2}
  \int_M e^{2(1-\varepsilon)\rho_{\lambda_k}(x)}\bigl(\psi_k(x)^2+|\nabla\psi_k(x)|^2\bigr)\,\dV \;<\;\infty.
\end{equation}
\end{proposition}
\begin{proof}
By \Cref{lem:Vbounds2}, $V\to+\infty$. Since $V$ is continuous, it
is bounded below on $M$. Completeness implies essential self-adjointness
of $H$ \cite{BravermanMilatovicShubin2002}. We use the weighted energy
argument of \cite{Agmon1982,CyconFroeseKirschSimon1987,HislopSigal1996}.
Write $\psi=\psi_k$, $\lambda=\lambda_k$, and
$d=\rho_{\lambda_k}$. An $L^2$ eigenfunction belongs to the quadratic
form domain, so
\[
 \int_M\bigl(|\nabla\psi|^2+V_+\psi^2\bigr)\,dV<\infty,
 \qquad V_+:=\max(V,0).
\]
The allowed set $\mathcal A_k$ is nonempty. Otherwise
$\int_M(|\nabla\psi|^2+(V-\lambda)\psi^2)\,dV>0$, contrary to the
eigenfunction equation.

Set $\theta=1-\varepsilon$ and
$\eta_T=\theta\min(d,T)$, $\varphi_T=e^{\eta_T}$. The distance $d$
is locally Lipschitz and satisfies the eikonal inequality a.e.
The function $\varphi_T$ is bounded and
\[
 |\nabla\varphi_T|^2\le\theta^2(V-\lambda)_+\varphi_T^2
 \quad\text{a.e.}
\]
These bounds show that $\varphi_T\psi$ and $\varphi_T^2\psi$ belong
to the form domain. Testing the weak eigenfunction equation with
$\varphi_T^2\psi$ and expanding gives
\[
 \int_M|\nabla(\varphi_T\psi)|^2\,dV
 +\int_M(V-\lambda-|\nabla\eta_T|^2)\varphi_T^2\psi^2\,dV=0.
\]
On $\mathcal A_k$, $d=0$ and $\varphi_T=1$, so
\[
 \int_M|\nabla(\varphi_T\psi)|^2\,dV
 +(1-\theta^2)\int_M(V-\lambda)_+\varphi_T^2\psi^2\,dV
 \le C_k,
 \qquad C_k:=\int_{\mathcal A_k}(\lambda-V)\psi^2\,dV.
\]
Here $C_k<\infty$ is independent of $T$.
The set $K_k:=\{V\le\lambda+1\}$ is compact, and $d$ is bounded
there. On its complement $V-\lambda>1$. The preceding estimate bounds $\int_M\varphi_T^2\psi^2\,dV$
uniformly in $T$.
Moreover,
\[
 \varphi_T^2|\nabla\psi|^2
 \le2|\nabla(\varphi_T\psi)|^2
     +2|\nabla\eta_T|^2\varphi_T^2\psi^2.
\]
The same estimate bounds its integral uniformly in $T$.
Letting $T\to\infty$ by monotone convergence proves \eqref{eq:agmonL2}.
\end{proof}

\begin{lemma}\label{lem:agmonlb}
Under H1--H3 with $\alpha>1$, let $c_B$ be the constant of
\Cref{lem:Vbounds2}(a), so that $V(x)\ge c_B^2b(x)^{2\alpha-2}$ on
$\{b\ge R_0\}$ by \eqref{eq:Vboundsb}, and set
$r_{\mathrm{Ag}}=r_{\mathrm{Ag}}(\lambda_k):=(\lambda_k/c_B^2)^{1/(2\alpha-2)}$.
For every $x\in M$ with $b(x)=r\ge4\max(r_{\mathrm{Ag}},R_0)$,
\begin{equation}\label{eq:agmonlb}
  \rho_{\lambda_k}(x) \;\ge\; \beta_\alpha\,r^\alpha,
  \qquad
  \beta_\alpha := \frac{c_B\sqrt{1-2^{-(2\alpha-2)}}}{2C_0\alpha}.
\end{equation}
By \Cref{lem:discrete}, $\lambda_k\to\infty$. Hence
$r_{\mathrm{Ag}}\ge R_0$ for all but finitely many $k$, and the
threshold is then $4r_{\mathrm{Ag}}$.
\end{lemma}
\begin{proof}
Abbreviate $\tilde r:=\max(r_{\mathrm{Ag}},R_0)$. Let $\gamma$ be any admissible
path, $\gamma(0)\in\mathcal A_k$, $\gamma(1)=x$, and write $L(\gamma)$
for its Agmon length.

\emph{Step 1.} Every point of $\mathcal A_k$ has $b\le\tilde r$.
Indeed, if $b>\tilde r$, then \eqref{eq:Vboundsb} gives
$V>c_B^2r_{\mathrm{Ag}}^{2\alpha-2}=\lambda_k$.
Let $t_0$ be the last time at which $b(\gamma(t_0))=\tilde r$.
Such a time exists by continuity. On $[t_0,1]$ the path stays in
$\{b\ge\tilde r\}\subset\{b\ge R_0\}$, where $b$ is $C^2$ and H2 applies.
Write $s=b\circ\gamma$ on this interval. Then $s$ is absolutely
continuous and
$|s'|\le|\nabla b|\,|\dot\gamma|_g\le C_0|\dot\gamma|_g$ a.e.
Dropping the earlier part of the path gives
\[
  L(\gamma)\ge\frac1{C_0}\int_{t_0}^1
  \sqrt{(V(\gamma(t))-\lambda_k)_+}\,|s'(t)|\,dt.
\]

\emph{Step 2.} Since $s(t_0)=\tilde r$ and $s(1)=r$, the
one-dimensional area formula gives, for every nonnegative measurable
function $g$,
\[
  \int_{t_0}^1 g(s(t))|s'(t)|\,dt
  \ge\int_{\tilde r}^r g(\sigma)\,d\sigma.
\]
For $s\ge\tilde r$, \eqref{eq:Vboundsb} gives
$V\ge c_B^2s^{2\alpha-2}\ge\lambda_k$.
Taking $g(s)=c_B\sqrt{s^{2\alpha-2}-r_{\mathrm{Ag}}^{2\alpha-2}}$
in the preceding estimates gives
\[
  L(\gamma)\ge\frac{c_B}{C_0}\int_{\tilde r}^r
  \sqrt{s^{2\alpha-2}-r_{\mathrm{Ag}}^{2\alpha-2}}\,ds,
\]
using $\lambda_k=c_B^2r_{\mathrm{Ag}}^{2\alpha-2}$.

\emph{Step 3.} Restrict to the subinterval $[2\tilde r,r]$. The integrand is
nonnegative, so this only decreases the bound. For $s\ge2\tilde r\ge2r_{\mathrm{Ag}}$,
\[
  r_{\mathrm{Ag}}^{2\alpha-2}\le(s/2)^{2\alpha-2}
  =2^{-(2\alpha-2)}s^{2\alpha-2}.
\]
Thus
\[
  \sqrt{s^{2\alpha-2}-r_{\mathrm{Ag}}^{2\alpha-2}}
  \ge\sqrt{1-2^{-(2\alpha-2)}}\,s^{\alpha-1}.
\]
Hence
\[
  L(\gamma) \;\ge\; \frac{c_B\sqrt{1-2^{-(2\alpha-2)}}}{C_0}\int_{2\tilde r}^rs^{\alpha-1}\,ds
  \;=\; \frac{c_B\sqrt{1-2^{-(2\alpha-2)}}}{C_0\alpha}\bigl(r^\alpha-(2\tilde r)^\alpha\bigr).
\]
Since $r\ge4\tilde r$, $2\tilde r\le r/2$, so
$(2\tilde r)^\alpha\le(r/2)^\alpha\le r^\alpha/2$, giving
$r^\alpha-(2\tilde r)^\alpha\ge r^\alpha/2$. It follows that
$L(\gamma)\ge\beta_\alpha r^\alpha$. Taking the infimum over $\gamma$
gives \eqref{eq:agmonlb}.
\end{proof}

\begin{corollary}\label{cor:taildecay}
Under H1--H3 with $\alpha>1$, let the eigenfunctions $u_k$ be
normalized in $L^2(M,e^{-f}dV)$. For every $\varepsilon\in(0,1)$ there is
$C_{\varepsilon}>0$. It may depend on the fixed metric, weight, and
constants in H1--H3, but not on $k$. For every $k$ with $\lambda_k>0$
and all $r\ge4\max(r_{\mathrm{Ag}}(k),R_0)$,
\begin{equation}\label{eq:taildecayL2}
  \int_{\{b>r\}}u_k^2e^{-f}\,\dV \;\le\; C_{\varepsilon}\,e^{-2(1-\varepsilon)\beta_\alpha r^\alpha}.
\end{equation}
\end{corollary}
\begin{proof}
We keep track of the constants in the weighted energy identity of
\Cref{prop:agmon}. Let $\psi=\psi_k=e^{-f/2}u_k$, so that
$\|\psi\|_{L^2(M,dV)}=1$, and set
$\varphi=e^{\theta\min(\rho_{\lambda_k},T)}$ with $\theta\in(0,1)$.
The form-domain argument in that proof gives
\begin{equation}\label{eq:groundstate}
  \int|\nabla(\varphi\psi)|^2\,dV + \int(V-\lambda_k)\varphi^2\psi^2\,dV
  \;=\; \int|\nabla\varphi|^2\psi^2\,dV.
\end{equation}
Let $V_{\min}:=\inf_MV> -\infty$. On $\mathcal A_k$ we have
$\rho_{\lambda_k}=0$ and $\varphi=1$, so
$\int_{\mathcal A_k}(\lambda_k-V)\varphi^2\psi^2\,dV
\le\lambda_k-V_{\min}$ by normalization.
The eikonal inequality gives
$|\nabla\varphi|^2\le\theta^2\varphi^2(V-\lambda_k)_+$ a.e.
Splitting the potential term at $\mathcal A_k$ and dropping the
nonnegative gradient term in \eqref{eq:groundstate} gives
\begin{equation}\label{eq:preunif}
  (1-\theta^2)\int_{\mathcal A_k^c}(V-\lambda_k)_+\varphi^2\psi^2\,dV
  \;\le\;\lambda_k-V_{\min}.
\end{equation}
Set $\Lambda_0:=c_B^2R_0^{2\alpha-2}>0$,
$\tilde r_k:=\max(r_{\mathrm{Ag}}(k),R_0)$, and
$\kappa_\alpha:=4^{2\alpha-2}-1>0$. On $\{b>4\tilde r_k\}$,
\Cref{lem:Vbounds2}(a) gives
\[
  V-\lambda_k
  \ge4^{2\alpha-2}\max(\lambda_k,\Lambda_0)-\lambda_k
  \ge\kappa_\alpha\max(\lambda_k,\Lambda_0).
\]
Using \eqref{eq:preunif} and then monotone convergence as $T\to\infty$,
\[
 \int_{\{b>4\tilde r_k\}}e^{2\theta\rho_{\lambda_k}}\psi_k^2\,dV
 \le\frac{\lambda_k-V_{\min}}
           {(1-\theta^2)\kappa_\alpha\max(\lambda_k,\Lambda_0)}
 \le\frac{1+(-V_{\min})_+/\Lambda_0}
           {(1-\theta^2)\kappa_\alpha}.
\]
One may take
\[
 C_\varepsilon:=
 \frac{1+(-V_{\min})_+/\Lambda_0}
      {\varepsilon(2-\varepsilon)\kappa_\alpha},
 \qquad \theta=1-\varepsilon.
\]
This constant is independent of $k$. By \Cref{lem:agmonlb},
$\rho_{\lambda_k}(x)\ge\beta_\alpha b(x)^\alpha$ whenever
$b(x)\ge4\tilde r_k$. For $r\ge4\tilde r_k$,
\[
  e^{2(1-\varepsilon)\beta_\alpha r^\alpha}
  \int_{\{b>r\}}\psi_k^2\,dV
  \le\int_{\{b>4\tilde r_k\}}
      e^{2(1-\varepsilon)\rho_{\lambda_k}}\psi_k^2\,dV
  \le C_\varepsilon.
\]
Since $\psi_k^2\,dV=u_k^2e^{-f}\,dV$, this proves \eqref{eq:taildecayL2}.
\end{proof}

\subsection{Transverse pinching produces zero-growth modes}\label{ssec:neckviolatesH4}

The following example shows that transverse geometry can create
high-energy eigenfunctions of growth order zero even when H1--H3
hold. In the natural separated eigenbasis this violates the
requirement $\gamma_k\to\infty$ in H4. Because H4 is stated for a
fixed eigenbasis, mixing vectors inside a repeated full eigenspace can
change the individual growth orders. The conclusion below is
stated for the separated eigenbasis. The global warped-product structure is
used to obtain reducing angular sectors and hence global separated
eigenfunctions. The estimates after that reduction use only the tail.

\begin{theorem}\label{thm:neckviolatesH4}
Let $(M^n,g)$, $n\ge2$, be a complete smooth exact warped product
as in \S\ref{sec:warped}, with closed connected fiber $N^{n-1}$. Let $f$ be radial and satisfy H1--H3 with $\alpha>1$.
Suppose the warping function tends to zero on every noncompact end.
Then there are separated eigenfunctions $u_{j,0}$ with
$\gamma_{j,0}=0$ and $\lambda_{j,0}\to\infty$.
Every eigenbasis chosen sector by sector fails H4.
\end{theorem}
\begin{proof}
\emph{Step 1.} For $u=p(r)\phi_j(\omega)$
with $-\Delta_N\phi_j=\mu_j\phi_j$, where
$0=\mu_0<\mu_1\le\mu_2\le\cdots\to\infty$ by Weyl's law on the
closed fiber $N$, the warped product Laplacian formula
$\Delta_Mu=u_{rr}+(n-1)\frac{\varphi'}\varphi u_r+\frac1{\varphi^2}\Delta_Nu$
and $L_fu=\Delta u-f'u_r$, for radial $f=f(r)$, give
\[
  L_fu = \left[p''+\Bigl((n-1)\frac{\varphi'}\varphi-f'\Bigr)p'
    -\frac{\mu_j}{\varphi^2}p\right]\phi_j.
\]
For each fixed $j$, $L_fu=-\lambda u$ therefore becomes a radial ODE for
$p$. There is no mode mixing, since the coefficients depend on $j$
only through the constant $\mu_j$.

\emph{Step 2.} Each angular sector inherits discreteness from $H$. By
Lemma~\ref{lem:conjugation}, $-L_f$ is unitarily equivalent, through
$\psi=e^{-f/2}u$, to $H=-\Delta+V$ on $L^2(M,\dV)$, where
$V=\frac14|\nabla f|^2-\frac12\Delta f$. By Lemma~\ref{lem:Vbounds2},
$V\to+\infty$ under H1--H3 alone, on any manifold satisfying them.
That lemma's proof uses only H3's bound on $\Delta f$ and the
resulting bounds on $|\nabla f|$, never $\varphi$. Lemma~\ref{lem:discrete}
then gives that $H$ has purely discrete spectrum on all of $M$.

Writing $\mathcal I$ for the base interval, separation gives
\[
L^2(M,\dV)=\bigoplus_{j\ge0}
L^2(\mathcal I,\varphi^{n-1}dr)\otimes\operatorname{span}\{\phi_j\}.
\]
The operator $H$ preserves this decomposition, since it commutes with
$-\Delta_N$. Write $H=\bigoplus_{j\ge0}H_j$. Since each angular
sector is a reducing subspace for $H$, the
restriction $H_j$ of the compact-resolvent operator $H$ to that
sector also has compact resolvent. Each $H_j$, $j\ge0$, has
purely discrete spectrum. In particular each sector,
including $j=0$, the radial one, has a well-defined lowest
eigenvalue $\lambda_{j,0}<\infty$, with ground state
$u_{j,0}=p_{j,0}(r)\phi_j(\omega)$.

\emph{Step 3.} Fix $j\ge1$ and use outward arclength $r$ on a fixed
end. Since $\varphi(r)\to0$ and $\mu_j>0$, there is $r_1$ such that
\[
  q_j(r):=\frac{\mu_j}{\varphi(r)^2}-\lambda_{j,0}>0
  \qquad\text{for all }r\ge r_1.
\]
Let
\[
  W(r):=\varphi(r)^{n-1}e^{-f(r)}.
\]
The radial sector equation from Step~1 can then be written in the
self-adjoint form
\begin{equation}\label{eq:neckWp}
  (W p')'=q_j Wp,
  \qquad p:=p_{j,0}.
\end{equation}
The lowest eigenfunction of the $j$-th one-dimensional sector has no
zeros. After changing sign, we may assume $p>0$.

\emph{Step 4.} The sector ground state is bounded on this end. By
\eqref{eq:neckWp},
\[
  (Wp')'=q_jWp>0
  \qquad (r\ge r_1),
\]
so $Wp'$ is strictly increasing there. We claim that $Wp'\le0$ on
$[r_1,\infty)$. Suppose instead that $Wp'(r_2)>0$ at some $r_2\ge r_1$.
Monotonicity then gives $Wp'(r)\ge c>0$ for every $r\ge r_2$, so
\[
  p'(r)\ge \frac{c}{W(r)},\qquad
  p(r)\ge c A(r),\qquad
  A(r):=\int_{r_2}^r\frac{ds}{W(s)},
\]
after decreasing $c$ if necessary. Since $f(r)\to\infty$ and
$\varphi(r)\to0$, one has $W(r)\to0$, so $A(r)\to\infty$.

We show that
\begin{equation}\label{eq:A2Wdiverges}
  \int_{r_2}^\infty A(r)^2W(r)\,dr=\infty.
\end{equation}
Indeed, put $t=A(r)$. Then $dt=dr/W(r)$, so
\[
  \int_{r_2}^\infty A(r)^2W(r)\,dr
  =\int_0^\infty t^2W(r(t))^2\,dt.
\]
If the latter integral were finite, Cauchy--Schwarz would give, for
any $T>0$,
\[
  \int_T^\infty W(r(t))\,dt
  \le
  \left(\int_T^\infty t^2W(r(t))^2\,dt\right)^{1/2}
  \left(\int_T^\infty t^{-2}\,dt\right)^{1/2}<\infty.
\]
On the other hand, by the change of variables $t=A(r)$,
\[
  \int_0^{A(R)}W(r(t))\,dt=R-r_2\longrightarrow\infty,
\]
a contradiction. This proves \eqref{eq:A2Wdiverges}.

Since $p\in L^2((r_2,\infty),W\,dr)$, the lower bound
$p(r)\ge cA(r)$ contradicts \eqref{eq:A2Wdiverges}. Hence $Wp'\le0$
on the tail. As $W>0$, $p'\le0$ there, so $p$ is bounded.

Apply this on every end. Normalize $\phi_j$ in $L^2(N)$. On a warped slice
$\{r\}\times N$,
\[
  \frac1{\vol(\{r\}\times N)}
  \int_{\{r\}\times N}u_{j,0}^2\,d\sigma
  =\frac{p(r)^2}{\vol(N)}.
\]
By H1 and H2, $f'(r)>0$ sufficiently far out on each end. A large
level of $b=f^{1/\alpha}$ is a union of one or two slices. The normalized average over the whole level set is a weighted average
of the bounded slice averages, with weights given by their volumes.
On the remaining compact region $u_{j,0}$ is bounded, so $N=0$
is admissible in H4, and
\[
  \gamma_{j,0}=0
\]
for every $j\ge1$.

\emph{Step 5.} The corresponding eigenvalues tend to infinity with
$j$. The warping function is bounded on the compact region and tends
to zero on every end. The supremum $\Phi:=\sup_{\mathcal I}\varphi$ is finite.
In the
Schr\"odinger realization of the $j$-th sector the additional angular
potential satisfies
\[
  \frac{\mu_j}{\varphi(r)^2}\ge \frac{\mu_j}{\Phi^2}.
\]
If $V_{\min}:=\inf_MV> -\infty$, the min--max principle yields
\[
  \lambda_{j,0}\ge \frac{\mu_j}{\Phi^2}+V_{\min}\longrightarrow\infty.
\]
The separated eigenfunctions $\{u_{j,0}\}_{j\ge1}$, each with
$\gamma_{j,0}=0$, occur at unbounded eigenvalues. Any spectral basis
chosen sector by sector has zero-growth vectors at
unbounded indices and fails H4.
\end{proof}

\begin{remark}
Theorem~\ref{thm:neckviolatesH4} requires $\varphi$ to tend to zero
on every noncompact end. It gives no obstruction
to a neck contained in a compact region. A compact change of
$\varphi$ leaves the tail equation of each fixed angular channel
unchanged. Whether the resulting separated eigenbasis satisfies H4
depends on that tail equation and is not addressed here. The
theorem concerns persistent transverse pinching, not the existence of
a finite neck.
\end{remark}

\section{Two Growth Filtrations}\label{sec:twofiltrations}

We compare the spectral filtration with the filtration by asymptotic
polynomial growth. The estimates below distinguish this growth order
from the finite-scale mass distribution.

\subsection{Motivation}

The two orderings can differ. In \Cref{prop:incompatibility},
angular energy produces eigenvalues much larger than the corresponding
growth degrees. The radial sectors studied later behave differently.

\subsection{The two filtrations}

\begin{definition}\label{def:filtrations}
Under H4, let $\lambda_k$ and $\gamma_k$ be the eigenvalues and
intrinsic growth orders defined there. The spectral filtration is
\[
  E_\lambda := \mathrm{span}\{u_k:\lambda_k\le\lambda\},\qquad\lambda\ge0.
\]
The growth filtration is
\[
  \mathcal P_N := \mathrm{span}\{u_k:\gamma_k\le N\},\qquad N\ge0.
\]
Both are nested. $E_{\lambda_1}\subset E_{\lambda_2}$ for
$\lambda_1\le\lambda_2$, and $\mathcal P_{N_1}\subset\mathcal P_{N_2}$
for $N_1\le N_2$.
\end{definition}

\begin{example}\label{ex:OU}
Let $M=\mathbb R^n$ and $f=|x|^2/4$. Then H1--H3 hold with
$\alpha=2$, $c_1=c_2=1/4$, $c_3=c_5=1$, any $c_4>0$, and any
$\eta\in(0,\tfrac12)$ once $R_0$ is large. The $SO(n)$-invariant sector is the setting for a radial eigenbasis
in \Cref{ass:transitive}. Its $j$-th eigenfunction is a degree-$j$
polynomial in $r^2$ with $\mu_j=j$, so $\gamma_j=2j$. On
$\mathbb R^1$ this is the full Hermite eigenbasis, with
$\lambda_k=k/2$ and $\gamma_k=k$. Either way, both sequences are
linear in their index. This is a special algebraic feature of the
Gaussian weight, which serves as the model case throughout.
\end{example}

\begin{proposition}
\label{prop:PN_structure}
Under H1--H4, the following hold. {\rm(i)} $d(N):=\dim\mathcal P_N<\infty$
for each $N$. {\rm(ii)} $\mathcal P_N\subset\mathcal P_{N+1}$.
{\rm(iii)} $\bigcup_N\mathcal P_N$ is dense in $L^2(M,e^{-f}\dV)$.
{\rm(iv)} $L_f(\mathcal P_N)=\mathcal P_N\ominus\mathbb Ru_0$, a
subspace of codimension $1$ in $\mathcal P_N$. The constant
eigenfunction $u_0$, with $\lambda_0=0$, is never in the image.
\end{proposition}
\begin{proof}
(i) By H4, $\gamma_k\to\infty$, so $\{k:\gamma_k\le N\}$ is finite.
(ii) This follows from the definition. (iii) By the spectral theorem, $\{u_k\}$ is a
complete orthonormal basis, and $\gamma_k\to\infty$ places every
$u_k$ in some $\mathcal P_N$. (iv) For $u=\sum_{\gamma_k\le N}c_ku_k$,
$L_fu=-\sum_{\gamma_k\le N}c_k\lambda_ku_k$. The $k=0$ term
contributes $0$, so $L_fu$ lies in
$\mathrm{span}\{u_k:\gamma_k\le N,\lambda_k>0\}=\mathcal P_N\ominus\mathbb Ru_0$.
Conversely, let $v=\sum_{\gamma_k\le N,\,\lambda_k>0}c_ku_k$, with no
$u_0$ component. Then $v=L_fu$ for
$u=-\sum_{\gamma_k\le N,\,\lambda_k>0}(c_k/\lambda_k)u_k\in\mathcal P_N$,
where the sum excludes $k=0$ and so avoids division by
$\lambda_0=0$. This proves equality with $\mathcal P_N\ominus\mathbb Ru_0$.
\end{proof}

\subsection{From filtrations to radii}\label{ssec:filtrationstoradii}

The spectral filtration provides a comparison scale through the
conjugated potential. The concentration radius depends on the full weighted mass of each
eigenfunction; its growth degree alone does not determine it.

\subsubsection{The turning radius}

Setting $\psi_k:=u_ke^{-f/2}$, the equation
$L_fu_k=-\lambda_ku_k$ becomes $-\Delta\psi_k+V\psi_k=\lambda_k\psi_k$,
where $V:=\tfrac14|\nabla f|^2-\tfrac12\Delta f$.
By \Cref{lem:Vbounds2}(a), valid for $\alpha>1$, $c_B^2r^{2\alpha-2}\le
V\le C_B^2r^{2\alpha-2}$ on $\Sigma_r$ for $r\ge R_0$, with
$c_B,C_B$ as given there.

\begin{definition}\label{def:Agmon}
For $\lambda>0$ define the \emph{Agmon safe radius}
\begin{equation}\label{eq:AgmonSafe}
  r_{\mathrm{Ag}}(\lambda)
  := \left(\frac{\lambda}{c_B^2}\right)^{1/(2\alpha-2)},
  \qquad r_{\mathrm{Ag}}(k):=r_{\mathrm{Ag}}(\lambda_k).
\end{equation}
For comparison, set
\begin{equation}\label{eq:AgmonInner}
  r_{\mathrm{Ag}}^-(\lambda)
  := \left(\frac{\lambda}{C_B^2}\right)^{1/(2\alpha-2)}.
\end{equation}
Then $r_{\mathrm{Ag}}^-(\lambda)\le r_{\mathrm{Ag}}(\lambda)$ and both are
comparable to $\lambda^{1/(2\alpha-2)}$.  On the tail $\{b\ge R_0\}$,
\[
 b<r_{\mathrm{Ag}}^-(\lambda) \Longrightarrow V<\lambda,
 \qquad
 b>r_{\mathrm{Ag}}(\lambda) \Longrightarrow V>\lambda.
\]
The radius $r_{\mathrm{Ag}}$ is the outer, or safe, radius beyond which the
potential is certainly in the forbidden regime.  When $V$ is radial,
any actual solution of $V(r)=\lambda$ lies between
$r_{\mathrm{Ag}}^-(\lambda)$ and $r_{\mathrm{Ag}}(\lambda)$.
\end{definition}

The radii above depend only on the potential bounds and $\lambda$, and are
independent of the eigenfunction. In the nonradial setting they are comparison
radii; the level sets of $V$ need not coincide with those of $b$.  The terminology comes from the Agmon distance and the exponential decay
estimates of \Cref{def:agmon} and \Cref{prop:agmon}. Outside
$r_{\mathrm{Ag}}(\lambda)$ the potential is uniformly above the energy.
Inside $r_{\mathrm{Ag}}^-(\lambda)$ it is uniformly below the energy on the tail.

\subsubsection{The concentration radius}

\begin{definition}\label{def:massdensity}
The concentration radius is
\begin{equation}\label{eq:medianDef}
  r_C(k) := \inf\Bigl\{r>0:\int_{\{b\le r\}}u_k^2e^{-f}\dV\ge\tfrac12\Bigr\}.
\end{equation}
For later use, set $m_k(r):=J_k(r)e^{-f(r)}$,
$J_k(r):=\vol(\Sigma_r)I_k(r)=\int_{\Sigma_r}u_k^2|\nabla b|\,d\sigma$.
The definition of $r_C(k)$ uses the cumulative mass $F$
(\Cref{def:FEPhi}), not the auxiliary quantity $m_k$.
The latter equals the mass density $F_k'$ only when $|\nabla b|\equiv1$.
In general it is comparable to $F_k'$ from both sides. See
\Cref{lem:IFprime} below. This comparison turns a bound on $m_k$
into a bound on the mass. H4 is useful here because it controls $I_k$.
We then obtain bounds on $F_k$ and $r_C(k)$, as in the proof of
\Cref{prop:moment}.
\end{definition}

\begin{lemma}\label{lem:PhiUm}
$\Phi(r)=U(r)m(r)/r$, where $m(r)$ is the mass density just defined,
and $\Phi$, $U$ are as in Definitions~\ref{def:FEPhi} and
\ref{def:freq} respectively.
\end{lemma}
\begin{proof}
By Definition~\ref{def:FEPhi},
\[
  \Phi(r)=e^{-f(r)}\vol(\Sigma_r)D(r).
\]
Definition~\ref{def:massdensity} gives
$m(r)=I(r)\vol(\Sigma_r)e^{-f(r)}$, while $U=rD/I$ implies $D=UI/r$.
Therefore
\[
  \Phi(r)=\frac{U(r)m(r)}r.
\]
\end{proof}

The following two results describe the limitations of using critical
points of $m_k$ to locate its mass.

\begin{lemma}\label{lem:IFprime}
Under H2, for all regular $r\ge R_0$,
\[
  c_0^2\,e^{f(r)}F'(r) \;\le\; I(r)\,\vol(\Sigma_r) \;\le\; C_0^2\,e^{f(r)}F'(r).
\]
\end{lemma}
\begin{proof}
Since $f=f(b)$ is radial, $f\equiv f(r)$ on $\Sigma_r$, so
\[
  I(r)\vol(\Sigma_r) = \int_{\Sigma_r}u^2|\nabla b|\,d\sigma, \qquad
  e^{f(r)}F'(r) = \int_{\Sigma_r}u^2\,\frac{1}{|\nabla b|}\,d\sigma.
\]
By Remark~\ref{rem:gradb}, $c_0\le|\nabla b|\le C_0$ pointwise, so
$c_0^2/|\nabla b|\le|\nabla b|\le C_0^2/|\nabla b|$. Integrating
against $u^2\,d\sigma$ proves the claim.
\end{proof}

\begin{definition}\label{def:growthprefactor}
For each eigenfunction set
\begin{equation}\label{eq:Bk}
  B_k:=\max\left\{1,\sup_{r\ge R_0}
  r^{-2(\gamma_k+1)}I_k(r)\right\}.
\end{equation}
Then $B_k<\infty$ by H4 and \Cref{lem:gammaequiv}, since
$\gamma_k+1>\gamma_k$. The number $B_k$ measures finite-scale size. H4 makes it finite for each
$k$, but does not compare these constants as $k$ varies. The number
$\gamma_k$ measures only the asymptotic polynomial degree.
\end{definition}

\begin{proposition}\label{prop:moment}
Under H1--H4, let
\[
  G_k:=\gamma_k+1,\qquad
  S_k:=G_k\log(e+G_k)+\log B_k.
\]
For every $p>0$, with $q:=p+2G_k$,
\begin{equation}\label{eq:momentbound}
  M_p(k):=\int_{R_0}^\infty r^p m_k(r)\,dr
  \le B_k C_M^{\,q+1}\Gamma\!\left(\frac{q+1}{\alpha}\right),
\end{equation}
where $C_M$ is the constant of \eqref{eq:volmoment}. Consequently
there is a constant $C$, independent of $k$, such that
\begin{equation}\label{eq:concUBmoment}
  r_C(k)\le C\left(R_0+S_k^{1/\alpha}\right).
\end{equation}
\end{proposition}
\begin{proof}
By \eqref{eq:Bk},
$I_k(r)\le B_kr^{2G_k}$ for $r\ge R_0$. Thus
\[
  m_k(r)\le B_kr^{2G_k}e^{-f(r)}\vol(\Sigma_r).
\]
With $q=p+2G_k$, \eqref{eq:volmoment} gives
\[
  M_p(k)\le B_k\int_{R_0}^\infty
  r^q e^{-f(r)}\vol(\Sigma_r)\,dr
  \le B_k C_M^{q+1}\Gamma\!\left(\frac{q+1}{\alpha}\right),
\]
which proves \eqref{eq:momentbound}.

Set $b_k:=\log B_k\ge0$ and choose
\[
  p:=2S_k=2\bigl(G_k\log(e+G_k)+b_k\bigr).
\]
Then $p\ge2G_k$ and $B_k^{1/p}\le e^{1/2}$. Put
$x:=(q+1)/\alpha$. Since $S_k\ge G_k\log(e+G_k)$,
\[
  \frac{q+1}{p}=1+\frac{2G_k+1}{2S_k}\le2,
  \qquad x\asymp_\alpha S_k.
\]
Stirling's bound gives, after enlarging a constant depending only on
$\alpha$ and $C_M$,
\[
  M_p(k)^{1/p}
  \le C S_k^{1/\alpha}
  \exp\left(C\frac{G_k\log(e+S_k)}{S_k}\right).
\]
The exponential factor is uniformly bounded. Indeed,
$S_k\ge G_k\log(e+G_k)$ and the function
$\log(e+s)/s$ is decreasing for $s>0$, so
\[
  \frac{G_k\log(e+S_k)}{S_k}
  \le \frac{\log(e+G_k\log(e+G_k))}{\log(e+G_k)}\le C.
\]
Thus
\begin{equation}\label{eq:MpSk}
  M_p(k)^{1/p}\le C S_k^{1/\alpha}.
\end{equation}

It remains to pass from the moment to the median radius. By
\Cref{lem:IFprime},
$c_0^2F'(r)\le m_k(r)$ for $r\ge R_0$. Hence
\[
  \int_0^\infty r^pF'(r)\,dr
  \le R_0^p+c_0^{-2}M_p(k).
\]
For $t\ge R_0$,
\[
  1-F(t)\le t^{-p}\left(R_0^p+c_0^{-2}M_p(k)\right).
\]
Choose
$t^p=2R_0^p+2c_0^{-2}M_p(k)$. Then $F(t)\ge1/2$, so
$r_C(k)\le t$. Since $p\ge1$,
\[
  r_C(k)\le4\max(1,c_0^{-2})
  \left(R_0+M_p(k)^{1/p}\right).
\]
Now use \eqref{eq:MpSk}. This proves \eqref{eq:concUBmoment}.
\end{proof}

\medskip
\noindent\textbf{Conditional consequence.}
If
\begin{equation}\label{eq:Bksubexp}
  \log B_k\lesssim G_k\log(e+G_k),
\end{equation}
with an implicit constant independent of $k$, then
$S_k\lesssim G_k\log(e+G_k)$. Hence \Cref{prop:moment} gives
\[
  r_C(k)\lesssim\bigl(\gamma_k\log(e+\gamma_k)\bigr)^{1/\alpha}
\]
for all sufficiently large $k$.

\begin{remark}\label{rem:logsharp}
H4 shows that $B_k<\infty$ for every $k$, but it does not give an
estimate for $B_k$ that is uniform in $k$. The constants in the growth
bound defining $\gamma_k$ are allowed to depend on the eigenfunction.
In \Cref{prop:moment}, the finite-scale prefactor enters only through
$\log B_k$.
\end{remark}

\begin{remark}\label{rem:normalizedH4}
A stronger normalized growth bound produces the model scale without the
logarithmic loss. Suppose there are constants $C,\kappa>0$, independent
of $k$, such that, for all sufficiently large $k$,
\begin{equation}\label{eq:normalizedH4}
  I_k(r)\le C\left(\frac{\kappa r}{\gamma_k^{1/\alpha}}\right)^{2\gamma_k}
  \qquad(r\ge R_0).
\end{equation}
Then
\[
  m_k(r)\le C\kappa^{2\gamma_k}\gamma_k^{-2\gamma_k/\alpha}
  r^{2\gamma_k}e^{-f(r)}\vol(\Sigma_r).
\]
An argument similar to the proof of \Cref{prop:moment} applies with
$p=2\gamma_k$. Then
$q=4\gamma_k$ and $x=(4\gamma_k+1)/\alpha$. Stirling then gives
\[
  M_p(k)^{1/p}\lesssim
  \gamma_k^{-1/\alpha}(x/e)^{x/p}
  \lesssim\gamma_k^{1/\alpha}.
\]
This proves $r_C(k)\lesssim\gamma_k^{1/\alpha}$. This additional growth bound
is sufficient. It does not follow from H4.
\end{remark}

\begin{remark}\label{rem:concLBmissing}
H4 is an asymptotic growth condition. It does not by itself prevent
an eigenfunction from placing a large fraction of its weighted mass
well inside the scale $\gamma_k^{1/\alpha}$. Such finite-scale
behavior does not change $\gamma_k$. A lower bound
$r_C(k)\gtrsim\gamma_k^{1/\alpha}$ needs information beyond
the definition of H4. The moment condition H5 below provides such information.
\end{remark}

\begin{openquestion}\label{prob:LLPunconditional}
Do H1--H4 imply the Lower Localization Property?
A related question is whether they imply
$\log B_k\lesssim G_k\log(e+G_k)$ for the prefactors in \eqref{eq:Bk}.
The latter would give an upper concentration bound in terms of the
growth order alone.
\end{openquestion}

\subsection{Compatibility}\label{ssec:compatibility}

\begin{definition}\label{def:compatibility}
Assume $\alpha>1$. The spectral and growth filtrations are compatible if there are constants
$0<c\le C<\infty$ and an index $K$ such that the following estimate holds.
The constants are independent of $k$, but may depend on the fixed manifold,
weight, and constants in the standing hypotheses.
\begin{equation}\label{eq:compatibilitygoal}
  c\,\gamma_k^{(2\alpha-2)/\alpha} \;\le\; \lambda_k \;\le\; C\,\gamma_k^{(2\alpha-2)/\alpha}
  \quad\text{for all } k\ge K.
\end{equation}
This asymptotic formulation is essential: H4 allows finitely many positive
eigenfunctions with $\gamma_k=0$, and the compact core is unconstrained by
H1--H3.
\end{definition}

Compatibility is stated in terms of $(\lambda_k,\gamma_k)$ rather
than the radii $r_C(k),r_{\mathrm{Ag}}(k)$. By \Cref{def:Agmon}, the safe radius satisfies
$r_{\mathrm{Ag}}(k)\asymp\lambda_k^{1/(2\alpha-2)}$. A model growth
profile $u_k^2\sim r^{2\gamma_k}$ produces the model density
$r^{2\gamma_k}e^{-r^\alpha}$, whose peak lies at
$r\asymp\gamma_k^{1/\alpha}$. Matching these two model scales leads to
\eqref{eq:compatibilitygoal}. This calculation does not locate the
true mass of $u_k$. The actual concentration radius depends also on
finite-scale information, as \Cref{prop:moment} makes explicit.
LLP is the lower localization estimate needed for the spectral comparison.

\begin{proposition}\label{prop:incompatibility}
There is a complete smooth rotationally symmetric weighted surface
satisfying H1--H4 with $\alpha=2$ for which $\lambda_k/\gamma_k$ is
unbounded as $k\to\infty$. Its radial eigenfunctions form a compatible subsequence.
\end{proposition}
\begin{proof}
Fix an irrational $a\in(0,1)$. On $\R^2$, let $r$ denote distance
from the origin, which is the pole of the rotationally symmetric metric, and set
\[
  g=dr^2+\frac{r^2}{1+ar^2}\,d\theta^2,
  \qquad f=\frac{r^2}{2}-\frac12\log(1+ar^2).
\]
Write
\[
  \varphi(r):=\frac{r}{\sqrt{1+ar^2}},
  \qquad g=dr^2+\varphi(r)^2d\theta^2.
\]
Since $\varphi(r)=r-\frac a2r^3+O(r^5)$ at the origin, the metric
extends smoothly across the pole. It is complete because the radial
coordinate has infinite length. Also
\[
  f'(r)=r-\frac{ar}{1+ar^2}
       =r\left(1-\frac{a}{1+ar^2}\right)\ge(1-a)r,
\]
so $f\ge0$. On the tail,
\[
  f(r)=\frac12r^2+O(\log r),
  \qquad f'(r)=r+O(r^{-1}),
\]
and hence $f\asymp r^2$ and $|\nabla f|^2\asymp r^2$. Moreover,
\[
  \frac{\varphi'}{\varphi}=\frac{1}{r(1+ar^2)},
  \qquad
  \Delta_g f=f''+\frac{\varphi'}{\varphi}f'=1+O(r^{-2}).
\]
After increasing $R_0$ if necessary, these estimates give H1--H3
with $\alpha=2$.

The weighted measure and the drift operator can be written explicitly.
Since
\[
  dV_g=\frac{r}{\sqrt{1+ar^2}}\,dr\,d\theta,
  \qquad
  e^{-f}=e^{-r^2/2}\sqrt{1+ar^2},
\]
we have
\[
  e^{-f}dV_g=r e^{-r^2/2}\,dr\,d\theta.
\]
Also
\[
  \Delta_g
  =\partial_r^2+\frac{1}{r(1+ar^2)}\partial_r
   +\left(\frac1{r^2}+a\right)\partial_\theta^2.
\]
Using the formula for $f'$ above,
\[
  \frac{1}{r(1+ar^2)}-f'(r)=\frac1r-r.
\]
Therefore
\[
  L_f
  =\partial_r^2+\left(\frac1r-r\right)\partial_r
   +\frac1{r^2}\partial_\theta^2+a\partial_\theta^2
  =L_{\rm OU}+a\partial_\theta^2,
\]
where $L_{\rm OU}=\Delta_{\R^2}-r\partial_r$.

Let $-Y_\ell''=\ell^2Y_\ell$. Put $z=r^2/2$ and
\[
  R_{m,\ell}(r):=r^\ell L_m^{(\ell)}(z).
\]
The Laguerre equation
\[
  z(L_m^{(\ell)})''+(\ell+1-z)(L_m^{(\ell)})'
  +mL_m^{(\ell)}=0
\]
gives, after substitution in the polar form of $L_{\rm OU}$,
\[
  L_{\rm OU}\bigl(R_{m,\ell}Y_\ell\bigr)
  =-(2m+\ell)R_{m,\ell}Y_\ell.
\]
Together with this decomposition, this yields
\[
  u_{m,\ell}=c_{m,\ell}r^\ell L_m^{(\ell)}(r^2/2)Y_\ell(\theta),
  \qquad
  \lambda_{m,\ell}=2m+\ell+a\ell^2,
  \quad m,\ell\ge0.
\]
For $\ell\ge1$, include both sine and cosine angular modes.
Using the weighted measure above, Fourier orthogonality in
$\theta$ and Laguerre orthogonality in $z$ show that these functions
form a complete orthonormal basis; see
\cite[\S\S18.2--18.3, 18.8]{DLMF}.

Set $d:=2m+\ell$. Since $L_m^{(\ell)}$ has degree $m$ with nonzero
leading coefficient, the radial factor in $u_{m,\ell}$ has degree
$d$ in $r$. Let $r(\rho)$ be the radial coordinate of the level
$\Sigma_\rho=\{b=\rho\}$. Since $b=\sqrt f\asymp r$, we have
$r(\rho)\asymp\rho$. The normalized circle average is therefore
\[
  \frac1{\vol(\Sigma_\rho)}\int_{\Sigma_\rho}u_{m,\ell}^2\,d\sigma
  \asymp r(\rho)^{2d}\asymp\rho^{2d}
  \qquad(\rho\to\infty),
\]
where the angular average of $Y_\ell^2$ is a positive constant.
Thus
\[
  \gamma_{m,\ell}=d=2m+\ell.
\]
For every $\Lambda<\infty$, the inequalities
$2m\le\lambda_{m,\ell}\le\Lambda$ and
$a\ell^2\le\lambda_{m,\ell}\le\Lambda$ leave only finitely many
pairs $(m,\ell)$. Moreover,
\[
  \lambda_{m,\ell}=d+a\ell^2\le d+ad^2.
\]
Hence $\lambda_{m,\ell}\to\infty$ in spectral order forces
$d=\gamma_{m,\ell}\to\infty$. This proves H4 for the displayed
eigenbasis.

The failure of compatibility comes from the angular modes:
\[
  \frac{\lambda_{0,\ell}}{\gamma_{0,\ell}}
  =1+a\ell\longrightarrow\infty.
\]
On the other hand, the radial modes satisfy
\[
  \lambda_{m,0}=\gamma_{m,0}=2m,
\]
so they form a compatible subsequence.

It remains to check that the failure does not depend on the choice of
eigenbasis inside repeated eigenspaces. If
$\lambda_{m,\ell}=\lambda_{m',\ell'}$, then
\[
  2(m-m')+(\ell-\ell')+a(\ell^2-\ell'^2)=0.
\]
Since $a$ is irrational, $\ell\ne\ell'$ is impossible. Thus
$\ell=\ell'$ and then $m=m'$. The only multiplicity is the sine--cosine
multiplicity for a fixed pair $(m,\ell)$ with $\ell\ge1$. Every
nonzero linear combination of these two angular modes has the same
radial factor and the same growth order $2m+\ell$. Hence the
unbounded ratio above is independent of the eigenbasis.
\end{proof}

\begin{openquestion}\label{q:compatiblesubsequence}
Under H1--H4 with $\alpha>1$, must the eigenbasis in H4 contain
indices $k_j\to\infty$ such that
\[
  \lambda_{k_j}\asymp\gamma_{k_j}^{(2\alpha-2)/\alpha},
\]
with comparison constants independent of $j$?
\end{openquestion}

\subsection{Comparing the two radii}\label{ssec:comparingradii}

\begin{theorem}\label{thm:rCleqrAg}
Under H1--H4 with $\alpha>1$, there is $C=C(\alpha,c_1,\ldots,c_5,\eta)$
and $K_0$, which depends on the operator $L_f$ and not only on
$\alpha,c_1,\ldots,c_5,\eta$, such that
\begin{equation}\label{eq:rCleqrAg}
  r_C(k) \;\le\; C\,r_{\mathrm{Ag}}(k) \qquad\text{for all } k\ge K_0.
\end{equation}
\end{theorem}
\begin{proof}
Fix $\varepsilon=1/2$ in \Cref{cor:taildecay}. Take $k$ large enough that $r_{\mathrm{Ag}}(k)\ge R_0$.
This requirement is included in the choice of $K_0$ below.
Set $r=4r_{\mathrm{Ag}}(k)$ in \eqref{eq:taildecayL2}. Then
\[
  1-F_k\bigl(4r_{\mathrm{Ag}}(k)\bigr) \;=\; \int_{\{b>4r_{\mathrm{Ag}}(k)\}}u_k^2e^{-f}\,\dV
  \;\le\; C_{1/2}\,e^{-\beta_\alpha(4r_{\mathrm{Ag}}(k))^\alpha}.
\]
By \eqref{eq:AgmonSafe},
$r_{\mathrm{Ag}}(k)=c_B^{-1/(\alpha-1)}\lambda_k^{1/(2\alpha-2)}$,
so $r_{\mathrm{Ag}}(k)\to\infty$ by H4. The
right side tends to $0$, and there is $K_0$ with
$F_k(4r_{\mathrm{Ag}}(k))>1/2$ for all $k\ge K_0$. By definition of
$r_C(k)$, this gives $r_C(k)\le4r_{\mathrm{Ag}}(k)$ for all $k\ge K_0$,
which is \eqref{eq:rCleqrAg} with $C=4$.
\end{proof}

\begin{remark}\label{rem:K0scope}
H1--H3 are asymptotic hypotheses, stated only for $r\ge R_0$.
They do not restrict the geometry of $(M,g)$ or the weight $f$
on the compact core $\{r<R_0\}$. The finitely many low eigenvalues and eigenfunctions depend on that core.
For $k<K_0$, there is no reason to expect
$r_C(k)/r_{\mathrm{Ag}}(k)$ to be bounded by a constant depending only on
$\alpha,c_1,\ldots,c_5,\eta$. Two operators sharing
identical asymptotic constants can have arbitrarily different spectra
in this range, by varying the unconstrained geometry of the compact
core. The threshold $K_0$ in \Cref{thm:rCleqrAg} depends on the
operator. We have not proved that this dependence is unavoidable,
but we expect a universal $K_0$ to fail without control of the
compact core.
\end{remark}

\begin{remark}\label{rem:compatapriori}
We compare $r_C$ and $r_{\mathrm{Ag}}$ as follows. LLP gives
$\gamma_k^{1/\alpha}\lesssim r_C(k)$, and
$r_{\mathrm{Ag}}(k)\asymp\lambda_k^{1/(2\alpha-2)}$. The unconditional
comparison proved here is
$r_C(k)\lesssim r_{\mathrm{Ag}}(k)$ for all sufficiently large $k$,
\Cref{thm:rCleqrAg}. These two bounds give the lower spectral
inequality in Compatibility. The reverse inequality needs a separate argument. The moment estimate \Cref{prop:moment} provides another upper bound
on $r_C(k)$ when an upper estimate for the finite-scale prefactor $B_k$
is available. H4 gives $B_k<\infty$ for each $k$, but not a bound of
this kind that is uniform in $k$.
\end{remark}

\section{The Localization Problem, and a Sufficient Hypothesis for
LLP}\label{sec:localization}\label{sec:compatibility}

\subsection{The Lower Localization Property and the Reduction Theorem}

\begin{definition}\label{def:LLP}
$\{u_k\}$ satisfies the Lower Localization Property if there exist
$c>0$ and an index $K$, independent of $k$, such that
\begin{equation}\label{eq:LLP}
  r_C(k) \;\ge\; c\,\gamma_k^{1/\alpha} \qquad\text{for all }k\ge K.
\end{equation}
\end{definition}

\begin{theorem}\label{thm:reduction}
Under H1--H4 with $\alpha>1$, the following hold.
\begin{enumerate}[label=\rm(\alph*)]
\item Unconditionally,
  $r_C(k)\lesssim r_{\mathrm{Ag}}(k)$ for all sufficiently large $k$.
  The independent moment estimate is \eqref{eq:concUBmoment}.
\item If LLP holds, then
  $\lambda_k\gtrsim\gamma_k^{(2\alpha-2)/\alpha}$.
\item The reverse inequality fails in general
  (\Cref{prop:incompatibility}). Under \Cref{ass:transitive} and
  \eqref{eq:Dregularity}, it holds for the invariant sector
  (\Cref{cor:SL-compat}).
\end{enumerate}
\end{theorem}
\begin{proof}
Part (a) is \Cref{thm:rCleqrAg}, together with
\Cref{prop:moment}. If LLP holds, then
$\gamma_k^{1/\alpha}\lesssim r_C(k)\lesssim r_{\mathrm{Ag}}(k)$.
By \Cref{def:Agmon},
$r_{\mathrm{Ag}}(k)\asymp\lambda_k^{1/(2\alpha-2)}$. Combining these estimates gives
$\gamma_k^{1/\alpha}\lesssim\lambda_k^{1/(2\alpha-2)}$, which is
part (b). Part (c) follows from
\Cref{prop:incompatibility} and \Cref{cor:SL-compat}.
\end{proof}

\subsection{A two-moment condition for LLP}

For each normalized eigenfunction, let
\[
  d\pi_k:=u_k^2e^{-f}\,dV,\qquad
  \overline f_k:=\int_M f\,d\pi_k,\qquad
  s_k^2:=\int_M(f-\overline f_k)^2\,d\pi_k.
\]
Here $\pi_k$ is a probability measure and $s_k$ is the standard
deviation of $f$ with respect to $\pi_k$. These moments are finite,
since on the tail H4 and H2 give a polynomial bound for the coarea
integrand, and \Cref{lem:volmoment} integrates every polynomial power
against the weighted level-set volume. The compact core causes no
problem because $f$ is bounded there.

\begingroup
\renewcommand{\theassumption}{H5}
\renewcommand{\theHassumption}{H5}
\begin{assumption}[Two-moment localization]\label{ass:H5}
In addition to H1--H4, there are constants $c_H>0$ and $K_H$ such that
\begin{equation}\label{eq:H5moment}
  \overline f_k-s_k\ge c_H\gamma_k
  \qquad\text{for all }k\ge K_H.
\end{equation}
\end{assumption}
\endgroup
\addtocounter{assumption}{-1}

\begin{proposition}\label{prop:H5LLP}
Under H1--H5,
\[
  r_C(k)^\alpha\ge \overline f_k-s_k\ge c_H\gamma_k
  \qquad(k\ge K_H).
\]
Hence LLP holds with constant $c_H^{1/\alpha}$.
\end{proposition}
\begin{proof}
Fix $k$ and write $\mu:=\overline f_k$ and $s:=s_k$. If $s=0$, then
$f=\mu$ $\pi_k$-almost everywhere, and the first inequality is immediate.
Suppose $s>0$, and let $t<\mu-s$. Put $d:=\mu-t>s$ and
$a:=s^2/d$. On $\{f\le t\}$,
$|(f-\mu)-a|\ge d+a$. Applying the preceding estimate gives
\[
  \pi_k\{f\le t\}
  \le \frac{\int_M((f-\mu)-a)^2\,d\pi_k}{(d+a)^2}
  =\frac{s^2+a^2}{(d+a)^2}
  =\frac{s^2}{s^2+d^2}<\frac12.
\]
The lower median of $f$ is at least $\mu-s$. Since $f=b^\alpha$,
this median is $r_C(k)^\alpha$. The second inequality is H5.
\end{proof}

\begin{proposition}\label{prop:H5consequence}
Under H1--H5,
\[
  r_C(k)\gtrsim\gamma_k^{1/\alpha}
\]
for all sufficiently large $k$. If $\alpha>1$, then also
\[
  \lambda_k\gtrsim\gamma_k^{(2\alpha-2)/\alpha}.
\]
If, in addition, \eqref{eq:Bksubexp} holds, then
\[
  \gamma_k^{1/\alpha}\lesssim r_C(k)
  \lesssim\bigl(\gamma_k\log(e+\gamma_k)\bigr)^{1/\alpha}.
\]
\end{proposition}
\begin{proof}
The lower bound is \Cref{prop:H5LLP}. For $\alpha>1$, combine it
with \Cref{thm:rCleqrAg} and \Cref{def:Agmon}. The final bound follows from \eqref{eq:Bksubexp} and the conditional
consequence following \Cref{prop:moment}.
\end{proof}

\begin{remark}\label{rem:H5moment}
The estimate in \Cref{prop:H5LLP} is sharp if only the mean and
variance are known. For $\mu>s>0$, a probability measure with equal mass at
$\mu-s$ and $\mu+s$ has mean $\mu$, standard deviation $s$, and lower
median $\mu-s$. H5 is a sufficient two-moment condition for LLP,
not a necessary one. We do not claim that H1--H4 imply H5.

The condition holds for the standard Gaussian basis. Consider the
tensor-product Hermite basis on $\mathbb R^n$ with $f=|x|^2/4$, and
write $\nu=(\nu_1,\ldots,\nu_n)$ and
$m=|\nu|=\gamma_\nu$. Under the probability measure
$d\pi_\nu=u_\nu^2e^{-f}dx$, the coordinate factors are independent and
\[
  f=\sum_{i=1}^n\frac{x_i^2}{4}.
\]
For the normalized one-dimensional Hermite factor of degree $\nu_i$,
the Hermite recurrence \cite[\S18.9]{DLMF}, applied twice to the
multiplication operator, gives
\[
  \mathbb E_\nu\!\left(\frac{x_i^2}{4}\right)
     =\nu_i+\frac12,
  \qquad
  \operatorname{Var}_\nu\!\left(\frac{x_i^2}{4}\right)
     =\frac12\bigl(\nu_i^2+\nu_i+1\bigr).
\]
Since the coordinate variables are independent, their covariances
vanish. Summing the means and variances gives
\[
  \overline f_\nu=m+\frac n2,
  \qquad
  s_\nu^2
  =\frac12\Bigl(\sum_{i=1}^n\nu_i^2+m+n\Bigr).
\]
Because $\sum_i\nu_i^2\le m^2$,
\[
  s_\nu^2\le\frac12(m^2+m+n),
  \qquad
  \frac{s_\nu}{m}
  \le\frac1{\sqrt2}\sqrt{1+\frac1m+\frac{n}{m^2}}
  \longrightarrow\frac1{\sqrt2}.
\]
In particular, $s_\nu\le3m/4$ for all sufficiently large $m$.
Since $\overline f_\nu\ge m$, we obtain
\[
  \overline f_\nu-s_\nu\ge\frac m4
\]
for all sufficiently large $m$. Thus H5 holds for this full
eigenbasis.

H5 also holds for the separated basis in \Cref{prop:incompatibility}.
Set
\[
  d:=2m+\ell=\gamma_{m,\ell},
  \qquad z:=\frac{r^2}{2}.
\]
By the weighted-measure identity in the proof of \Cref{prop:incompatibility}, after integrating the angular
variable the $z$-marginal of $d\pi_{m,\ell}$ is proportional to
\[
  z^\ell\bigl(L_m^{(\ell)}(z)\bigr)^2e^{-z}\,dz.
\]
The Laguerre three-term recurrence and orthogonality
\cite[\S18.9]{DLMF} therefore give
\[
  \mathbb E z=2m+\ell+1=d+1
\]
and
\begin{align*}
  \operatorname{Var}(z)
  &=(m+1)(m+\ell+1)+m(m+\ell)\\
  &=\frac{(d+1)^2+1-\ell^2}{2}
  \le\frac{(d+1)^2+1}{2}.
\end{align*}
In this example
\[
  f=h(z),
  \qquad
  h(z):=z-\frac12\log(1+2az).
\]
Since
\[
  h'(z)=1-\frac{a}{1+2az}\in(0,1],
  \qquad
  h''(z)=\frac{2a^2}{(1+2az)^2}>0,
\]
the function $h$ is convex and $1$-Lipschitz. Jensen's inequality
then gives
\[
  \overline f_{m,\ell}=\mathbb E h(z)
  \ge h(\mathbb E z)
  =d+1-\frac12\log\bigl(1+2a(d+1)\bigr).
\]
For the standard deviation, the minimizing property of the mean and
the Lipschitz bound give
\[
\begin{aligned}
  s_{m,\ell}^2
  &=\operatorname{Var}(h(z))\\
  &\le \mathbb E\bigl(h(z)-h(\mathbb E z)\bigr)^2\\
  &\le \mathbb E\bigl(z-\mathbb E z\bigr)^2
   =\operatorname{Var}(z).
\end{aligned}
\]
Consequently
\[
  \overline f_{m,\ell}-s_{m,\ell}
  \ge d+1-\frac12\log\bigl(1+2a(d+1)\bigr)
      -\sqrt{\frac{(d+1)^2+1}{2}}.
\]
After division by $d$, the right-hand side converges to
$1-1/\sqrt2>0$. Hence there are $c>0$ and $d_0$, independent of
$m$ and $\ell$, such that
\[
  \overline f_{m,\ell}-s_{m,\ell}\ge c\,d
  \qquad(d\ge d_0).
\]
Thus H5 gives the lower localization estimate in this example even
though the upper spectral inequality fails for the full spectrum.
\end{remark}

H5 gives the lower localization estimate. The upper bound from the
moment estimate also depends on the finite-scale prefactor $B_k$. The upper spectral inequality requires additional hypotheses,
as \Cref{prop:incompatibility} shows.

\section{Compatibility Under Radiality: the Exact Warped Case}\label{sec:warped}\label{ssec:warpedrigidity}

This section and \S\ref{sec:sturm} give two proofs of compatibility
under radial hypotheses on the level sets $\Sigma_r$. Neither uses H5
or the exact radial polynomiality assumed in \S\ref{sec:bochner}.

An exact warped product has a reducing radial sector. Under the
regularity imposed below, H4 on this sector implies $\alpha=2$ and
$\lambda_k\asymp\gamma_k$ for every $k$ with $\lambda_k>0$.

We consider complete smooth exact warped products with
one-dimensional base and closed connected fiber $N^{n-1}$:
\begin{equation}\label{eq:wpmetric}
  g=dt^2+\varphi(t)^2g_N.
\end{equation}
The global models are $\mathbb R\times_\varphi N$, with $\varphi>0$,
and smooth one-ended completions of $(0,\infty)\times_\varphi N$
obtained by collapsing the fiber at $t=0$. A completion is included
only when the metric extends smoothly across the added point. In this
one-ended case, smooth collapse to a point requires $N$ to be a round
sphere, up to scaling of $g_N$. All manifolds are without boundary.
In the two-ended case $N$ is an arbitrary closed connected manifold.

The weight $f=f(t)$ is radial. Averaging over $N$ commutes with $L_f$,
so functions independent of the fiber form a reducing spectral
subspace. A nonzero radial eigenfunction cannot vanish identically
on an end, by uniqueness for its radial ODE.

On each end let $r$ be outward arclength. Then $r=t$ on the positive
end and $r=-t$ on the negative end. Write $S_r$ for its warped slice,
and view the restrictions of $f$ and $\varphi$ as functions of $r$.
Unless stated otherwise, primes and tail estimates refer to this
coordinate on a fixed end. All tail hypotheses are required on
every end. The constants may be chosen uniformly over the ends.
In a global product, $|t|\le d_g((t,y),(0,y_0))\le |t|+C$, since
$N$ is compact. In the one-ended completion, $r$ is distance from
the added point. H1 then gives $b=f^{1/\alpha}\asymp r$ on each end.
H2 and properness imply $f'(r)>0$ and
$f'(r)\asymp r^{\alpha-1}$ there, so $b'(r)\asymp1$.
A large level of $b$ is one slice in the one-ended case and the
union of two slices in the two-ended case. H4 is always measured
on the whole level set rather than separately on its components.

The following quantities are used throughout this subsection.
\begin{equation}\label{eq:wpnotation}
\begin{gathered}
  V_r \;:=\; (\log\vol S_r)' \;=\; (n-1)\varphi'/\varphi,
  \qquad
  W \;:=\; \vol(S_r)e^{-f},
  \qquad
  \Theta \;:=\; f'-V_r, \\[2pt]
  \psi \;:=\; \sqrt W\,u,
  \qquad
  Q \;:=\; (\sqrt W)''/\sqrt W.
\end{gathered}
\end{equation}

The substitution $\psi=\sqrt W\,u$ includes the cross-sectional
volume. Its potential $Q$ differs from the full-manifold potential
$V=\tfrac14|\nabla f|^2-\tfrac12\Delta f$ by
\begin{equation}\label{eq:QversusV}
  Q-V=\tfrac14V_r^2+\tfrac12V_r'.
\end{equation}

The radial functions form the reducing sector corresponding to the
constant angular mode. For $\alpha>1$, this sector inherits compact
resolvent from \Cref{lem:discrete}. For $0<\alpha\le1$, discreteness
is included in the radial H4 assumption. In either case the growth
condition is imposed on the radial eigenbasis in its own index.

Radial functions satisfy $L_fu=u''+(V_r-f')u'=\tfrac1W(Wu')'$.

We use the following tail bounds for the radial drift:
\begin{equation}\label{eq:Dregularity}
  \Theta(r)\asymp r^{\alpha-1} \quad\text{(two-sided)}, \qquad
  \Theta'(r)=O(r^{\alpha-2}), \qquad \Theta''(r)=O(r^{\alpha-3}).
\end{equation}

The rigidity argument in this subsection uses the first two bounds.
The third is used later in \S\ref{ssec:countingzeros}, in the
proof of \Cref{thm:osccount}. For $\alpha>1$, the growth bound follows
from H1--H3 and either $f''(r)=O(r^{\alpha-2})$ or
$\Theta'(r)=O(r^{\alpha-2})$, by \Cref{lem:wpanticancel}.
The bounds on $\Theta'$ and $\Theta''$ are additional assumptions.

\begin{theorem}\label{thm:wprigidity}
Let $(M,g,e^{-f})$ be an exact warped product as above, with radial
$f$ satisfying H1--H3. On every end assume
$\Theta'(r)=O(r^{\alpha-2})$.
Suppose the radial eigenfunctions satisfy H4 in their own index.
Then $\alpha=2$ and $\lambda_k\asymp\gamma_k$ for every $k$ with
$\lambda_k>0$. In the one-ended case there is $c_*\in(0,\infty)$,
independent of $k$, such that
\[
  \gamma_k=c_*^{-1}\lambda_k\qquad\text{for every }k.
\]
In either case $\lambda_k\asymp\gamma_k^{(2\alpha-2)/\alpha}$,
since the exponent equals $1$ when $\alpha=2$.
\end{theorem}

\begin{theorem}\label{thm:wpcompat}
Under the hypotheses of \Cref{thm:wprigidity}, suppose there is
$c_1>0$ such that $\Theta_E(r)\sim c_1r$ on every end $E$, where
$r$ is outward arclength. Then, for every $k$,
\[
 \gamma_k=\lambda_k/c_1,\qquad
 \lim_{r\to\infty}r\,w_E(r)=\gamma_k
 \quad\text{on every end }E.
\]
Here $w_E:=u_{k,E}'/u_{k,E}$ for the radial profile $u_{k,E}$.
The quotients are well defined for large $r$ by \Cref{lem:wpnonosc}.
The limits are taken with $k$ fixed. There is also $C'>0$,
independent of $k$ and $E$, such that for every $k$ there is
$R_k<\infty$ with $|r\,w_E(r)-\gamma_k|\le C'$ for $r\ge R_k$
on every end. In the one-ended case $c_1=c_*$, and we write $w=w_E$.
\end{theorem}

In the one-ended case, $r\,w(r)$ is the frequency function $U$
of \Cref{def:freq}, up to a factor comparable to $1$, for a radial
eigenfunction. Here $|\nabla b|=b'(r)$ and the unit normal is
$\partial_r$. Since $u_k$ is radial, it is constant on each level set.
At the level $\rho=b(r)$, we have $I=u_k(r)^2b'(r)$ and
$D=u_k(r)u_k'(r)$. Thus
$U(\rho)=\rho D/I=\bigl(\rho/b'(r)\bigr)\,w(r)$, and
$\rho/b'(r)\asymp r$. Accordingly, \Cref{thm:wpcompat} is a statement
about the frequency function, up to a factor comparable to
$1$.

The additional limit in \Cref{thm:wpcompat} provides pointwise control
of the logarithmic derivative, rather than only a growth order.
It does not follow from \eqref{eq:Dregularity}: for example,
$\Theta(r)=r/(2+\sin\log r)$ satisfies those bounds at $\alpha=2$,
but $r/\Theta(r)$ has no limit.

\subsection{Nonoscillation}

\begin{lemma}\label{lem:wpanticancel}
Let $M$ be a warped product, and let $f$ satisfy H1--H3.
If $f''(r)=O(r^{\alpha-2})$, then $\limsup_{r\to\infty}V_r/f'\le
\eta<\tfrac12$ (the constant $\eta$ of H3), and $V_r$ is bounded
below by a fixed constant for large $r$. After enlarging
the tail,
$\tfrac12(1-\eta)f'\le\Theta=f'-V_r\le f'+C$ for some fixed
constant $C$. If $\alpha>1$, we obtain the two-sided bound
$\Theta(r)\asymp r^{\alpha-1}$ in \eqref{eq:Dregularity}.

If instead $\Theta'(r)=O(r^{\alpha-2})$, then
$\Theta(r)\asymp r^{\alpha-1}$ for $\alpha>1$, and
$\Theta(r)=O(1)$ for $0<\alpha\le1$, without any additional bound on $f''$.
\end{lemma}
\begin{proof}
First assume $f''(r)=O(r^{\alpha-2})$. By definition, $\Delta f=f''+V_rf'$. By H2, $f'\asymp r^{\alpha-1}$.
Using the stated hypothesis $f''=O(r^{\alpha-2})$,
\[
  \Delta f/|\nabla f|^2 \;=\; V_r/f' + O(r^{-\alpha}).
\]
H3's upper bound $\Delta f\le\eta|\nabla f|^2$ then gives
$\limsup V_r/f'\le\eta$. For all sufficiently large $r$,
$V_r/f'\le(1+\eta)/2$, so
$\Theta=f'-V_r\ge\tfrac12(1-\eta)f'$.

For the other direction, H3's lower bound $\Delta f\ge-c_4r^{\alpha-1}$
gives
\[
  V_rf' \;=\; \Delta f-f'' \;\ge\; -c_4r^{\alpha-1}-f''.
\]
Divide by $f'\asymp r^{\alpha-1}$ and use $f''=O(r^{\alpha-2})=o(f')$.
It follows that $V_r\ge-C+o(1)$, where the fixed constant $C$ depends only
on $c_4$, $\alpha$, and the H2 constants.
We also have $\Theta=f'-V_r\le f'+C$. If $\alpha>1$, then $f'\to\infty$
and $\Theta\lesssim f'\asymp r^{\alpha-1}$.

Now assume $\Theta'(r)=O(r^{\alpha-2})$. Put $z=f'>0$ on the tail.
By H1 and H2, $z\asymp r^{\alpha-1}$. Since
$\Delta f=z'+(z-\Theta)z$, H3 yields
\[
 (1-\eta)z+\frac{z'}z\le\Theta
 \le z+\frac{z'}z+C.
\]
Integrate over $[r,r+1]$. The bounds on $z$ give
$\int_r^{r+1}z'/z=\log(z(r+1)/z(r))=O(1)$, and
$\int_r^{r+1}z\asymp r^{\alpha-1}$.
The derivative bound gives
$\int_r^{r+1}\Theta(s)\,ds=\Theta(r)+O(r^{\alpha-2})$.
For fixed $c,C>0$,
\[
 cr^{\alpha-1}-C(1+r^{\alpha-2})\le\Theta(r)
 \le Cr^{\alpha-1}+C(1+r^{\alpha-2}).
\]
For $\alpha>1$ the error on the left is $o(r^{\alpha-1})$.
For $0<\alpha\le1$ both bounds are finite uniformly in $r$.
This proves the remaining bounds.
\end{proof}

\begin{lemma}\label{lem:wpalphagt1}
Under the hypotheses of \Cref{thm:wprigidity}, the radial H4 assumption
implies $\alpha>1$.
\end{lemma}
\begin{proof}
Suppose $0<\alpha\le1$. The estimate under the second hypothesis
of \Cref{lem:wpanticancel} gives $\Theta=O(1)$ on the tail.
The assumed bound
$\Theta'(r)=O(r^{\alpha-2})$ gives
\[
  Q(r)=\frac14\Theta(r)^2-\frac12\Theta'(r)=O(1).
\]
The Liouville map $u\mapsto\psi=\sqrt W\,u$ identifies the radial
operator, outside a compact set, with the one-dimensional Schr\"odinger
operator $-d^2/dr^2+Q(r)$ on a half-line.  Such an operator cannot have
compact resolvent when $Q$ is bounded above on the tail.  Indeed, choose
a fixed $\zeta\in C_c^\infty((0,1))$ with $\|\zeta\|_2=1$ and translate
it to pairwise disjoint intervals tending to infinity.  The translates
are orthonormal in $L^2$, but their quadratic-form energies
$\int(|\zeta'|^2+Q\zeta^2)$ remain uniformly bounded.  The unit
ball of the form domain is not relatively compact in $L^2$, which is
equivalent to failure of compact resolvent.

This contradicts the radial version of H4, which assumes a complete
discrete eigenbasis with eigenvalues tending to $+\infty$. Hence
$\alpha>1$.
\end{proof}

\begin{lemma}\label{lem:wpQ}
Assume $\alpha>1$ and the first two conditions in \eqref{eq:Dregularity}.
Then there are
$c,C,R_2>0$ so that
\[
  cr^{2\alpha-2} \;\le\; Q(r) \;\le\; Cr^{2\alpha-2}
  \qquad \text{for all } r\ge R_2,
\]
where $Q=\tfrac14\Theta^2-\tfrac12\Theta'$. For $\lambda>0$ set
$\rho(\lambda):=\max\bigl(R_2,\,(2\lambda/c)^{1/(2\alpha-2)}\bigr)$. Then
\[
  Q(r)\ge2\lambda \quad\text{for all } r\ge\rho(\lambda).
\]
We also have $\rho(\lambda)\asymp\lambda^{1/(2\alpha-2)}$ as
$\lambda\to\infty$.
\end{lemma}
Under H1--H3, this scale is comparable to the Agmon turning radius
$r_{\mathrm{Ag}}(\lambda)$ of Definition~\ref{def:Agmon}.
The two radii can differ because the potentials $Q$ and $V$ differ
by the volume term in \eqref{eq:QversusV}.
\begin{proof}
Write $g:=\sqrt W$, so $(\log g)'=g'/g=-\Theta/2$, that is,
$g'=-\tfrac12\Theta g$. Differentiating again,
\[
  g'' \;=\; -\tfrac12\Theta'g - \tfrac12\Theta g'
     \;=\; -\tfrac12\Theta'g - \tfrac12\Theta\bigl(-\tfrac12\Theta g\bigr)
     \;=\; \Bigl(\tfrac14\Theta^2-\tfrac12\Theta'\Bigr)g,
\]
so $Q=g''/g=\tfrac14\Theta^2-\tfrac12\Theta'$.

The growth bound in \eqref{eq:Dregularity} gives
$\tfrac14\Theta^2\asymp r^{2\alpha-2}$.
The second condition gives $\Theta'=O(r^{\alpha-2})$,
which is $o(r^{2\alpha-2})$ since $\alpha>1$. Hence $\tfrac14\Theta^2$
dominates $\tfrac12|\Theta'|$ for large $r$, giving
$Q(r)\asymp r^{2\alpha-2}$. Explicitly,
\[
  cr^{2\alpha-2}\;\le\; Q(r) \;\le\; Cr^{2\alpha-2}
  \qquad\text{for } r\ge R_2,
\]
for some fixed constants $c,C,R_2$.

For $r\ge\rho(\lambda)$, in particular
$r\ge(2\lambda/c)^{1/(2\alpha-2)}$, we get
$Q(r)\ge cr^{2\alpha-2}\ge c\cdot(2\lambda/c)=2\lambda$. The scaling $\rho(\lambda)\asymp\lambda^{1/(2\alpha-2)}$ follows
from the definition.
\end{proof}

\begin{lemma}\label{lem:wpnonosc}
Fix an end $E$ and a radial function $u$, and let $\psi:=\sqrt W\,u$ there.
Then $u$ solves
$L_fu=-\lambda u$ on $E$ if and only if
$-\psi''+Q\psi=\lambda\psi$, and
$u\in L^2(E,e^{-f}dV)$ if and only if $\psi\in L^2(dr)$ on that end.
For $\lambda>0$, on $\{r\ge\rho(\lambda)\}$, with $\rho(\lambda)$
as in \Cref{lem:wpQ}, $Q-\lambda\ge\tfrac12Q>0$. If $u=u_k\in L^2_f$
and $\lambda=\lambda_k$, then $\psi_k$ has no zero on
$[\rho(\lambda_k),\infty)$, and $\psi_k\psi_k'<0$ there. So $u_k$ has
no zero on $\{r\ge\rho(\lambda_k)\}$, and $w:=u_k'/u_k$ is well
defined there, with $w<\Theta/2$.
\end{lemma}
\begin{proof}
The equivalence of the two equations is the Liouville
transformation of $L_fu=u''+(V_r-f')u'$, using
$(\log\sqrt W)'=-\Theta/2$ and $Q=(\sqrt W)''/\sqrt W$. Since
$\psi^2=Wu^2$ and radial integration uses the density
$\vol(S_r)\,dr$,
$\int\psi^2\,dr=\int u^2W\,dr=\int_Eu^2e^{-f}\,dV$, with all
integrals taken on the fixed end. This is the $L^2$ equivalence. By \Cref{lem:wpQ}, $Q\ge2\lambda$ for
$r\ge\rho(\lambda)$, so $Q-\lambda\ge\tfrac12Q>0$ there.

Write $p:=Q-\lambda_k$, so $\psi_k''=p\psi_k$ with $p>0$ on
$[\rho(\lambda_k),\infty)$. Suppose $\psi_k(r_1)=0$ for some
$r_1\ge\rho(\lambda_k)$. Since $\psi_k$ is not identically zero,
$\psi_k'(r_1)\ne0$. After replacing $\psi_k$ by $-\psi_k$ if necessary,
assume $\psi_k'(r_1)>0$. Then $\psi_k>0$ just to
the right of $r_1$. Wherever $\psi_k>0$ we have
$\psi_k''=p\psi_k>0$, so $\psi_k'$ is increasing there, and
$\psi_k'\ge\psi_k'(r_1)>0$. Thus $\psi_k$ never returns to $0$, and
$\psi_k(r)\ge\psi_k'(r_1)(r-r_1)$ for all $r\ge r_1$. This is not
in $L^2(dr)$, a contradiction. Hence $\psi_k$ has no zero on
$[\rho(\lambda_k),\infty)$, and we may take $\psi_k>0$ there.

Now $\psi_k''=p\psi_k>0$, so $\psi_k$ is convex on
$[\rho(\lambda_k),\infty)$. If $\psi_k'(r_1)\ge0$ for some
$r_1\ge\rho(\lambda_k)$, then $\psi_k'\ge0$ for $r\ge r_1$, so
$\psi_k\ge\psi_k(r_1)>0$ there, again not in $L^2(dr)$. Therefore
$\psi_k'<0$ on $[\rho(\lambda_k),\infty)$. Finally, since
$u_k=\psi_k/\sqrt W$ and $(\log\sqrt W)'=-\Theta/2$,
\[
  w=\frac{u_k'}{u_k}=\frac{\psi_k'}{\psi_k}+\frac\Theta2<\frac\Theta2 .
\]
\end{proof}

\subsection{The far-field asymptotic}

\begin{lemma}\label{lem:wpasympt}
Assume $\alpha>1$ and the first two conditions in \eqref{eq:Dregularity}.
Let $\lambda_k>0$ be a radial eigenvalue, and let $w=u_k'/u_k$ be
as in \Cref{lem:wpnonosc}. Then, for $r\ge\rho(\lambda_k)$,
\begin{equation}\label{eq:wpwasympt}
 w(r)=\frac{\lambda_k}{\Theta(r)}
 +O\!\left(\lambda_k r^{1-2\alpha}\right)
 +O\!\left(\frac{\lambda_k^2}{\Theta(r)^3}\right).
\end{equation}
The implicit constants depend only on the drift bounds and on $\alpha$,
not on $k$. On this region the first error is also $O(r^{-1})$,
with a constant independent of $k$.
\end{lemma}
\begin{proof}
Write $\lambda=\lambda_k$. Since $u_k''-\Theta u_k'+\lambda u_k=0$,
the function $w$ satisfies the Riccati equation
\begin{equation}\label{eq:wpriccati}
  w'=\Theta w-w^2-\lambda=-(w-w_-)(w-w_+),
  \qquad
  w_\pm:=\frac{\Theta\pm D}{2},\quad D:=\sqrt{\Theta^2-4\lambda},
\end{equation}
where $w_\pm$ are the two roots of $w^2-\Theta w+\lambda=0$.

We first check that $D$ is real and comparable to $\Theta$ on
$\{r\ge\rho(\lambda)\}$. There $Q=\tfrac14\Theta^2-\tfrac12\Theta'\ge2\lambda$ by \Cref{lem:wpQ}.
Also $|\Theta'|\le Cr^{\alpha-2}$ and $\Theta^2\ge cr^{2\alpha-2}$.
After enlarging $R_2$ if needed, $|\Theta'|\le\tfrac18\Theta^2$ for $r\ge R_2$.
Hence
$\tfrac5{16}\Theta^2\ge2\lambda$, that is,
\begin{equation}\label{eq:wpThetalam}
  \Theta^2\ge6\lambda,\qquad \frac{4\lambda}{\Theta^2}\le\frac23,\qquad
  \frac{\Theta}{\sqrt3}\le D\le\Theta \qquad\text{on }\{r\ge\rho(\lambda)\}.
\end{equation}

Two facts about $w_-$ follow. First, $w_-=2\lambda/(\Theta+D)$, and
expanding $2/(1+\sqrt{1-x})=1+\tfrac x4+O(x^2)$ at
$x=4\lambda/\Theta^2\le\tfrac23$,
\begin{equation}\label{eq:wpwminus}
  w_-=\frac\lambda\Theta\Bigl(1+\frac{\lambda}{\Theta^2}+O\Bigl(\frac{\lambda^2}{\Theta^4}\Bigr)\Bigr)
  =\frac\lambda\Theta+O\Bigl(\frac{\lambda^2}{\Theta^3}\Bigr).
\end{equation}
Second, differentiating $w_-=2\lambda/(\Theta+D)$ and using
$D'=\Theta\Theta'/D$,
\[
  w_-'=-\frac{2\lambda\Theta'}{D(\Theta+D)},
  \qquad\text{so}\qquad
  |w_-'|\le\frac{2\sqrt3\,\lambda|\Theta'|}{\Theta^2}
\]
by \eqref{eq:wpThetalam}. Set
\begin{equation}\label{eq:wpT}
  \tau(r):=K\lambda\,r^{1-2\alpha},
\end{equation}
where $K$ depends on the constants in $\Theta\asymp r^{\alpha-1}$
and $\Theta'=O(r^{\alpha-2})$. Choose it large enough that
$|w_-'|\le\tau D$ on $\{r\ge\rho(\lambda)\}$. Then $\tau$ is decreasing in $r$, and since
$\lambda\le\tfrac16\Theta(r)^2\le Cr^{2\alpha-2}$ there,
$\tau(r)\le K'r^{-1}$ with $K'$ independent of $k$.

We show $|w-w_-|\le2\tau$ on $\{r\ge\rho(\lambda)\}$. With
\eqref{eq:wpwminus} this gives \eqref{eq:wpwasympt}.

\emph{Upper bound.} Let $h:=w-w_-$. By \eqref{eq:wpriccati} and
$w_+-w_-=D$,
\[
  h'=h(D-h)-w_-' .
\]
By \Cref{lem:wpnonosc}, $w<\Theta/2=w_-+D/2$, so $h<D/2$ on
$\{r\ge\rho(\lambda)\}$. Suppose $h(r_1)>2\tau(r_1)$ for some
$r_1\ge\rho(\lambda)$. Where $2\tau<h<D/2$ we have $D-h>D/2$ and
$|w_-'|\le\tau D$, so
\[
  h'\ge\tfrac12Dh-\tau D=\tfrac12D\,(h-2\tau)>0,
\]
and since $\tau'\le0$, $(h-2\tau)'\ge\tfrac12D\,(h-2\tau)$. Hence
$h-2\tau$ stays positive for $r\ge r_1$ and grows at least like
$\exp\bigl(\tfrac12\int_{r_1}^rD\bigr)$, which is $\exp(cr^\alpha)$
by \eqref{eq:wpThetalam}. But $D/2\asymp r^{\alpha-1}$, so $h$ would
exceed $D/2$ at some finite $r$, contradicting $h<D/2$. Hence
$h\le2\tau$ on $\{r\ge\rho(\lambda)\}$.

\emph{Lower bound.} Let $g:=w_--w$. Then
\[
  g'=g(D+g)+w_-' \;\ge\; gD-\tau D=D\,(g-\tau).
\]
Suppose $g(r_1)>\tau(r_1)$ for some $r_1\ge\rho(\lambda)$. Then
$(g-\tau)'\ge D\,(g-\tau)$ wherever $g>\tau$, so $g-\tau$ stays
positive for $r\ge r_1$ and grows like
$\exp\bigl(\int_{r_1}^rD\bigr)$. In particular $g\to\infty$, and
$w=w_--g$ becomes negative. Once $w<0$, \eqref{eq:wpriccati} gives
$w'=\Theta w-w^2-\lambda\le-w^2$, so $w$ reaches $-\infty$ at a
finite $r$. That is a zero of $u_k$ on $\{r\ge\rho(\lambda)\}$,
contradicting \Cref{lem:wpnonosc}. We conclude that $g\le\tau$.

Together, $|w-w_-|\le2\tau=2K\lambda_k r^{1-2\alpha}$, and hence
\[
  w=w_-+O\!\left(\lambda_k r^{1-2\alpha}\right).
\]
Combining this with \eqref{eq:wpwminus} gives \eqref{eq:wpwasympt}.
The bound $\tau\le K'r^{-1}$ gives the weaker estimate.
\end{proof}

\begin{remark}
The first error in \eqref{eq:wpwasympt} is bounded by $2\tau$.
For each fixed $k$, it is integrable when $\alpha>1$.
The second error comes from the expansion \eqref{eq:wpwminus} of
$w_-$. The displayed bound is integrable when $\alpha>4/3$.
At $r=\rho(\lambda_k)$ it is of order $\Theta(r)$, so it may
remain large near that radius. The constants in \eqref{eq:wpwasympt}
are independent of $k$, but the integrals of the errors may depend
on $k$. \Cref{lem:wpgrowthorder} uses this integrability to determine
$\log|u_k(r)|-\lambda_k\mathcal J(r)$.
\end{remark}

\subsection{The growth order, and rigidity}

Let $u_{k,E}$ be the radial profile on an end $E$. For large $\rho$,
write $r_E(\rho)$ for the solution of $f_E(r)=\rho^\alpha$ and
$v_E(\rho):=\vol(S_{r_E(\rho)})$. Then $r_E(\rho)\asymp\rho$, and
\begin{equation}\label{eq:wpwholeaverage}
 A_k(\rho):=\frac1{\vol(\Sigma_\rho)}\int_{\Sigma_\rho}u_k^2\,d\sigma
 =\frac{\sum_E v_E(\rho)|u_{k,E}(r_E(\rho))|^2}
        {\sum_E v_E(\rho)}.
\end{equation}
The sum has one or two terms, so
\begin{equation}\label{eq:wpminmax}
 \min_E|u_{k,E}(r_E(\rho))|^2\le A_k(\rho)
 \le\max_E|u_{k,E}(r_E(\rho))|^2.
\end{equation}

\begin{lemma}\label{lem:wpgrowthorder}
Assume H1, H2, $\alpha>4/3$, and the hypotheses of \Cref{lem:wpasympt}.
Let $\lambda_k>0$ be a radial eigenvalue. Fix an end $E$, suppress
its subscript on $u_k$ and $\Theta$, and set
$\mathcal J(r):=\int_{R_0}^r ds/\Theta(s)$. Then
\[
  \log|u_k(r)|=\lambda_k\mathcal J(r)+C_{k,E}+o(1)
  \qquad(r\to\infty),
\]
where $C_{k,E}$ is finite and may depend on $k$ and $E$.
Define the growth order on this end by
\[
 \gamma_{k,E}:=\max\left\{0,\limsup_{r\to\infty}
                     \frac{\log|u_k(r)|}{\log r}\right\}.
\]
For $\lambda_k>0$, this order is finite if and only if
$L_E:=\limsup_{r\to\infty}\mathcal J(r)/\log r<\infty$. When $L_E<\infty$,
\begin{equation}\label{eq:wpgammaform}
  \gamma_{k,E}=\lambda_kL_E.
\end{equation}
In the one-ended case $\gamma_{k,E}=\gamma_k$.
\end{lemma}
\begin{proof}
By \Cref{lem:wpasympt}, for $r\ge\rho(\lambda_k)$,
\[
 \left|w(r)-\frac{\lambda_k}{\Theta(r)}\right|
 \le C\lambda_k r^{1-2\alpha}+C\lambda_k^2r^{3-3\alpha}.
\]
Since $\alpha>4/3$, both $1-2\alpha<-1$ and $3-3\alpha<-1$.
Thus, for fixed $k$, the right side is integrable on
$[\rho(\lambda_k),\infty)$. By \Cref{lem:wpnonosc}, $u_k$ has no
zero there, so $\log|u_k|$ is well defined. Set
\[
 F(r):=\log|u_k(r)|-\lambda_k\mathcal J(r).
\]
Then
\[
 F'(r)=w(r)-\frac{\lambda_k}{\Theta(r)}\in L^1([\rho(\lambda_k),\infty)).
\]
Fix $R\ge\rho(\lambda_k)$. For $r\ge R$,
\[
 F(r)=F(R)+\int_R^r F'(s)\,ds.
\]
The integral converges absolutely as $r\to\infty$. Hence
\[
 C_{k,E}:=F(R)+\int_R^\infty F'(s)\,ds
\]
is finite, and
\[
 F(r)-C_{k,E}=-\int_r^\infty F'(s)\,ds=o(1).
\]
This proves the asymptotic formula.

Dividing by $\log r$ gives
\[
 \frac{\log|u_k(r)|}{\log r}
 =\lambda_k\frac{\mathcal J(r)}{\log r}
  +\frac{C_{k,E}+o(1)}{\log r}.
\]
The last term tends to zero. After increasing $R_0$ if necessary,
$\Theta>0$ on $[R_0,\infty)$, so $\mathcal J$ is nonnegative. Therefore
\[
 \limsup_{r\to\infty}\frac{\log|u_k(r)|}{\log r}
 =\lambda_kL_E.
\]
For $L_E<\infty$ this is nonnegative and gives \eqref{eq:wpgammaform}.
If $L_E=\infty$, the same identity gives $\gamma_{k,E}=\infty$.

In the one-ended case, \eqref{eq:wpwholeaverage} gives
\[
 A_k(\rho)=|u_k(r_E(\rho))|^2.
\]
Thus the infimum in H4 is
\[
 \max\left\{0,\limsup_{\rho\to\infty}
 \frac{\log|u_k(r_E(\rho))|}{\log\rho}\right\}.
\]
Indeed, every exponent above this limsup gives the H4 bound for large
$\rho$, while every exponent below it fails along a sequence; the
remaining compact range is absorbed into the constant. By H1 and H2,
$r_E(\rho)\asymp\rho$, so
$\log r_E(\rho)/\log\rho\to1$. Hence the last limsup equals the one
in the definition of $\gamma_{k,E}$, and therefore
$\gamma_{k,E}=\gamma_k$.
\end{proof}

\begin{proposition}\label{prop:wplinear}
Assume H4 in the radial sector and the hypotheses of
\Cref{lem:wpgrowthorder} on every end. In the one-ended case there
is $c_*\in(0,\infty)$, independent of $k$, such that
\[
 \gamma_k=c_*^{-1}\lambda_k\qquad\text{for every }k,
 \qquad c_*^{-1}=\limsup_{r\to\infty}\frac{\mathcal J(r)}{\log r}.
\]
The same equality holds in the two-ended case if
$\mathcal J_E(r)/\log r\to L\in(0,\infty)$ on both ends with the
same $L$. In that case $c_*^{-1}=L$, where
$\mathcal J_E(r):=\int_{R_0}^r ds/\Theta_E(s)$.
\end{proposition}
\begin{proof}
In the one-ended case put
$L:=\limsup_{r\to\infty}\mathcal J(r)/\log r$.
If $L=\infty$, \Cref{lem:wpgrowthorder} gives infinite growth order
for every positive eigenvalue, contrary to H4. If $L=0$, the same
lemma gives $\gamma_k=0$ for every $k$, again contrary to H4.
We have $0<L<\infty$, and \eqref{eq:wpgammaform} proves the claimed
identity with $c_*=1/L$. The constant eigenfunction has
$\gamma_0=\lambda_0=0$.

For two ends with the stated common limit, \Cref{lem:wpgrowthorder}
and $\log r_E(\rho)/\log\rho\to1$ give, for each $k$ with $\lambda_k>0$,
\[
 \lim_{\rho\to\infty}
 \frac{\log|u_{k,E}(r_E(\rho))|}{\log\rho}=\lambda_kL
 \quad\text{on both ends}.
\]
The minimum--maximum bound \eqref{eq:wpminmax} yields
$\log A_k(\rho)/(2\log\rho)\to\lambda_kL$.
By H4's definition, $\gamma_k=\lambda_kL$.
The identity also holds for the constant eigenfunction.
\end{proof}

\begin{proof}[Proof of \Cref{thm:wprigidity}]
By \Cref{lem:wpalphagt1}, the radial H4 hypothesis first implies
$\alpha>1$. Under the drift derivative hypothesis,
\Cref{lem:wpanticancel} gives
$\Theta_E(r)\asymp r^{\alpha-1}$ on every end.
Fix an end and use its outward coordinate $r$.
If $\alpha>2$, then $\Theta\ge cr^{\alpha-1}$. Hence $\mathcal J(r)=\int^rds/\Theta(s)<\infty$
converges to a finite limit as $r\to\infty$, since $\alpha-1>1$. In
particular,
\[
  L := \limsup_{r\to\infty}\mathcal J(r)/\log r = 0.
\]
This number depends only on $\Theta$ and is independent of $k$ and $\lambda_k$.

Since $\alpha>2>4/3$, \Cref{lem:wpgrowthorder} gives
$\gamma_{k,E}=0$ on every end. By the upper bound in
\eqref{eq:wpminmax}, $\gamma_k=0$ for every $k$, contradicting H4.

For $\alpha<2$, the preceding argument does not cover the whole range.
Lemma~\ref{lem:wpgrowthorder} requires $\alpha>4/3$, so it does not apply
when $\alpha\in(1,4/3]$. We use \Cref{lem:wpasympt} instead.

By \Cref{lem:wpasympt},
\[
  w(r) = \frac{\lambda_k}{\Theta(r)} + O(r^{-1}) + O\Bigl(\frac{\lambda_k^2}{\Theta(r)^3}\Bigr)
  \qquad \text{for } r\ge\rho(\lambda_k).
\]
Fix $k$ with $\lambda_k>0$. Since $\lambda_k/\Theta(r)^2\to0$ as $r\to\infty$, there is
$r_1(k)\ge\rho(\lambda_k)$ beyond which the second error term is at
most $\tfrac12\lambda_k/\Theta(r)$. Hence
\[
  w(r) \;\ge\; \tfrac12\frac{\lambda_k}{\Theta(r)} - C'r^{-1}
\]
there, for a fixed constant $C'$. Integrating from $r_1(k)$ to $r$,
\[
  \log |u_k(r)| \;\ge\; \tfrac12\lambda_k\mathcal J(r) - C'\log r - C_k'',
\]
where $C_k''$ is a fixed constant depending on $k$. It absorbs
$\log |u_k(r_1(k))|$ and the $r_1(k)$-dependent shift in $\mathcal J$.

By \eqref{eq:Dregularity}, $\Theta\le Cr^{\alpha-1}$, so
\[
  \mathcal J(r) \;\ge\; c\bigl(r^{2-\alpha}-R_0^{2-\alpha}\bigr)
  \ge c' r^{2-\alpha}
\]
for all sufficiently large $r$. This grows polynomially, since
$2-\alpha>0$. Any positive power of $r$
eventually dominates $\log r$. Hence the lower bound on $\log |u_k(r)|$ grows faster than any multiple
of $\log r$ on every end. Since $r_E(\rho)\asymp\rho$, the minimum
in \eqref{eq:wpminmax} also grows faster than every polynomial.
This would give $\gamma_k=\infty$, contradicting H4.

We must have $\alpha=2$. In the one-ended case, \Cref{prop:wplinear} gives
$\gamma_k=c_*^{-1}\lambda_k$, so $\lambda_k\asymp\gamma_k$.

For two ends, $\Theta_E(r)\asymp r$ provides fixed $a,A>0$ such that
$a\log r-O(1)\le\mathcal J_E(r)\le A\log r+O(1)$ on both ends.
By \Cref{lem:wpgrowthorder}, after increasing a $k$-dependent constant,
\[
 a\lambda_k\log\rho-C_k
 \le\log|u_{k,E}(r_E(\rho))|
 \le A\lambda_k\log\rho+C_k.
\]
Apply \eqref{eq:wpminmax} and the growth-order definition to obtain
\[
 a\lambda_k\le\gamma_k\le A\lambda_k.
\]
Since $\lambda_k\to\infty$, we obtain $\lambda_k\asymp\gamma_k$.
The exponent $(2\alpha-2)/\alpha$ equals $1$ when $\alpha=2$.
\end{proof}

\begin{proof}[Proof of \Cref{thm:wpcompat}]
By \Cref{thm:wprigidity}, $\alpha=2$. Since $\Theta_E(r)\sim c_1r$,
\[
 \frac{\mathcal J_E(r)}{\log r}\longrightarrow\frac1{c_1}
 \quad\text{on every end}.
\]
Proposition~\ref{prop:wplinear} gives $\gamma_k=\lambda_k/c_1$.
It also gives $c_1=c_*$ in the one-ended case.
The conclusions hold for the constant eigenfunction.

Fix $\lambda_k>0$. The sharper estimate \eqref{eq:wpwasympt} gives
\[
 r\,w_E(r)=\frac{\lambda_k r}{\Theta_E(r)}
           +O\!\left(\frac{\lambda_k+\lambda_k^2}{r^2}\right),
\]
with a constant independent of $k$ and $E$.
For fixed $k$ the error tends to zero, and the first term tends to
$\lambda_k/c_1=\gamma_k$. This proves the limits.

For the uniform bound, choose $R_k\ge\max(1,\lambda_k)$ beyond the
nonoscillation radii on all ends. Increase $R_k$ so that
$\lambda_k|r/\Theta_E(r)-1/c_1|\le1$ for $r\ge R_k$ on every end.
Then $(\lambda_k+\lambda_k^2)/r^2\le2$, so
$|r\,w_E(r)-\gamma_k|\le C'$ with $C'$ independent of $k$ and $E$.
\end{proof}

\begin{remark}\label{rem:wpscope}
In the one-ended case, \Cref{prop:wplinear} shows that H4 implies
\[
  0<\limsup_{r\to\infty}\frac{\mathcal J(r)}{\log r}<\infty,
  \qquad \mathcal J(r):=\int_{R_0}^r\frac{ds}{\Theta(s)}.
\]
H3, through $\eta<\tfrac12$, prevents the volume term $V_r$ from
cancelling the leading part of $f'$, as in the proof of
\Cref{thm:wprigidity}. Neither condition alone implies the rigidity
conclusion. Without H3, H4 can hold in full with $\alpha=n-1$ for every
$n\ge4$, \Cref{thm:H4notalpha2}. Both directions of the proof of
\Cref{thm:wprigidity} use \Cref{lem:wpasympt}, which needs only the first
two conditions in \eqref{eq:Dregularity}.

The hypothesis $\Theta(r)\sim c_1r$ of \Cref{thm:wpcompat} is stronger
than the logarithmic-average conclusion of \Cref{prop:wplinear}. For
example, $\Theta(r)=r/(2+\sin(\log r))$ satisfies
\eqref{eq:Dregularity} at $\alpha=2$, but $r/\Theta(r)$ does not converge.
Regular variation alone does not give convergence of the logarithmic
average either. If
\[
  L(r):=2+\sin(\log\log r),\qquad \Theta(r):=\frac{r}{L(r)},
\]
then $L$ is slowly varying, but $\mathcal J(r)/\log r$ oscillates
\cite{BinghamGoldieTeugels1987}.

In the exact warped setting, the level-set contribution to $\Delta f$
is represented by the one-variable coefficient $V_r$, so
$\Delta f=f''+V_rf'$, and radial eigenfunctions satisfy a scalar ODE.
On a general manifold, the corresponding level-set geometry need not
depend only on the radial variable. Additional structure, such as the
transitive symmetry used in \S\ref{sec:sturm}, can restore a
one-dimensional reduction.
\end{remark}

\subsection{H4 without H3}\label{sssec:H4noH3}

The following example shows that H3 cannot be dropped from
\Cref{thm:wprigidity}. It is a one-ended exact warped product.

\begin{theorem}\label{thm:H4notalpha2}
Let $n\ge4$ and set $\alpha:=n-1$. Choose a smooth increasing warping
function $\varphi$ on $[0,\infty)$ such that
\[
  \varphi(r)=r\quad\text{near }0,
  \qquad
  \varphi(r)=e^{r^\alpha/\alpha}\quad\text{for }r\ge2.
\]
Let $M=\mathbb R^n$ with the complete smooth rotationally symmetric
metric
\[
  g=dr^2+\varphi(r)^2g_{\mathbb S^{n-1}}.
\]
Fix $c>0$, and choose a smooth nonnegative radial function $f$ with
\[
  f(r)=r^\alpha+\frac c2r^2\qquad(r\ge2).
\]
Then H1, H2 and H4 hold, but H3 fails. For every eigenfunction of
$L_f$ with eigenvalue $\lambda_k$,
\begin{equation}\label{eq:H4noH3exponent}
  \lim_{\rho\to\infty}
  \frac{\log\Bigl(\frac1{\vol\Sigma_\rho}
  \int_{\Sigma_\rho}u_k^2\,d\sigma\Bigr)}{2\log\rho}
  =\frac{\lambda_k}{c}.
\end{equation}
Hence
\[
  \gamma_k=\lambda_k/c
\]
for every $k$. In particular H4 holds in full and $\gamma_k\to\infty$,
but $\alpha=n-1\ne2$.
\end{theorem}

\begin{proof}
On the tail,
\[
  V_r:=(\log\vol S_r)'=(n-1)\frac{\varphi'}{\varphi}
  =\alpha r^{\alpha-1}.
\]
Thus
\[
  V_r-f'=-cr,
\]
and
\[
  \varphi(r)^{n-1}e^{-f(r)}=e^{-cr^2/2}\qquad(r\ge2).
\]
Since $r$ is the distance from the pole, $f\asymp r^\alpha$ and
$|\nabla f|^2\asymp r^{2\alpha-2}$. These are H1 and H2.

H3 fails. Indeed, on $r\ge2$,
\[
  \Delta f=f''+V_rf'
  =\alpha^2r^{2\alpha-2}+\alpha c\,r^\alpha
   +\alpha(\alpha-1)r^{\alpha-2}+c.
\]
On the other hand,
\[
  |\nabla f|^2
  =\alpha^2r^{2\alpha-2}+2\alpha c\,r^\alpha+c^2r^2.
\]
Hence
\[
  \frac{\Delta f}{|\nabla f|^2}\longrightarrow1,
\]
so no $\eta<\tfrac12$ can satisfy H3.

We prove H4 by separation of variables. Let
$-\Delta_{\mathbb S^{n-1}}\phi_j=\mu_j\phi_j$, with
$0=\mu_0<\mu_1\le\cdots\to\infty$. In the $j$-th spherical harmonic
sector, $u=p(r)\phi_j(\omega)$ satisfies on the tail
\begin{equation}\label{eq:H4noH3sector}
  p''-crp'-\mu_je^{-2r^\alpha/\alpha}p=-\lambda p.
\end{equation}
Let
\[
  m(r):=\varphi(r)^{n-1}e^{-f(r)}.
\]
The sector is self-adjoint in $L^2((0,\infty),m(r)\,dr)$. Under the
unitary substitution $\psi=\sqrt m\,p$, its nonnegative realization
has the one-dimensional Schr\"odinger form
\[
  -\psi''+\left(Q(r)+\frac{\mu_j}{\varphi(r)^2}\right)\psi,
  \qquad
  Q:=\frac{(\sqrt m)''}{\sqrt m}.
\]
Near the pole, $m(r)=r^{n-1}e^{-f(r)}$, up to a smooth positive
factor, so
\[
  Q(r)=\frac{(n-1)(n-3)}{4r^2}+O(1).
\]
Since $n\ge4$, $Q$ is bounded below. Fix $C_Q>0$ with
$Q\ge-C_Q$ on $(0,\infty)$. On the tail,
\[
  Q(r)=\frac{c^2r^2}{4}-\frac c2.
\]
Every sector has compact resolvent.

The sector ground energies tend to infinity. For large $j$, let
$t_j\ge2$ solve
\[
  \mu_je^{-2t_j^\alpha/\alpha}=\frac{c^2t_j^2}{4}.
\]
Then $t_j\to\infty$. Since $\varphi$ is increasing, for $r\le t_j$,
\[
  Q(r)+\frac{\mu_j}{\varphi(r)^2}
  \ge -C_Q+\mu_je^{-2t_j^\alpha/\alpha}
  =\frac{c^2t_j^2}{4}-C_Q.
\]
For $r\ge t_j$,
\[
  Q(r)+\frac{\mu_j}{\varphi(r)^2}
  \ge\frac{c^2t_j^2}{4}-\frac c2.
\]
The bottom of the $j$-th sector tends to infinity. The direct
sum over the spherical harmonic sectors has compact
resolvent. The constants span the kernel, so
$0=\lambda_0<\lambda_1\le\lambda_2\le\cdots\to\infty$ and there is a
complete orthonormal eigenbasis.

It remains to determine the growth order. Fix one sector and an
eigenvalue $\lambda$, and write $\nu:=\lambda/c$. Equation
\eqref{eq:H4noH3sector} is Hermite's equation with the perturbation
$q(r):=\mu_je^{-2r^\alpha/\alpha}$. The unperturbed equation has a
slow solution
\[
  h_1(r)=r^\nu(1+O(r^{-2}))
\]
and a second solution
\[
  h_2(r)\sim r^{-\nu-1}e^{cr^2/2}.
\]
With the convention $\mathcal W(h_1,h_2)=h_1h_2'-h_1'h_2$,
the Wronskian is $W_0e^{cr^2/2}$. Variation of parameters gives
\[
  \widetilde h_1(r)=h_1(r)+\int_r^\infty
  \frac{h_1(r)h_2(s)-h_2(r)h_1(s)}{W_0e^{cs^2/2}}
  q(s)\widetilde h_1(s)\,ds.
\]
In the norm $\|y\|_R:=\sup_{r\ge R}r^{-\nu}|y(r)|$, the integral
operator is bounded by a constant times
\[
  r^{-2}q(r)+\int_r^\infty s^{-1}q(s)\,ds.
\]
This tends uniformly to zero as $R\to\infty$, since $q$ decays faster
than every power. Hence the operator is a contraction for $R$ large,
and
\[
  \widetilde h_1(r)=r^\nu(1+o(1)).
\]
A second solution obtained by reduction of order has the
$e^{cr^2/2}$ factor. It is not square integrable against
$e^{-cr^2/2}dr$. Every sector eigenfunction therefore satisfies
\[
  p(r)=A r^\nu(1+o(1)),\qquad A\ne0.
\]

For large $\rho$, the level set $\Sigma_\rho=\{f=\rho^\alpha\}$ is
the sphere $r=r_\rho$, where
\[
  r_\rho=\rho\bigl(1+O(\rho^{2-\alpha})\bigr).
\]
If $u=p\phi_j$ and $\phi_j$ is $L^2$-normalized, then
\[
  \frac1{\vol\Sigma_\rho}\int_{\Sigma_\rho}u^2\,d\sigma
  =\frac{p(r_\rho)^2}{\vol(\mathbb S^{n-1})}.
\]
For a linear combination inside a repeated full eigenspace, angular
orthogonality replaces the right side by a finite sum of such terms.
All terms have the same exponent $2\lambda/c$, and at least one
leading coefficient is nonzero. This proves
\eqref{eq:H4noH3exponent}. By the definition of the growth order,
$\gamma_k=\lambda_k/c$. Hence H4 holds.
\end{proof}

\begin{remark}\label{rem:H4noH3}
The example depends on a cancellation on the tail, where
\[
  \Theta=f'-V_r=cr.
\]
The spectral growth has effective exponent $2$, although
$f\asymp r^\alpha$ with $\alpha=n-1\ge3$. H3 rules out this
cancellation. The full Schr\"odinger potential
$V=\tfrac14|\nabla f|^2-\tfrac12\Delta f$ satisfies
\[
  V=-\frac{\alpha^2}{4}r^{2\alpha-2}
    +\frac{c^2}{4}r^2+O(r^{\alpha-2}),
\]
so it tends to $-\infty$. Discreteness here comes from the radial
Liouville potential after the cross-sectional volume is included.

When $n=3$ and $\alpha=2$, the same tail gives
$\Delta f/|\nabla f|^2\to2/(2+c)$. H3 then holds exactly when
$c>2$. This is consistent with \Cref{thm:wprigidity}.
\end{remark}

\subsection{The effective exponent of the drift}\label{sssec:beta}

In this subsection, H3 is used in one place. \Cref{lem:wpanticancel}
uses it to show $\Theta\asymp r^{\alpha-1}$. Every step after that uses only $\Theta$, so the drift can be measured on its own.

\begin{definition}\label{def:beta}
Let $(M,g)$ and $f$ be as in this subsection, with $f$ satisfying H1
and H2. We say that $\Theta=f'-V_r$ has effective exponent
$\beta>1$ if, on every end,
\[
  \Theta(r)\asymp r^{\beta-1},\qquad
  \Theta'(r)=O(r^{\beta-2}),
\]
with the same exponent $\beta$.
Since the radial weight $W=\vol(S_r)e^{-f}$ satisfies
$(\log W)'=-\Theta$, the first condition gives
$-\log W\asymp r^\beta$ on each end.
\end{definition}

\begin{remark}\label{rem:betaequalsalpha}
Under H1--H3 and either derivative hypothesis in \Cref{lem:wpanticancel},
$\Theta\asymp r^{\alpha-1}$ when $\alpha>1$. A regular effective
exponent, when it exists, equals $\alpha$. If $0<\alpha\le1$, the same
lemma bounds $\Theta$ from above, which excludes an effective exponent
$\beta>1$. Without H3 the two exponents can differ. In \Cref{thm:H4notalpha2}, $\Theta=cr$, so
$\beta=2$, but $\alpha=n-1$.
\end{remark}

\begin{theorem}\label{thm:betarigidity}
Let $(M,g)$ be an exact warped product and $f$ a radial weight
satisfying H1 and H2. Suppose $\Theta$ has effective exponent
$\beta>1$ on all ends, in the sense of \Cref{def:beta}, and the radial
eigenfunctions satisfy H4 in their own index. Then $\beta=2$, and
$\lambda_k\asymp\gamma_k$ for every $k$ with $\lambda_k>0$, with
constants independent of $k$.
\end{theorem}

\begin{proof}
The ODE estimates in \Cref{lem:wpQ,lem:wpnonosc,lem:wpasympt,lem:wpgrowthorder}
use only the radial equation, square integrability, and the first two
drift bounds in \eqref{eq:Dregularity}. Apply them with $\beta$ in
place of $\alpha$ in the drift bounds. H1 and H2 retain the exponent
of $f$ and give $r_E(\rho)\asymp\rho$. The radial equation is still
$L_fu=\tfrac1W(Wu')'$. \Cref{lem:wpQ} gives
$Q\asymp r^{2\beta-2}$ and $\rho(\lambda)\asymp\lambda^{1/(2\beta-2)}$,
which needs $\beta>1$. \Cref{lem:wpasympt} is unchanged.
\Cref{lem:wpgrowthorder} needs $\beta>4/3$, for the integrability of
$\lambda_k^2/\Theta^3\asymp r^{3-3\beta}$. Apply these estimates on
each end and combine them using \eqref{eq:wpminmax}. If $\beta>2$,
then every $\mathcal J_E$ converges and all end growth orders vanish,
contrary to H4. If $1<\beta<2$, the lower estimate
in the proof of \Cref{thm:wprigidity} gives faster-than-polynomial
growth on every end, again contrary to H4. Hence $\beta=2$.
The final comparison in that proof gives
$a\lambda_k\le\gamma_k\le A\lambda_k$, so
$\lambda_k\asymp\gamma_k$.
\end{proof}

The proof of \Cref{thm:wprigidity} uses the first two drift bounds
with $\beta=\alpha$. Under H1--H3, \Cref{lem:wpanticancel} gives the two-sided bound on $\Theta$
from the assumed bound on $\Theta'$.

\begin{definition}\label{def:betacompatibility}
With $\beta$ as in \Cref{def:beta}, the spectral and growth
filtrations are compatible with effective exponent $\beta$ if
there are constants $0<c\le C$, independent of $k$, with
\[
  c\,\gamma_k^{(2\beta-2)/\beta}\;\le\;\lambda_k\;\le\;C\,\gamma_k^{(2\beta-2)/\beta}
  \qquad\text{for all sufficiently large }k.
\]
\Cref{def:compatibility} is the case $\beta=\alpha$.
\end{definition}

\begin{remark}\label{rem:betacompatibility}
This is the same statement as \Cref{def:compatibility}, written in
terms of $\Theta$ instead of $f$. By \Cref{lem:wpQ} with $\beta$, the
turning radius is $\rho(\lambda_k)\asymp\lambda_k^{1/(2\beta-2)}$. Heuristically, if one replaces an eigenfunction by the model growth
profile $u_k^2\sim r^{2\gamma_k}$ and $-\log W\asymp r^\beta$, then
the corresponding model density peaks at $r\asymp\gamma_k^{1/\beta}$.
Matching this model scale with
$\rho(\lambda_k)\asymp\lambda_k^{1/(2\beta-2)}$ leads to
$\lambda_k\asymp\gamma_k^{(2\beta-2)/\beta}$. This heuristic motivates the definition but does not describe the true concentration
radius. For $\beta=2$ the
exponent is $1$. The example of \Cref{thm:H4notalpha2} is compatible
with effective exponent $2$, since there $\gamma_k=\lambda_k/c$.
\end{remark}

\begin{remark}\label{rem:betapicture}
If the radial drift has a regular effective power-law exponent, then
radial H4 requires that exponent to be $2$, so $-\log W\asymp r^2$.
Under H3 together with the derivative hypothesis of
\Cref{lem:wpanticancel}, one has $\Theta\asymp f'$. Thus the effective
exponent equals the exponent $\alpha$ of $f$, and $\alpha=2$. The
comparison is between the drift scales, not between $W$ and $e^{-f}$
pointwise. \Cref{thm:H4notalpha2} shows what can happen when
H3 fails: the volume term cancels the leading part of $f'$ and leaves an
effective quadratic drift even though $f$ itself has higher-order
growth.
\end{remark}

\subsection{Uniqueness of the weight}

Under the effective-exponent hypothesis, radial H4 constrains the
weight up to its volume term.
Below, $\Lambda(r):=\log\vol(S_r)$ on each end.

\begin{corollary}\label{cor:H4uniqueness}
Let $(M,g)$ be an exact warped product with $\varphi$ bounded below
on every end. Let $f$ and $\tilde f$ be radial weights, each satisfying
H1, H2 and H4 in its own radial index, where full H4 is not assumed.
Suppose each drift, $\Theta_f:=f'-V_r$ and
$\Theta_{\tilde f}:=\tilde f'-V_r$, has an effective exponent in the
sense of \Cref{def:beta}. The exponents for the two weights may differ.
Then, on every end,
\[
  f=\Lambda+\Phi_f,\qquad \tilde f=\Lambda+\Phi_{\tilde f},\qquad
  \Phi_f\asymp r^2,\quad \Phi_{\tilde f}\asymp r^2 .
\]
In particular $f\asymp\tilde f$. If $\Lambda(r)/r^2\to\infty$ on
every end, then $f/\tilde f\to1$ at infinity.
\end{corollary}
\begin{proof}
By \Cref{thm:betarigidity}, both effective exponents equal $2$.
Work on a fixed end. There are constants $0<c\le C$ with
$cr\le\Theta_f\le Cr$ and
$cr\le\Theta_{\tilde f}\le Cr$ for $r\ge R_0$. Since $f'=V_r+\Theta_f$
and $\Lambda'=V_r$, integrating from $R_0$ gives $f=\Lambda+\Phi_f$
with $\Phi_f:=f(R_0)-\Lambda(R_0)+\int_{R_0}^r\Theta_f$, and
\[
  \tfrac c2\,r^2-C'\;\le\;\Phi_f\;\le\;\tfrac C2\,r^2+C' .
\]
The same holds for $\tilde f$. Since $\varphi$ is bounded below on the tail,
$\Lambda\ge-C_0$ there. Put $a:=\Lambda+C_0\ge0$, $p:=\Phi_f-C_0$ and
$q:=\Phi_{\tilde f}-C_0$. For $r$ large, $p$ and $q$ lie in
$[\tfrac c4r^2,\tfrac C2r^2]$. Hence
\[
  \frac f{\tilde f}=\frac{a+p}{a+q}
\]
lies between $\min(1,p/q)$ and $\max(1,p/q)$, and lies in
$[\tfrac c{2C},\tfrac{2C}c]$. This is $f\asymp\tilde f$. If
$a/r^2\to\infty$, the ratio tends to $1$. There are at most two ends,
so the comparisons hold with common constants on their union.
\end{proof}

\begin{remark}\label{rem:H4uniqueness-sharp}
The lower bound on $\varphi$ is used in \Cref{cor:H4uniqueness}.
Without it, the conclusion can fail. For a
concrete smooth example, choose a smooth rotationally symmetric metric
with $\varphi(r)=r$ near $r=0$ and
$\varphi(r)=e^{-r^2/(n-1)}$ for $r\ge2$. The resulting metric is a complete smooth
warped product with one noncompact end and $V_r=-2r$ on that end.
Choose smooth proper radial weights with
\[
  f(r)=r\quad (r\ge2),\qquad \widetilde f(r)=r^2\quad (r\ge2),
\]
patched smoothly near the origin. Then on the tail
\[
  \Theta_f=2r+1,\qquad \Theta_{\widetilde f}=4r.
\]
Both drifts have effective exponent $2$, and their radial
one-dimensional Schr\"odinger potentials are confining. Moreover,
\[
  \mathcal J_f(r)=\frac12\log r+O(1),\qquad
  \mathcal J_{\widetilde f}(r)=\frac14\log r+O(1).
\]
The same one-dimensional Riccati calculation used above, now with the
effective drift exponent $2$, yields
\[
  \gamma_k^f=\frac12\lambda_k^f,\qquad
  \gamma_k^{\widetilde f}=\frac14\lambda_k^{\widetilde f}
\]
for every positive radial eigenvalue. Hence radial H4 holds for both
weights, but $f/\widetilde f\sim1/r\to0$. Hence some lower control on
the volume term is necessary for the comparability conclusion.

The lower bound on $\varphi$ ensures comparability, but it does not by
itself imply $f/\widetilde f\to1$. On Euclidean space, prescribe for
$r\ge2$ the derivative $f'(r)=r(2+\sin\log r)$ and patch $f$ smoothly
near the origin. Let $\widetilde f=r^2$. Since $V_r=(n-1)/r$, both
radial drifts are $\asymp r$ and satisfy the effective-exponent
regularity with $\beta=2$. In this case
\[
  \frac{\mathcal J_f(r)}{\log r}\longrightarrow\frac1{\sqrt3},\qquad
  \frac{\mathcal J_{\widetilde f}(r)}{\log r}\longrightarrow\frac12.
\]
\Cref{lem:wpgrowthorder} then gives the exact identities
\[
  \gamma_k^f=\frac1{\sqrt3}\lambda_k^f,\qquad
  \gamma_k^{\widetilde f}=\frac12\lambda_k^{\widetilde f},
\]
so radial H4 holds for both. On the other hand,
\[
  f(r)=r^2\Bigl[1+\tfrac15\bigl(2\sin\log r-\cos\log r\bigr)\Bigr]+O(1),
\]
so $f/\widetilde f$ oscillates between
$1-\sqrt5/5$ and $1+\sqrt5/5$ asymptotically and has no limit.
\end{remark}

The effective exponent in \Cref{thm:betarigidity} is a hypothesis
on a one-dimensional drift. On a general manifold, neither such a
drift nor its power-law behavior is supplied by H4.
We return to this distinction in \Cref{q:gaussianweight}.

\section{Compatibility Under Radiality: a Sturm--Liouville Approach}\label{sec:sturm}\label{ssec:sturmreduction}

We replace the exact warped-product structure by a transitive
isometric symmetry of the level sets, \Cref{ass:transitive}. This symmetry produces a radial invariant subspace with its own complete
spectral decomposition. In this subspace, compatibility is a
one-dimensional Sturm--Liouville problem. Transverse spectral modes may still exist in the full operator, but they
do not couple to the invariant sector.

The reduction and the growth-order comparison are established below.

\Cref{thm:osccount} counts the zeros of the sector eigenfunctions and
gives $k\asymp\lambda_k^{\alpha/(2\alpha-2)}$.
\Cref{cor:SL-compat} then transfers the rigidity theory of
\S\ref{sec:warped} to the sector and gives $\alpha=2$ and
$\lambda_k\asymp\gamma_k\asymp k$.

\subsection{Hypothesis and reduction}

\begingroup
\renewcommand{\theassumption}{TS1}
\renewcommand{\theHassumption}{TS1}
\begin{assumption}[Transitive isometric symmetry]\label{ass:transitive}
$(M,g)$ admits an isometric action of a compact Lie group $G$. This
action preserves $f$, so it also preserves $b$. It acts transitively on each
regular level set $\Sigma_r$.

The $G$-invariant subspace $L^2_G(M,e^{-f}dV)\subset L^2(M,e^{-f}dV)$
is closed and $L_f$-invariant. This holds since $L_f$ commutes with
the isometric, $f$-preserving action of $G$.

We assume this subspace carries its own complete eigenbasis
$\{v_j\}_{j\ge0}$ of $L_f|_{L^2_G}$, with eigenvalues $\{\mu_j\}$.
We also assume that, for each $v_j$, there is an intrinsic growth
order $\gamma_j$ satisfying H4's requirements within this sector,
that is, finite for each $j$, with $\gamma_j\to\infty$. This is not
concluded from H4 applied to the full spectrum. It is assumed
explicitly here.
\end{assumption}
\endgroup

\begin{lemma}\label{ass:radial}
If \Cref{ass:transitive} holds, then every $v_j$ in the
$G$-invariant eigenbasis is radial, that is, $v_j(x)=p_j(b(x))$ for
a function $p_j$ of one variable. The quantities $\rho:=|\nabla b|^2$ and
$\tau:=\Delta b-\langle\nabla f,\nabla b\rangle$, that is,
$\tau=L_fb$, are also radial.
\end{lemma}
\begin{proof}
Since $v_j\in L^2_G$, $v_j$ is $G$-invariant. That is,
$v_j(g\cdot x)=v_j(x)$ for every $g\in G$ and $x\in M$. Since $G$
acts transitively on each $\Sigma_r$, for any $x,y\in\Sigma_r$ there
is some $g\in G$ with $g\cdot x=y$. Hence
$v_j(y)=v_j(g\cdot x)=v_j(x)$. This means $v_j$ is constant on each
$\Sigma_r$, that is, radial.

The same argument applies to $\rho$ and $\tau$. Both are obtained from the $G$-invariant functions $f,b$ by applying
$\nabla$ and $\Delta$. These operations commute with the isometric
action, since $G$ acts by isometries preserving $f$. Hence $\rho$ and
$\tau$ are themselves $G$-invariant. They are radial by the same
transitivity argument just given. No separate argument is needed,
and no exact-polynomiality hypothesis is required.
\end{proof}

\begin{remark}\label{rem:invariantH4}
We index the invariant sector by $k=0,1,2,\ldots$ in increasing
spectral order and write $\lambda_k,\gamma_k,u_k$ for its eigenvalues,
growth orders, and eigenfunctions. H4 is required for this sequence.
Within a repeated full eigenspace, cancellation can change the growth
order of a linear combination. A growth condition on a chosen
full eigenbasis is not used as a substitute for the invariant-sector
condition in \Cref{ass:transitive}.
\end{remark}

Each $p_k$ solves, for $r\ge R_0$,
\begin{equation}\label{eq:radialSL}
  \rho(r)\,p_k''(r) + \tau(r)\,p_k'(r) + \lambda_k\,p_k(r) \;=\; 0.
\end{equation}
By \Cref{rem:gradb}, $c_0^2\le \rho(r)\le C_0^2$. By
\Cref{lem:deltab,lem:deltabUpper} and H2, $\tau(r)$ is bounded from
both sides in terms of the constants of H1--H3, by a two-sided
polynomial bound of the same type as the one for $\Delta b$.

\begin{lemma}\label{lem:SLform}
Equation~\eqref{eq:radialSL} is equivalent, on $r\ge R_0$, to the
self-adjoint form
\begin{equation}\label{eq:SLform}
  \bigl(\mu(r)\,p_k'(r)\bigr)' \;+\; \lambda_k\,Q_{\mathrm{SL}}(r)\,p_k(r) \;=\; 0,
  \qquad \mu(r):=\exp\Bigl(\int_{R_0}^r\frac{\tau(s)}{\rho(s)}\,ds\Bigr),
  \quad Q_{\mathrm{SL}}(r):=\frac{\mu(r)}{\rho(r)}.
\end{equation}
\end{lemma}
\begin{proof}
Since $\rho(r)\ge c_0^2>0$ (\Cref{rem:gradb}), dividing
\eqref{eq:radialSL} by $\rho(r)$ gives
$p_k''+\frac{\tau}{\rho}p_k'+\frac{\lambda_k}{\rho}p_k=0$.
Differentiate $\mu(r):=\exp\big(\int_{R_0}^r\tau(s)/\rho(s)\,ds\big)$.
The chain rule and the fundamental theorem of calculus give
$\mu'(r)=\mu(r)\cdot\tau(r)/\rho(r)$. Using this, the product rule
gives
\[
  (\mu p_k')' = \mu'p_k'+\mu p_k'' = \mu\frac{\tau}{\rho}p_k'+\mu p_k'',
\]
which is $\mu$ times the first two terms of the divided equation
above. Multiplying the divided equation through by $\mu$ therefore gives
$(\mu p_k')' + \lambda_k\frac{\mu}{\rho}p_k = 0$, which is
\eqref{eq:SLform} with $Q_{\mathrm{SL}}:=\mu/\rho$.
\end{proof}

In \S\ref{ssec:countingzeros} this equation is rewritten in arclength
and put in Schr\"odinger form. The zeros of $p_k$ are then counted
with the Pr\"ufer angle, \Cref{lem:pruferzeros}.

\begin{remark}\label{rem:muQbounds}
Recall $\tau(r):=\Delta b-\langle\nabla f,\nabla b\rangle$.

For $\alpha>1$, \Cref{lem:deltab,lem:deltabUpper} and the bounds on
$\rho$ give
\[
  |\tau(r)|\le Cr^{\alpha-1}.
\]
The term $\langle\nabla f,\nabla b\rangle
=\alpha r^{\alpha-1}\rho(r)$ has this scale.

The same bound follows from a geometric identity. Recall
$W(r):=\vol(\Sigma_r)e^{-f(r)}$. Coarea gives the radial $L^2$ density
$W(r)/\sqrt{\rho(r)}$, while
$|\nabla(p\circ b)|^2=\rho|p'|^2$ gives the radial energy density
$\sqrt{\rho(r)}\,W(r)$. The divergence coefficient in the
one-dimensional equation is $\sqrt{\rho}W$, up to normalization. Since
$\mu(R_0)=1$, there is a constant $c>0$ such that
\[
  \mu(r) \;=\; c\,\sqrt{\rho(r)}\,W(r).
\]
Equivalently,
$Q_{\mathrm{SL}}=\mu/\rho=cW/\sqrt{\rho}$, which is the coarea
$L^2$ density up to the same constant.

Since $b:=f^{1/\alpha}$, we have $f(r)=r^\alpha$ exactly on
$\Sigma_r$. By \Cref{lem:volmoment},
\[
  \vol(\Sigma_r)\lesssim r^{1-\alpha}e^{\eta r^\alpha},
  \qquad \eta<\tfrac12.
\]
The bounds on $\rho$ from H2 and the identity above then give
\[
  \mu(r) \;\lesssim\; r^{1-\alpha}\exp\bigl(-(1-\eta)r^\alpha\bigr).
\]
For $\alpha>1$, the bounds on $\Delta b$ and $\rho$ also give
$|\tau(r)|\le C r^{\alpha-1}$ on the tail. Since
$(\log\mu)'=\tau/\rho$ and $\rho$ is bounded above and below,
integration gives a lower bound $\mu(r)\ge c e^{-Cr^\alpha}$. Hence,
for suitable positive constants $c,C,c',C'$,
\[
  c e^{-Cr^\alpha}\le\mu(r)\le C e^{-c'r^\alpha},
  \qquad
  c e^{-C'r^\alpha}\le Q_{\mathrm{SL}}(r)\le C e^{-c'r^\alpha}.
\]
The constants in the two exponents may differ. These bounds do not give a comparison with
a single fixed profile. They are used only in the
confining range $\alpha>1$.
\end{remark}

\subsection{Oscillation and comparison}

\begin{lemma}\label{lem:sturmosc}
If H1--H3 hold, with $\alpha>1$, and \Cref{ass:transitive} holds,
let $Z_k$ be the number of zeros of $p_k$ in $(R_0,\infty)$. Then
\[
  Z_k = k+O\bigl(1+\sqrt{\lambda_k}\bigr),
\]
with the implicit constant independent of $k$.
\end{lemma}
\begin{proof}
It is convenient to work with the nonnegative operator
$A_G:=-L_f|_{L^2_G}$. Since the compact group $G$ acts by isometries
preserving $f$, averaging over $G$ defines the orthogonal projection
onto $L^2_G$:
\[
  (P_Gu)(x):=\int_G u(g\cdot x)\,dg,
\]
where $dg$ is normalized Haar measure. This averaging commutes with
$L_f$. Hence $L^2_G$ is a reducing subspace and $A_G$ is self-adjoint.
By H1--H3 and $\alpha>1$, \Cref{lem:discrete} implies compact resolvent.

\emph{Step 1. The global one-dimensional quotient.} A regular level
set $\Sigma_r$ is, by \Cref{ass:transitive}, a single $G$-orbit and is
a hypersurface. The $G$-action has cohomogeneity one. The orbit
space $X:=M/G$ is one-dimensional. On its principal part,
let $s$ denote arclength along a normal geodesic. A $G$-invariant
function is represented by a scalar function $v(s)$, and the
quadratic form and norm of $A_G$ have the Sturm--Liouville form
\[
  \mathcal E_G(v,v)=\int_X |v'(s)|^2 W(s)\,ds,
  \qquad
  \|v\|_G^2=\int_X |v(s)|^2W(s)\,ds,
\]
where $W(s)=\vol(G\cdot\gamma(s))e^{-f(\gamma(s))}$ on the principal
part. At a singular orbit this is understood with the natural smooth
self-adjoint endpoint condition. Hence $A_G$ is a scalar
Sturm--Liouville realization on the connected one-dimensional orbit
space. By the standard Wronskian argument, its eigenvalues in the
invariant sector are simple. With the zero-based enumeration of \Cref{rem:invariantH4}, if
$N_G^{<}(\lambda)$ denotes the number of invariant eigenvalues
strictly below $\lambda$, then
\begin{equation}\label{eq:NGatk}
  N_G^{<}(\lambda_k)=k.
\end{equation}

On the tail $\{b\ge R_0\}$, H2 implies $|\nabla b|\ge c_0>0$, so there
are no critical orbits and $b$ is a valid radial coordinate. The
orbit-space arclength and $b$ are related by $ds=db/\sqrt\rho$, with
$c_0^2\le\rho\le C_0^2$. Zeros counted in the $s$-coordinate
are exactly the zeros of $p_k(b)$ counted in $(R_0,\infty)$.

\emph{Step 2. Cutting off the compact core.} Let $s_0$ be the orbit
corresponding to $\Sigma_{R_0}$. Since $b$ is proper, the quotient of
$\{b\le R_0\}$ is compact. It is the compact side $X_-$ of the cut at
$s_0$. The other side $X_+$ is the noncompact tail. Let
$A_\pm^D$ and $A_\pm^N$ denote the scalar Sturm--Liouville operators
on $X_\pm$ with Dirichlet and Neumann conditions at $s_0$, respectively.
At the other endpoint, we retain the natural self-adjoint condition.
Let $N_\pm^B(\lambda)$ count the eigenvalues of $A_\pm^B$ strictly below
$\lambda$, with multiplicity, for $B\in\{D,N\}$.
Dirichlet--Neumann bracketing for the quadratic form gives
\begin{equation}\label{eq:bracket}
  N_-^D(\lambda)+N_+^D(\lambda)
  \le N_G^{<}(\lambda)
  \le N_-^N(\lambda)+N_+^N(\lambda).
\end{equation}

The compact-side operators are one-dimensional Sturm--Liouville
operators on a finite orbit-space interval. The possible endpoint
corresponding to a singular orbit is of the standard regular-singular
type arising from a smooth cohomogeneity-one metric. The classical
finite-interval estimate gives
\begin{equation}\label{eq:corebound}
  N_-^{D/N}(\lambda)=O(1+\sqrt\lambda)
\end{equation}
as $\lambda\to\infty$. See, for example, \cite{Zettl2005}.
Here $D/N$ denotes either boundary condition.

The tail realizations have compact resolvent. After the unitary
conjugation of \Cref{lem:conjugation}, the full operator is $-\Delta+V$
with $V\to+\infty$ by \Cref{lem:Vbounds2}. A boundary condition at the
fixed compact hypersurface $\Sigma_{R_0}$ does not affect confinement
at infinity. Dirichlet
and Neumann eigenvalues for the scalar tail problem interlace, and so
\begin{equation}\label{eq:tailinterlace}
  |N_+^D(\lambda)-N_+^N(\lambda)|\le1.
\end{equation}
Combining \eqref{eq:bracket}--\eqref{eq:tailinterlace} gives
\begin{equation}\label{eq:NGeqNplus}
  N_G^{<}(\lambda)=N_+^D(\lambda)+O(1+\sqrt\lambda).
\end{equation}

\emph{Step 3. Tail zeros.} On $X_+$, use the $b$-coordinate. The
scalar equation is exactly \eqref{eq:SLform}. Let $q_\lambda$ be its
solution satisfying the Dirichlet condition $q_\lambda(R_0)=0$.
The Sturm oscillation theorem for the half-line Dirichlet problem
identifies the number of zeros of $q_\lambda$ in $(R_0,\infty)$ with
$N_+^D(\lambda)$, up to the choice of endpoint convention. See
\cite{Titchmarsh1962}. At $\lambda=\lambda_k$, $p_k$ solves the same
second-order equation. The Dirichlet counting function is finite, so Sturm oscillation gives
only finitely many zeros for $q_{\lambda_k}$. If $p_k$ and
$q_{\lambda_k}$ are linearly
independent, Sturm separation makes their zeros interlace. If they are
proportional, the zero counts agree. Hence
\begin{equation}\label{eq:ZkeqNplus}
  Z_k=N_+^D(\lambda_k)+O(1).
\end{equation}
Finally, \eqref{eq:NGatk}, \eqref{eq:NGeqNplus}, and
\eqref{eq:ZkeqNplus} give
\[
  Z_k=k+O(1+\sqrt{\lambda_k}),
\]
as claimed.
\end{proof}

\subsection{Counting zeros}\label{ssec:countingzeros}

The zeros of a solution of $\Psi''+q\Psi=0$ are counted by the
Pr\"ufer angle. Writing a nontrivial solution in Pr\"ufer variables
turns its zeros into crossings of integer multiples of $\pi$ by a phase
angle. When $q>0$, the angle advances at leading order like $\sqrt q$.
The next lemma contains three estimates.

\begin{lemma}\label{lem:pruferzeros}
Let $q$ be continuous on $[a,b]$, and let $\Psi\not\equiv0$ solve
$\Psi''+q\Psi=0$ there. Let $Z$ be the number of zeros of $\Psi$ in
$(a,b]$.
\begin{enumerate}[label=\rm(\alph*)]
\item If $q\le0$ on $[a,b]$, then $Z\le1$.
\item If $q\le\Lambda$ on $[a,b]$ for a constant $\Lambda>0$, then
  \[
    Z\;\le\;\frac{\sqrt\Lambda}{\pi}\,(b-a)+1.
  \]
\item If $q$ is absolutely continuous and $q\ge m$ on $[a,b]$ for a
  constant $m>0$, then
  \[
    \Bigl|\,Z-\frac1\pi\int_a^b\sqrt{q(s)}\,ds\,\Bigr|
    \;\le\;1+\frac1{4\pi m}\int_a^b|q'(s)|\,ds.
  \]
\end{enumerate}
\end{lemma}
\begin{proof}
Zeros of $\Psi$ are isolated, since $\Psi(s_0)=\Psi'(s_0)=0$ would give $\Psi\equiv0$.

(a) Suppose $\Psi$ has two zeros in $(a,b]$. Take consecutive ones,
$s_1<s_2$, and replace $\Psi$ by $-\Psi$ if needed, so that $\Psi>0$
on $(s_1,s_2)$. Then $\Psi'(s_1)\ge0$ and $\Psi'(s_2)\le0$. But
$\Psi''=-q\Psi\ge0$ on $(s_1,s_2)$, so $\Psi'$ is nondecreasing
there. This gives $0\le\Psi'(s_1)\le\Psi'(s_2)\le0$, so $\Psi'\equiv0$
on $[s_1,s_2]$. Then $\Psi\equiv0$ on $[s_1,s_2]$, and so
$\Psi\equiv0$. This contradicts $\Psi\not\equiv0$.

(b) Write $\Psi=\varrho\sin\phi$ and $\Psi'=\sqrt\Lambda\,\varrho\cos\phi$,
with $\varrho:=(\Psi^2+\Psi'^2/\Lambda)^{1/2}>0$ and $\phi$ continuous.
Differentiating both relations and eliminating $\varrho'$ gives
\[
  \phi'=\sqrt\Lambda\cos^2\phi+\frac{q}{\sqrt\Lambda}\sin^2\phi
  \;\le\;\sqrt\Lambda,
\]
using $q\le\Lambda$. The zeros of $\Psi$ are the points where
$\phi\in\pi\mathbb Z$. At such a point $\phi'=\sqrt\Lambda>0$, so
$\phi$ crosses each multiple of $\pi$ upward and never downward.
Once $\phi$ has passed a level $m\pi$ it stays above it. Thus $Z$ is
the number of integers $m$ with $\phi(a)<m\pi\le\phi(b)$, and
\[
  \Bigl|\,Z-\frac{\phi(b)-\phi(a)}{\pi}\,\Bigr|\le1 .
\]
Since $\phi(b)-\phi(a)=\int_a^b\phi'\le\sqrt\Lambda\,(b-a)$, the
claim follows.

(c) Let $\omega:=\sqrt q$, so $\omega\ge\sqrt m>0$. Write
$\Psi=\varrho\sin\phi$ and $\Psi'=\omega\varrho\cos\phi$, with
$\varrho:=(\Psi^2+\Psi'^2/\omega^2)^{1/2}>0$. Differentiating both
relations and eliminating $\varrho'$ gives, almost everywhere,
\[
  \phi'=\omega-\frac{\omega'}{2\omega}\sin2\phi,
  \qquad\text{so}\qquad
  |\phi'-\omega|\le\frac{|\omega'|}{2\omega}=\frac{|q'|}{4q}\le\frac{|q'|}{4m}.
\]
As in (b), the zeros of $\Psi$ are upward crossings of $\pi\mathbb Z$
by $\phi$, since $\phi'=\omega>0$ at such points. Hence
$|Z-(\phi(b)-\phi(a))/\pi|\le1$. Integrating the bound on
$\phi'-\omega$ over $[a,b]$ proves the claim.
\end{proof}

\subsection{Main theorem}

\begin{theorem}\label{thm:osccount}
Suppose H1--H4 with $\alpha>1$ and \Cref{ass:transitive} hold.
Let $s$ be arclength
along the radial direction, $ds=dr/\sqrt{\rho}$, with $s=0$ at
$r=R_0$. Set $W(s):=\vol(\Sigma_s)\,e^{-f}$ and
$\Theta(s):=-\dfrac{d}{ds}\log W(s)$. Assume $\Theta$ satisfies
\eqref{eq:Dregularity}, with $s$ in place of $r$. If $Z_k$ denotes the
number of zeros of $p_k$ in $(R_0,\infty)$, then
\[
  Z_k \;\asymp\; k \;\asymp\; r_{\mathrm{Ag}}(k)^\alpha
  \;\asymp\; \lambda_k^{\alpha/(2\alpha-2)}.
\]
\end{theorem}
\begin{proof}
Write $p:=p_k$ and $\lambda:=\lambda_k$.

\emph{Step 1. The radial problem in arclength.} By
\eqref{eq:radialSL}, $\rho\,p_{rr}+\tau\,p_r+\lambda p=0$. Since
$p_r=p_s/\sqrt\rho$ and $p_{rr}=p_{ss}/\rho-\rho'p_s/(2\rho^{3/2})$,
this reads
\[
  p_{ss}+\frac{\tau-\rho'/2}{\sqrt\rho}\,p_s+\lambda p=0 .
\]
By \Cref{rem:muQbounds}, $\mu=c\sqrt\rho\,W$, and $\mu'/\mu=\tau/\rho$.
Hence $(\log W)_r=\tau/\rho-\rho'/(2\rho)$, and the coefficient of $p_s$
above is $\sqrt\rho\,(\log W)_r=(\log W)_s=-\Theta(s)$. Thus
\[
  p_{ss}-\Theta(s)\,p_s+\lambda p=0,
  \qquad\text{that is,}\qquad
  \tfrac1W\,(W p_s)_s=-\lambda p .
\]
This is the radial operator of \S\ref{ssec:warpedrigidity}, with no
error term. Here \Cref{ass:transitive} is used. The level sets are
orbits of $G$, so $\rho$, $\tau$ and $\vol(\Sigma_r)$ are functions of
$r$ alone, and the identities above hold pointwise.

Set $\Psi:=\sqrt W\,p$. A direct computation gives
\[
  -\Psi_{ss}+Q_s\,\Psi=\lambda\Psi,
  \qquad Q_s:=\tfrac14\Theta^2-\tfrac12\Theta'.
\]
Since $W>0$, the zeros of $\Psi$ are those of $p$. By
\Cref{lem:sturmosc}, writing $Z_k$ for their number,
$Z_k=k+O(1+\sqrt\lambda)$. Write $q:=\lambda-Q_s$, so
$\Psi''+q\Psi=0$.

\emph{Step 2. The potential.} By \eqref{eq:Dregularity} in the
variable $s$, $\Theta(s)\asymp s^{\alpha-1}$, $\Theta'(s)=O(s^{\alpha-2})$
and $\Theta''(s)=O(s^{\alpha-3})$. Since $\alpha>1$, the term
$\tfrac14\Theta(s)^2\asymp s^{2\alpha-2}$ dominates
$\tfrac12\Theta'(s)=O(s^{\alpha-2})$ for large $s$. Hence there are
$s_1\ge1$ and constants $0<c_1\le C_1$ and $C_2$ with
\[
  c_1s^{2\alpha-2}\le Q_s(s)\le C_1s^{2\alpha-2},
  \qquad
  |Q_s'(s)|=\bigl|\tfrac12\Theta(s)\Theta'(s)-\tfrac12\Theta''(s)\bigr|
  \le C_2s^{2\alpha-3},
  \qquad s\ge s_1 .
\]

\emph{Step 3. Three regions.} Set
\[
  s_-:=\Bigl(\frac{\lambda}{2C_1}\Bigr)^{1/(2\alpha-2)},
  \qquad
  s_+:=\Bigl(\frac{2\lambda}{c_1}\Bigr)^{1/(2\alpha-2)}.
\]
Both are $\asymp\lambda^{1/(2\alpha-2)}$, and $s_1<s_-<s_+$ once $k$
is large. By Step 2, $0\le Q_s\le\lambda/2$ on $[s_1,s_-]$, so
$\lambda/2\le q\le\lambda$ there. Also $Q_s\ge2\lambda$ on
$[s_+,\infty)$, so $q\le-\lambda<0$ there.

\emph{Step 4. The bulk.} Apply \Cref{lem:pruferzeros}(c) on
$[s_1,s_-]$ with $m=\lambda/2$. There $|q'|=|Q_s'|\le C_2s^{2\alpha-3}$,
so
\[
  \int_{s_1}^{s_-}|q'|\,ds
  \;\le\;C_2\,\frac{s_-^{2\alpha-2}}{2\alpha-2}
  \;=\;\frac{C_2}{2\alpha-2}\cdot\frac{\lambda}{2C_1}.
\]
Let $Z_0$ be the number of zeros of $\Psi$ in $(s_1,s_-]$. Then
\[
  \Bigl|\,Z_0-\frac1\pi\int_{s_1}^{s_-}\sqrt{\lambda-Q_s}\,ds\,\Bigr|
  \;\le\;1+\frac{C_2}{8\pi(\alpha-1)\,C_1}=:C_3,
\]
a constant independent of $k$. Since $\lambda/2\le\lambda-Q_s\le\lambda$
on the bulk,
\[
  \frac{\sqrt\lambda}{\pi\sqrt2}\,(s_--s_1)-C_3
  \;\le\;Z_0\;\le\;
  \frac{\sqrt\lambda}{\pi}\,(s_--s_1)+C_3 .
\]

\emph{Step 5. Outside the bulk.} Put $C_4:=\max_{[0,s_1]}|Q_s|$.
On $[0,s_1]$ we have $q\le\lambda+C_4$. By \Cref{lem:pruferzeros}(b),
$\Psi$ has at most $\frac{\sqrt{\lambda+C_4}}{\pi}s_1+1$ zeros in $(0,s_1]$.
On $[s_-,s_+]$, $q\le\lambda$, so by \Cref{lem:pruferzeros}(b) there
are at most $\frac{\sqrt\lambda}{\pi}(s_+-s_-)+1$ zeros in
$(s_-,s_+]$. On $[s_+,T]$ for every $T>s_+$, $q<0$, so by
\Cref{lem:pruferzeros}(a) there is at most one zero in $(s_+,\infty)$.

\emph{Step 6. Combining the estimates.} The four intervals cover $(0,\infty)$.
Steps 4 and 5 give
\[
  \frac{\sqrt\lambda}{\pi\sqrt2}\,(s_--s_1)-C_3
  \;\le\;Z_k\;\le\;
  \frac{\sqrt\lambda}{\pi}\,s_+ + \frac{\sqrt{\lambda+C_4}}{\pi}\,s_1 + C_3+3 .
\]
Here $s_\pm\asymp\lambda^{1/(2\alpha-2)}\to\infty$. The quantities
$s_1$, $C_3$ and $C_4$ are fixed. Both sides are
$\asymp\sqrt\lambda\cdot\lambda^{1/(2\alpha-2)}=\lambda^{\alpha/(2\alpha-2)}$,
and $Z_k\asymp\lambda_k^{\alpha/(2\alpha-2)}$.

By \Cref{lem:sturmosc}, $k=Z_k+O(1+\sqrt{\lambda_k})$. Since $\alpha>1$,
\[
  \frac{\sqrt{\lambda_k}}{\lambda_k^{\alpha/(2\alpha-2)}}
  \;=\;\lambda_k^{-1/(2\alpha-2)}\;\longrightarrow\;0,
\]
so $\sqrt{\lambda_k}=o\bigl(\lambda_k^{\alpha/(2\alpha-2)}\bigr)$, and the
correction from $Z_k$ to $k$ does not change the order. Hence
$k\asymp\lambda_k^{\alpha/(2\alpha-2)}$.

Finally, by \Cref{def:Agmon} and \Cref{lem:Vbounds2},
$\lambda_k\asymp r_{\mathrm{Ag}}(k)^{2\alpha-2}$. Thus
$\lambda_k^{\alpha/(2\alpha-2)}\asymp r_{\mathrm{Ag}}(k)^\alpha$.
\end{proof}

\begin{corollary}\label{cor:SL-compat}
Under the hypotheses of \Cref{thm:osccount}, $\alpha=2$, and there
is $c_*\in(0,\infty)$, independent of $k$, with
\[
  \gamma_k=c_*^{-1}\lambda_k\qquad\text{for every }k.
\]
Thus $\lambda_k\asymp\gamma_k\asymp k$, giving full
quantitative compatibility for the $G$-invariant sector.
\end{corollary}
\begin{proof}
By Step~1 of the proof of \Cref{thm:osccount}, in the arclength
variable $s$ the invariant sector satisfies
$p_{ss}-\Theta(s)p_s+\lambda p=0$, that is,
$\tfrac1W(Wp_s)_s=-\lambda p$ with $W=\vol(\Sigma_s)e^{-f}$, with no
error term. This is the radial equation of
\S\ref{ssec:warpedrigidity} with $s$ in place of $r$. We check that
the proof of \Cref{thm:betarigidity} applies to it.

That proof uses \Cref{lem:wpQ,lem:wpnonosc,lem:wpasympt,lem:wpgrowthorder},
\Cref{prop:wplinear}, and the proof of \Cref{thm:wprigidity}. These
use the warped product structure in three places only. The first is
the radial equation, which is the identity above. The second is the
coarea identity $\int_0^\infty u^2W\,ds=\int_{\{b\ge R_0\}}u^2e^{-f}\,dV$ for radial $u$,
used in \Cref{lem:wpnonosc}. It holds here since $|\nabla s|=1$,
and $u$, $f$ and $\vol(\Sigma_s)$ are functions of $s$ alone. The
third is the comparability $s\asymp b\asymp r(x)$, used in
\Cref{lem:wpgrowthorder} to measure growth orders in $s$. It holds
since $ds=dr/\sqrt\rho$ with $c_0^2\le\rho\le C_0^2$, and
$b\asymp r(x)$ by H1. The other hypotheses used are
\eqref{eq:Dregularity} for $\Theta$, which is assumed, and H4 for
the sector, which is part of \Cref{ass:transitive}. Hence
\Cref{thm:betarigidity} applies with $\beta=\alpha$ and gives
$\alpha=2$. Each invariant eigenfunction is constant on the whole
level set, so its growth order is that of this single profile.
The single-profile argument in \Cref{prop:wplinear}
gives $\gamma_k=c_*^{-1}\lambda_k$. Since $\gamma_k\to\infty$, this
gives $\lambda_k\asymp\gamma_k$. Finally, at $\alpha=2$
\Cref{thm:osccount} reads $k\asymp\lambda_k$.
\end{proof}

\begin{remark}\label{rem:classicaltime}
For the one-dimensional Hamiltonian
$|\xi|^2+cr^{2\alpha-2}=E$, the turning point is
$r_E=(E/c)^{1/(2\alpha-2)}$. Rescaling $r=r_Ex$ produces the action
scale $r_E\sqrt E\asymp E^{\alpha/(2\alpha-2)}$, consistent with
\Cref{thm:osccount}. This calculation concerns the radial model,
not localization of general eigenfunctions.
\end{remark}

Theorem~\ref{thm:osccount} relates the zero count $Z_k$, the spectral index $k$,
and $\lambda_k$, while \Cref{cor:SL-compat} relates $\lambda_k$
to the growth order $\gamma_k$. Together they give
$\lambda_k\asymp\gamma_k\asymp k$ within this radial class. What is
not settled is whether the regularity hypothesis on the radial drift
can be removed.

\begin{openquestion}\label{prob:oscgrowth}
Under \Cref{ass:transitive} alone, without \eqref{eq:Dregularity}
for $\Theta$, does $\gamma_k\asymp k$ hold? With
\eqref{eq:Dregularity} it does, by \Cref{cor:SL-compat}.
\end{openquestion}

\begin{remark}\label{rem:sturmscope}
Theorem~\ref{thm:osccount} is a statement about the $G$-invariant
sector alone. The theorem allows nontrivial transverse modes. Euclidean space with
its rotation-invariant sector is an example. The restrictions are the
transitive level-set symmetry, the sector version of H4 in
\Cref{ass:transitive}, and the radial-drift regularity used in
\Cref{thm:osccount}.
\end{remark}

\section{\texorpdfstring{From compatibility to $\alpha=2$}{From compatibility to alpha=2}}\label{sec:alphatwo}

Suppose compatibility holds in the sense of \Cref{def:compatibility}.
If $\lambda_k\asymp k^a$ and $\gamma_k\asymp k^b$, then
$a/b=(2\alpha-2)/\alpha$. Linear growth of both
sequences implies $\alpha=2$. For this conclusion, it suffices to have
linear growth along a common subsequence, indexed by its own position.

\begingroup
\renewcommand{\theassumption}{H6}
\renewcommand{\theHassumption}{H6}
\begin{assumption}[Linear growth along a common subsequence]\label{ass:linearcorr}
There are indices $k_1<k_2<\cdots$ and constants $c,C>0$ such that,
for all sufficiently large $j$,
\[
  c\,j \;\le\; \gamma_{k_j} \;\le\; C\,j, \qquad
  c\,j \;\le\; \lambda_{k_j} \;\le\; C\,j.
\]
Here $j$ indexes the subsequence, not the full spectrum.
The upper bound on $\gamma_{k_j}$ gives $\dim\mathcal P_N\gtrsim N$,
but no matching upper bound. If $\gamma_k\asymp k$ holds for the
entire sequence, then $\dim\mathcal P_N\asymp N$.
The converse can fail because the spectral index $k$ does not
necessarily order the growth numbers $\gamma_k$ monotonically.
Recall that $\mathcal P_N$ is finite-dimensional,
by \Cref{prop:PN_structure}(i).
\end{assumption}
\endgroup
\addtocounter{assumption}{-1}

\begin{remark}
For the Gaussian operator with $f(x)=|x|^2/4$ on $\R^n$, a
nonzero eigenfunction of total degree $m$ has $\lambda=m/2$ and
$\gamma=m$. There are $\binom{m+n}{n}$ basis elements of degree at
most $m$. The full spectrum, counted with multiplicity, satisfies
$\lambda_k\asymp\gamma_k\asymp k^{1/n}$, and compatibility holds.
Choosing one eigenfunction from each degree $j$ produces
$\lambda_{k_j}=j/2$ and $\gamma_{k_j}=j$.
H6 holds in every dimension, although the full sequences are
not linear in $k$ when $n\ge2$. The radial sector also satisfies H6.
\end{remark}

\begin{corollary}\label{cor:alpha2conditional}
Suppose H1--H3 hold with $\alpha>1$, and \Cref{ass:linearcorr}
holds with indices $k_j$. If
$\lambda_{k_j}\asymp\gamma_{k_j}^{(2\alpha-2)/\alpha}$,
then $\alpha=2$. This holds in particular under compatibility
(\Cref{def:compatibility}). The same conclusion holds if
\[
  \gamma_{k_j}^{(2\alpha-2)/\alpha}\lesssim\lambda_{k_j}
  \lesssim(\gamma_{k_j}\log\gamma_{k_j})^{(2\alpha-2)/\alpha}
\]
for all sufficiently large $j$.
\end{corollary}
\begin{proof}
The first comparison and \Cref{ass:linearcorr} give
\[
  \lambda_{k_j} \;\asymp\; \gamma_{k_j}^{(2\alpha-2)/\alpha}
  \;\asymp\; j^{(2\alpha-2)/\alpha},
\]
and also $\lambda_{k_j}\asymp j$. Comparing exponents,
$(2\alpha-2)/\alpha=1$, that is, $\alpha=2$.

With a logarithmic factor, the two-sided bound becomes
\[
  j^{(2\alpha-2)/\alpha} \;\lesssim\; j \;\lesssim\;
  (j\log j)^{(2\alpha-2)/\alpha}.
\]
Now $\log j$ grows slower than any positive power of $j$. So
comparing exponents of $j$ on each side still gives
$(2\alpha-2)/\alpha=1$.
\end{proof}

\begin{remark}\label{rem:radialcompare}
The conditional conclusion above uses H6 and the compatibility
comparison on its subsequence.
Under transitive symmetry and the drift regularity in
\eqref{eq:Dregularity}, \Cref{cor:SL-compat} proves $\alpha=2$ and
linear indexing within the invariant sector without assuming either
comparison separately.
\end{remark}

\begin{remark}\label{rem:indexingscope}
H6 concerns a common subsequence indexed by $j$. It does not require
linear growth in the original spectral index $k_j$.
It is not an assumption on eigenvalue multiplicities. Exact radial
polynomiality in \S\ref{sec:bochner} provides a different hypothesis
from which $\alpha=2$ follows.
\end{remark}

\begin{openquestion}\label{prob:compatnotlinear}
Does there exist a weighted manifold satisfying H1--H4 with
$\alpha\ne2$ for which the full spectral and growth filtrations are
compatible? Any such example must fail the subsequence condition
in \Cref{ass:linearcorr}.
\end{openquestion}

\section{An inverse Hermite/Laguerre theorem}\label{sec:bochner}

We assume exact polynomiality of the invariant eigenfunctions.
The proof uses a degree comparison for the reduced operator, related
to Bochner's classification \cite{Bochner1929}, rather than a
localization estimate.

\subsection{Hypothesis}

\begin{definition}\label{def:exactpoly}
Under the transitive isometric symmetry of
Assumption~\ref{ass:transitive} (\S\ref{sec:sturm}), let
$\{u_k\}_{k\ge0}$ be the eigenbasis of the $G$-invariant sector,
re-indexed by its own spectral enumeration as in
Remark~\ref{rem:invariantH4}. We impose the condition below only on
this sector, not on the full eigenbasis.

The family satisfies exact radial polynomiality of type
$\beta\in\{1,2\}$ if, for every $k$ and every $r\ge R_0$, writing
$s:=r^\beta$,
\[
  u_k(r) \;=\; q_k(s)
\]
for a real polynomial $q_k$ of degree $k$.

We call $\beta=1$ the Hermite-type case. There, $u_k$ is a degree-$k$ polynomial in $r$. We call $\beta=2$ the
Laguerre-type case. There, $u_k$ is a degree-$k$ polynomial
in $r^2$.
\end{definition}

\begin{remark}\label{rem:betachoice}
For a power substitution $s=r^\beta$, the coefficient of $q''$ in the
transformed equation is
\[
  \beta^2r^{2\beta-2}\rho(r).
\]
Since H2 bounds $\rho$ above and below, this coefficient has size
$s^{2-2/\beta}$. \Cref{lem:gequals1} shows that exact polynomiality
forces it to be a polynomial of degree at most two. Its degree must
therefore equal $2-2/\beta$. For finite $\beta>0$, the only
nonnegative integer possibilities are $0$ and $1$, giving $\beta=1$
and $\beta=2$. These are the Hermite and Laguerre cases. The Jacobi
case, whose leading coefficient is quadratic, cannot arise from a
finite power substitution under H2.

We consider only these two power substitutions here. The case
$\beta=2$ is realized by the examples of \S\ref{sec:examples}. No
example of the case $\beta=1$ is known to us,
\Cref{prob:hermiterealize}.
\end{remark}

This hypothesis adds a condition to Assumption~\ref{ass:transitive}.
Transitive symmetry already implies radiality. H4 in the invariant sector
ensures a finite growth order for each $k$, tending to infinity.
Exact polynomiality requires the degree $\beta k$ to come from a
polynomial identity in $r^\beta$ rather than merely from asymptotic growth.

\subsection{Radiality of the coefficient functions}

Consider the functions
\[
  \rho(x) := |\nabla b(x)|^2, \qquad
  \tau(x) := \Delta b(x) - \langle\nabla f(x),\nabla b(x)\rangle .
\]
These depend on $x$ only through $r=b(x)$. This is by
\Cref{ass:radial} (\S\ref{sec:sturm}), which establishes this
radiality from Assumption~\ref{ass:transitive}. We write
$\rho=\rho(r)$ and $\tau=\tau(r)$.

\begin{remark}
Let $u=p(b(x))$ be a radial function. The chain rule gives
\[
  \nabla u = p'(b)\nabla b, \qquad
  \Delta u = p''(b)|\nabla b|^2 + p'(b)\Delta b.
\]
With $\rho,\tau$ as above, we obtain
\begin{equation}\label{eq:radialLf}
  L_fu \;=\; \Delta u - \langle\nabla f,\nabla u\rangle
  \;=\; \rho(r)\,p''(r) + \tau(r)\,p'(r).
\end{equation}

Apply this to $u_k=p_k(b)$ and use $L_fu_k=-\lambda_ku_k$. We get
\begin{equation}\label{eq:radialeigeneq}
  \rho(r)\,p_k''(r) + \tau(r)\,p_k'(r) + \lambda_k\,p_k(r) \;=\; 0.
\end{equation}
This is the same reduction used in \S\ref{sec:sturm}.
\end{remark}

\begin{lemma}\label{lem:gequals1}
Suppose that H1--H3, \Cref{ass:transitive}, and exact radial
polynomiality of type $\beta\in\{1,2\}$ (\Cref{def:exactpoly}) hold.
For $\beta=1$ there are constants $\rho_0>0$, $P$, and $\tilde B$ such
that
\[
  \rho(r)=\rho_0,\qquad \tau(r)=P+\tilde Br.
\]
For $\beta=2$ there are constants $\rho_0>0$, $\rho_{-2}$,
$\tilde A$, and $\tilde B$ such that
\[
  \rho(r)=\rho_0+\frac{\rho_{-2}}{r^2},\qquad
  \tau(r)=\frac{\tilde A}{r}+\tilde Br-\frac{\rho_{-2}}{r^3}.
\]
In the Hermite-type case, translating the one-dimensional ODE
coordinate removes $P$. The geometric coordinate $b$ is not
translated.
\end{lemma}

\begin{proof}
Write $s=r^\beta$ and $q_k(s)=p_k(r)$. Then $q_k$ has degree $k$.
The chain rule gives
\[
  p_k'=\beta r^{\beta-1}q_k',\qquad
  p_k''=\beta(\beta-1)r^{\beta-2}q_k'
  +\beta^2r^{2\beta-2}q_k''.
\]
Thus
\begin{equation}\label{eq:radialSLgeneral}
  \hat\sigma(s)q_k''+\hat\tau(s)q_k'+\lambda_kq_k=0,
\end{equation}
where
\[
  \hat\sigma(s)=\beta^2r^{2\beta-2}\rho(r),\qquad
  \hat\tau(s)=\beta(\beta-1)r^{\beta-2}\rho(r)
  +\beta r^{\beta-1}\tau(r).
\]

Write $q_1(s)=D_1s+E_1$, with $D_1\ne0$. Since $q_1''=0$,
\eqref{eq:radialSLgeneral} gives
\[
  \hat\tau(s)=-\lambda_1\left(s+\frac{E_1}{D_1}\right).
\]
The function $\hat\tau$ is affine, with leading coefficient $-\lambda_1$.
Now write $q_2(s)=F_2s^2+G_2s+H_2$, with $F_2\ne0$. The $k=2$
equation gives
\[
  \hat\sigma(s)=
  \frac{-\hat\tau(s)(2F_2s+G_2)-\lambda_2(F_2s^2+G_2s+H_2)}{2F_2}.
\]
The coefficient $\hat\sigma$ is a polynomial of degree at most two.

If $\beta=1$, then $\hat\sigma=\rho$. By H2 it is bounded above and
below on the tail. A bounded polynomial on a half-line is constant,
so $\rho\equiv\rho_0>0$. Since $\hat\tau=\tau$ is affine,
$\tau=P+\tilde Br$.

Suppose $\beta=2$. Then $\hat\sigma(s)=4s\rho(r)$. Since $\rho$ is
bounded above and below, $\hat\sigma(s)/s$ is bounded above and below
by positive constants as $s\to\infty$. The polynomial $\hat\sigma$ has no quadratic term, but its linear coefficient is
positive. Hence
\[
  \hat\sigma(s)=A_0+A_1s,\qquad A_1>0.
\]
It follows that
\[
  \rho(r)=\frac{A_1}{4}+\frac{A_0}{4r^2}
  =:\rho_0+\frac{\rho_{-2}}{r^2}.
\]
Write $\hat\tau(s)=P_0+Ds$. Solving
$\hat\tau=2\rho+2r\tau$ for $\tau$ gives
\[
  \tau(r)=\frac{P_0-2\rho_0}{2r}+\frac D2r
  -\frac{\rho_{-2}}{r^3}.
\]
This has the stated form.
\end{proof}

\subsection{Bochner's classification}

Bochner's theorem \cite{Bochner1929} classifies the classical
orthogonal polynomial systems as eigenfunctions of second-order
operators. The normal forms are given below. See also \cite[\S\S18.3, 18.8]{DLMF}.

\begin{theorem}\label{thm:bochner}
Let $\sigma\not\equiv0$ and $\tau$ be real polynomials of degrees
at most two and one. Suppose $\sigma p_k''+\tau p_k'+\mu_kp_k=0$ has a polynomial solution
of degree $k$ for every $k\ge0$. Assume these polynomials form a complete
orthogonal system for a positive weight on an interval.
Up to an affine change of variable and a constant scaling of the
operator, the coefficient pairs are
\[
\begin{array}{ll}
\text{Hermite:}&(\sigma,\tau)=(1,-2r),\\
\text{Laguerre:}&(\sigma,\tau)=(r,a+1-r),\quad a>-1,\\
\text{Jacobi:}&(\sigma,\tau)=(1-r^2,b-a-(a+b+2)r),\quad a,b>-1.
\end{array}
\]
\end{theorem}

The proof below identifies the coefficients directly and does not
invoke this classification.

\subsection{Main theorem}

\begin{theorem}\label{thm:inversehermite}
Suppose $M^n$ satisfies H1--H4 with any $\alpha>0$, and TS1
(\Cref{ass:transitive}). Assume $\{u_k\}$ satisfies exact radial
polynomiality of type $\beta\in\{1,2\}$ (\Cref{def:exactpoly}). Then
$\alpha=2$, so $f=b^2$ on the tail. For $\beta=1$, an affine change
of the one-dimensional variable puts the invariant-sector equation in
Hermite form. For $\beta=2$, an affine change of $s=b^2$ puts it in
generalized Laguerre form.
\end{theorem}

\begin{proof}
We work first in the geometric coordinate $r=b$. Then $f=r^\alpha$
on the tail.

\emph{Step 1. The coefficients.} By \Cref{lem:gequals1}, there are
constants $\rho_0>0$, $\delta$, $\tilde A$, $P$, and $\tilde B$ such
that
\begin{equation}\label{eq:inversecoeffs}
  \rho(r)=\rho_0+\delta r^{-2},\qquad
  \tau(r)=\frac{\tilde A}{r}+P+\tilde Br-\delta r^{-3}.
\end{equation}
For $\beta=1$, $\delta=\tilde A=0$. For $\beta=2$, $P=0$, and
$\tilde B<0$ in both cases. The leading coefficient of $\hat\tau$ is
$-\lambda_1$. Hence $\tilde B=-\lambda_1$ for $\beta=1$ and
$\tilde B=-\lambda_1/2$ for $\beta=2$.

\emph{Step 2. The geometric identities.} Since $f=r^\alpha$,
\[
  \langle\nabla f,\nabla b\rangle
  =\alpha r^{\alpha-1}\rho(r),\qquad
  |\nabla f|^2=\alpha^2r^{2\alpha-2}\rho(r).
\]
Also
\[
  \Delta b=\tau+\alpha r^{\alpha-1}\rho,
\]
and
\[
  \Delta f=\alpha(\alpha-1)r^{\alpha-2}\rho
  +\alpha r^{\alpha-1}\Delta b.
\]

\emph{Step 3. The case $\alpha<2$.} By \eqref{eq:inversecoeffs},
$\rho(r)=\rho_0+O(r^{-2})$ and $\tau(r)=\tilde Br+O(1)$. Hence
\[
  \Delta b=\tilde Br+o(r).
\]
Since $\tilde B<0$, this tends to $-\infty$, contrary to
\Cref{lem:deltab}.

\emph{Step 4. The case $\alpha>2$.} The identities in Step 2 and
\eqref{eq:inversecoeffs} give
\[
  \Delta f-\eta|\nabla f|^2
  =(1-\eta)\alpha^2\rho_0r^{2\alpha-2}
  +\alpha\tilde B r^\alpha+o(r^{2\alpha-2}).
\]
Because $2\alpha-2>\alpha$ and $\eta<\tfrac12$, the first term is
positive and dominant. The left side tends to $+\infty$, contrary to
H3, so $\alpha=2$.

\emph{Step 5. Hermite type.} Let $\beta=1$. Then
$\rho\equiv\rho_0$ and $\tau=P+\tilde Br$. Put
$\tilde r=r/\sqrt{\rho_0}$. The radial equation becomes
\[
  p''+\left(\frac{P}{\sqrt{\rho_0}}+\tilde B\tilde r\right)p'
  +\lambda p=0.
\]
Since $p_k$ has degree $k$, comparison of the leading coefficient
gives $\lambda_k=-\tilde Bk$. Translate
\[
  y=\tilde r+\frac{P}{\tilde B\sqrt{\rho_0}}
\]
and then set $x=\sqrt{-\tilde B/2}\,y$. The equation becomes
\[
  h''-2xh'+2kh=0.
\]

\emph{Step 6. Laguerre type.} Let $\beta=2$ and $s=r^2$. In the
notation of the proof of \Cref{lem:gequals1},
\[
  \hat\sigma(s)=A_0+A_1s,\qquad A_1>0,
  \qquad
  \hat\tau(s)=P_0+Ds,
\]
with $D=-\lambda_1<0$. The equation is
\[
  (A_0+A_1s)q_k''+(P_0+Ds)q_k'+\lambda_kq_k=0.
\]
Since $q_k$ has degree $k$, comparison of the coefficient of $s^k$
gives
\[
  \lambda_k=-Dk.
\]
Set $s_0=A_0/A_1$, $z=s+s_0$, and
\[
  x=-\frac{D}{A_1}z,
  \qquad
  a+1=\frac{P_0-Ds_0}{A_1}.
\]
Because $-D/A_1>0$, this is an affine orientation-preserving change
on the tail. The equation becomes
\[
  xy''+(a+1-x)y'+ky=0,
\]
the generalized Laguerre equation.
\end{proof}

\begin{remark}\label{rem:bochnervscompat}
\Cref{thm:inversehermite} does not assume compatibility,
\Cref{def:compatibility}, the Lower Localization Property,
\Cref{def:LLP}, or any estimate of
\S\S\ref{sec:freq}--\ref{sec:localization}. It uses no result
restricted to $\alpha>1$, see \S\ref{sec:small}. Its hypothesis is
stronger than H4 in one specific way. Exact radial polynomiality of
type $\beta$ implies $\gamma_k=\beta k$ for the invariant sector. H4
only asks that $\gamma_k<\infty$ for each $k$ and
$\gamma_k\to\infty$. That bounds the growth order of $u_k$. It does not require $u_k$ to be a polynomial or, in particular,
a radial polynomial of a fixed degree.

The hypothesis is placed on the $G$-invariant sector because radial
functions span only the invariant subspace. Even when each level set
is a single $G$-orbit, radial functions do not span the full
$L^2(M,e^{-f}\dV)$ when the level sets have positive dimension.
\end{remark}

\section{Which Results Extend to \texorpdfstring{$0<\alpha\le1$}{0 < alpha <= 1}}\label{sec:small}

The exponent $\alpha$ is introduced in H1, and H1--H3 are stated for
every $\alpha>0$. Several results above require $\alpha>1$ because
they use confinement of the Schr\"odinger potential $V$. There is also
a more basic reason that the compatibility relation itself belongs to
this range. Under H4, both $\lambda_k$ and $\gamma_k$ tend to infinity.
If one imposed
\[
  \lambda_k\asymp\gamma_k^{(2\alpha-2)/\alpha}
\]
for $0<\alpha\le1$, then the right-hand side would remain bounded when
$\alpha=1$ and would tend to zero when $0<\alpha<1$. Either conclusion
contradicts $\lambda_k\to\infty$. The power law called
Compatibility in this paper is an $\alpha>1$ relation.
What remains open in the nonconfining range is whether H1--H4 can hold
at all. This section states which definitions and results still apply there.

\begin{center}
\begin{tabularx}{\textwidth}{@{}>{\raggedright\arraybackslash}p{0.42\textwidth}c>{\raggedright\arraybackslash}X@{}}
\toprule
Component & Needs $\alpha>1$? & Reason \\
\midrule
Frequency function and Caccioppoli estimates (\S\ref{sec:freq}) & No & two-sided bounds on $\Delta b$ only \\
Concentration radius, moment bound (\S\ref{sec:twofiltrations}) & No & \Cref{lem:volmoment}, H4, and $B_k$ \\
Schr\"odinger conjugation (\S\ref{sec:conjugate}) & No & unitary identity \\
Discrete spectrum and Agmon estimates (\S\ref{sec:conjugate}) & Yes & the confinement argument uses $V\to\infty$ \\
Turning radius, localization (\S\ref{sec:localization}) & Yes & built on Agmon decay \\
Compatibility (\S\S\ref{sec:compatibility}--\ref{sec:sturm}) & Yes & the exponent $(2\alpha-2)/\alpha$ is positive only for $\alpha>1$ \\
Inverse Hermite/Laguerre Theorem (\S\ref{sec:bochner}) & No & \Cref{lem:gequals1} and H3 only \\
\bottomrule
\end{tabularx}
\end{center}

\begin{remark}\label{rem:whyalphagt1}
Let $0<\alpha\le1$. Then $2\alpha-2\le0$ and $\alpha-1\le0$. On
$\{r\ge R_0\}$, H2 and the lower bound in H3 give
\[
  V=\tfrac14|\nabla f|^2-\tfrac12\Delta f
  \;\le\;\tfrac14c_3\,r^{2\alpha-2}+\tfrac12c_4\,r^{\alpha-1},
\]
so $V$ is bounded above on the tail, and hence on all of $M$ after
including the compact core. There is no confining potential of the form
used in \S\ref{sec:conjugate}. Hence \Cref{lem:discrete}, the Agmon
decay proved there, and the safe radius $r_{\mathrm{Ag}}(k)$ of
\Cref{def:Agmon} are unavailable. The formula defining that radius
uses the positive power $1/(2\alpha-2)$ and is not continued through
$\alpha=1$.

In this range H4 does not follow from H1--H3, and the basic model
violates it. Take $\R^n$ with $f=|x|^\alpha$ outside a compact set.
Then H1--H3 hold, but $V\to0$ if $\alpha<1$ and $V\to\tfrac14$ if
$\alpha=1$. By \Cref{lem:conjugation} and Weyl's theorem, $L_f$ has
nonempty essential spectrum, so H4 fails. Whether some $(M,g,f)$
satisfies H1--H4 with $0<\alpha\le1$ is not settled here.
\end{remark}

\begin{remark}\label{rem:whatextends}
\Cref{lem:deltab,lem:deltabUpper} are proved by a case split at
$\alpha=1$ and hold for every $\alpha>0$. The frequency function
results of \S\ref{sec:freq}
(\Cref{lem:Iprime,lem:logI,lem:Ulb} and \Cref{prop:Ucacc}) and the Caccioppoli
estimates use only these two bounds. The upper bound on the
concentration radius, \Cref{prop:moment}, uses \Cref{lem:volmoment}
and H4, both valid for every $\alpha>0$.

The proof of \Cref{thm:inversehermite} uses
\Cref{ass:radial,lem:gequals1,lem:properness,lem:deltab} and H3.
None of these is restricted to $\alpha>1$. For $0<\alpha\le1$,
completeness of the invariant eigenbasis is part of H4 and TS1. It
does not follow from H1--H3, \Cref{rem:completeness}. By contrast,
\Cref{thm:osccount} needs $\alpha>1$, since it is stated relative to
$r_{\mathrm{Ag}}(k)$.
\end{remark}

\section{Examples and Open Problems}\label{sec:examples}\label{sec:open}\label{sec:ruledout}

\subsection{Examples}

\begin{example}\label{ex:OUsector}
On $\R^n$ with $f(x)=|x|^2/4$, H1--H3 hold with $\alpha=2$,
$c_1=c_2=1/4$, $c_3=c_5=1$, any $c_4>0$, and any $\eta\in(0,\tfrac12)$
once $R_0\ge\sqrt{2n/\eta}$. The full eigenbasis, products of
one-variable Hermite polynomials, is not radial. The relevant instance
of \Cref{ass:radial} comes from $G=O(n)$, which preserves $f$ and
acts transitively on each sphere $\Sigma_r$. Here $b=|x|/2$. The
$G$-invariant sector is spanned by
$v_j(x)=L_j^{(n/2-1)}(|x|^2/4)$, where $L_j^{(\nu)}$ denotes the
generalized Laguerre polynomial of degree $j$ and parameter $\nu$.
The polynomial $v_j$ has degree $j$ in $|x|^2$ and degree $2j$ in $b$. Hence
$\gamma_j=2j$, and $\mu_j=j$. This sector
satisfies H4 on its own, \Cref{rem:invariantH4}, and exact radial
polynomiality of type $\beta=2$, \Cref{def:exactpoly}. At $n=1$ the
sector consists of the even Hermite polynomials, and
$L_j^{(-1/2)}(x^2/4)$ is the same statement. This paper contains no
example of type $\beta=1$.

The radial turning scale can also be computed explicitly. Since
$V=b^2/4-n/4$, the level $V=\mu_j$ occurs at
$b=2\sqrt{j+n/4}$. The turning scale is $b\asymp\sqrt j$, and
the safe radius of \Cref{def:Agmon} is comparable to $\sqrt j$. This is
the same scale as the radial zero count in \Cref{thm:osccount}.
\end{example}

\begin{example}\label{rem:hermiteopen}
Let $n\ge2$, $\rho_0>0$, and $M=\R^n$ with the rotationally symmetric metric
\[
  g=dr^2+\varphi(r)^2g_{S^{n-1}},\qquad
  \varphi(r)=r\,e^{Br^2/(2(n-1))}.
\]
Since $\varphi(r)/r$ is a smooth function of $r^2$ equal to $1$
at the origin, the metric extends smoothly across the pole. Take $f=\rho_0r^2$, so $b=\sqrt{\rho_0}\,r$ and
$\rho\equiv\rho_0$. A
direct computation gives
\[
  \tau=\frac{(n-1)\rho_0}{b}+(B-2\rho_0)\,b ,
\]
the Laguerre form of \Cref{lem:gequals1} with $\tilde A=(n-1)\rho_0$
and $\tilde B=B-2\rho_0$. H1--H3 hold with $\alpha=2$ if and only if
$0\le B<2\eta\rho_0$. In $s=r^2$ the radial equation is
$4sq''+\bigl(2n+2(B-2\rho_0)s\bigr)q'+\mu q=0$, and the $G$-invariant
sector consists of the Laguerre polynomials $L_j^{(n/2-1)}$ in
$(2\rho_0-B)s/2$, with $\mu_j=2(2\rho_0-B)\,j$ and $\gamma_j=2j$.
The case $B=0$ is the Euclidean metric. Every admissible $B>0$
produces a nonflat rotationally symmetric example of type $\beta=2$.
\end{example}

\begin{example}
On $\mathbb H^n$, take a smooth radial weight that agrees with
$f(r)=(n-1)r+\tfrac c2r^2$ for $r\ge1$, where $c>0$ and $r$ is the
distance to a fixed point. The leading term is quadratic,
so $\alpha=2$, and H1--H3 hold. The linear term cancels the
exponential volume growth,
$e^{-f(r)}\vol(\partial B_r)\sim e^{-(n-1)r-cr^2/2}\,e^{(n-1)r}=e^{-cr^2/2}$,
leaving a Gaussian tail. H1--H3 allow exponential volume growth,
provided the leading term of $f$ confines at the quadratic rate.
\end{example}

The next example is on $\R^n$. Its weights lie outside H1, so H4 does
not apply to them as stated. We instead examine the asymptotic
polynomial-growth condition in H4. A logarithmic perturbation of the
quadratic scale in either direction makes this condition fail in the
radial sector. These examples suggest that the quadratic scale is
special when measured by the radial drift $\Theta=f'-V_r$.

\begin{example}\label{ex:probes}
Set $r=|x|$ and consider on $\R^n$
\[
  f_+(r):=(1+r^2)\log(2+r^2),
  \qquad
  f_-(r):=\frac{1+r^2}{\log(2+r^2)}.
\]
Hence $f_+\asymp r^2\log r$ and $f_-\asymp r^2/\log r$. Neither is
comparable to a fixed power $r^\alpha$, so neither belongs to H1 for
any $\alpha>0$. They lie on opposite sides of the quadratic scale and
differ from it only by a logarithmic factor.

We measure radial polynomial growth on Euclidean spheres by
\[
  \widehat\gamma(u)
  :=\max\left\{0,\limsup_{r\to\infty}
  \frac{\log |u(r)|}{\log r}\right\}
\]
for radial eigenfunctions. In the radial sector, polynomial growth fails in different ways on
the two sides of the quadratic scale.

For either weight, put
\[
  \Theta_\pm(r):=f_\pm'(r)-\frac{n-1}{r},
  \qquad
  W_\pm(r):=r^{n-1}e^{-f_\pm(r)}.
\]
A radial eigenfunction satisfies
\[
  u''-\Theta_\pm u'+\lambda u=0,
  \qquad
  \frac1{W_\pm}(W_\pm u')'=-\lambda u.
\]
Direct differentiation gives
\[
  \Theta_+(r)=4r\log r+2r+O(r^{-1}),\qquad
  \Theta_+'(r)=4\log r+O(1),
\]
and
\[
  \Theta_-(r)=\frac{r}{\log r}
  +O\!\left(\frac{r}{(\log r)^2}\right),\qquad
  \Theta_-'(r)=O\!\left(\frac1{\log r}\right).
\]
Both $\Theta_\pm$ tend to infinity. The corresponding radial
Liouville potentials
\[
  Q_\pm=\frac14\Theta_\pm^2-\frac12\Theta_\pm'
\]
tend to $+\infty$, so the radial sector has compact resolvent and an
infinite sequence of eigenvalues.

Fix a positive radial eigenvalue $\lambda$ and an associated
$L^2(e^{-f_\pm}dx)$ eigenfunction. Put
$\psi:=\sqrt{W_\pm}\,u$. Then
$-\psi''+Q_\pm\psi=\lambda\psi$. On a sufficiently far tail,
$Q_\pm-\lambda>0$. Since $\psi\in L^2(dr)$, the convexity argument
of \Cref{lem:wpnonosc} shows that $\psi$ has fixed sign there and
$\psi\psi'<0$. Hence $u$ has fixed sign there and $w:=u'/u$ is
well defined. It satisfies
\[
  w'=\Theta_\pm w-w^2-\lambda.
\]
Set
\[
  D:=\sqrt{\Theta_\pm^2-4\lambda},\qquad
  w_-:=\frac{\Theta_\pm-D}{2}.
\]
For large $r$, $D\asymp\Theta_\pm$ and
\[
  w_-=\frac{\lambda}{\Theta_\pm}
      +O\!\left(\frac{\lambda^2}{\Theta_\pm^3}\right),
  \qquad
  |w_-'|\le C\lambda\frac{|\Theta_\pm'|}{\Theta_\pm^2}.
\]
Moreover,
\[
  w=\frac{\psi'}{\psi}+\frac{\Theta_\pm}{2}
  <\frac{\Theta_\pm}{2}=w_-+\frac D2.
\]
For the two explicit drifts, choose the decreasing functions
\[
  \tau_+(r):=K\lambda\frac1{r^3(\log r)^2},
  \qquad
  \tau_-(r):=K\lambda\frac{(\log r)^2}{r^3}.
\]
With $K$ large enough, $|w_-'|\le \tau_\pm D$ on the tail. Set
$h:=w-w_-$. Then
\[
  h'=h(D-h)-w_-'.
\]
If $h>2\tau_\pm$ at some sufficiently large point, then, as long as
$2\tau_\pm<h<D/2$,
\[
  h'\ge \tfrac12D(h-2\tau_\pm)>0.
\]
Since $\tau_\pm$ is decreasing, $h-2\tau_\pm$ then grows at least
exponentially in $\int D$. This contradicts $h<D/2$, so
$h\le2\tau_\pm$. For $g:=w_--w$ one has
\[
  g'=g(D+g)+w_-'.
\]
If $g>\tau_\pm$, then $g-\tau_\pm$ grows at least exponentially in
$\int D$. Hence $w$ becomes negative. The Riccati equation then gives
$w'\le-w^2$, so $u$ reaches a finite zero, contrary to the fixed-sign
tail. Thus $g\le\tau_\pm$, and
\[
  w=w_-+O(\tau_\pm)
  =\frac{\lambda}{\Theta_\pm}
   +O\!\left(\frac{\lambda^2}{\Theta_\pm^3}\right)
   +O(\tau_\pm).
\]
For both weights the two error terms are integrable at infinity.
Consequently
\begin{equation}\label{eq:probeasymptotic}
  \log|u(r)|=\lambda\int^r\frac{ds}{\Theta_\pm(s)}+O(1).
\end{equation}
For $f_+$,
\[
  \int^r\frac{ds}{\Theta_+(s)}
  =\frac14\log\log r+O(1),
\]
so every radial eigenfunction with positive eigenvalue has
\[
  |u(r)|=(\log r)^{\lambda/4+o(1)},
  \qquad \widehat\gamma(u)=0.
\]
The radial polynomial-growth orders do not tend to infinity.
For $f_-$,
\[
  \int^r\frac{ds}{\Theta_-(s)}
  =\frac12(\log r)^2+o((\log r)^2),
\]
and hence
\[
  \frac{\log|u(r)|}{\log r}\longrightarrow\infty
  \qquad(\lambda>0).
\]
Every radial eigenfunction with positive eigenvalue grows faster than every
power of $r$.

The two logarithmic perturbations give opposite behavior. All positive radial
eigenfunctions have growth order zero for $f_+$ and infinite growth order
for $f_-$. This conclusion concerns these two weights only.
\end{example}

\begin{example}
The manifolds of \Cref{thm:H4notalpha2}, with $n\ge4$ and
$\alpha=n-1$, satisfy H1, H2 and H4 but not H3. There $\alpha\ne2$,
but the drift has effective exponent $2$ in the sense of
\Cref{def:beta}. On the tail the radial equation is the
Ornstein--Uhlenbeck equation up to an exponentially decaying angular
term, and $\gamma_k=\lambda_k/c$. This shows that H4 alone does not give
$\alpha=2$.
\end{example}

\subsection{Open problems}

\begin{openquestion}\label{q:sharpened}
Under H1--H4 with $\alpha>1$, which additional hypotheses give a
constant $C_D\ge1$, independent of $k$, such that the local mass-ratio bound
\eqref{eq:localdoubling} holds at the turning radius,
\begin{equation}\label{eq:sharpenedcondition}
  \sup_{s\in[R/2,\,3R]} m_k(s) \;\le\; C_D\,m_k(R^\dagger),
  \qquad R=r_{\mathrm{Ag}}(k),
\end{equation}
with $R^\dagger\in[R,2R]$ as in \Cref{prop:Ucacc}, for all sufficiently
large $k$?
\end{openquestion}

\begin{remark}\label{rem:sharpenedequiv}
The mass-ratio bound in \Cref{q:sharpened} gives
$|U(R^\dagger)|\le CC_D(r_{\mathrm{Ag}}\sqrt{\lambda_k}+1)$ by
\Cref{rem:doubling}. This is an absolute frequency bound at one
radius, not a lower bound for the growth order.
\end{remark}

\begin{openquestion}\label{q:gammalambda}
Which additional geometric hypotheses ensure
$\lambda_k\lesssim\gamma_k^{(2\alpha-2)/\alpha}$ for the full spectrum?
\Cref{prop:incompatibility} shows that radial drift regularity alone
does not suffice. The compatible subsequence question is
\Cref{q:compatiblesubsequence}.
\end{openquestion}

\begin{openquestion}\label{q:linear}
Which geometric hypotheses give a Weyl-type law for the growth
filtration $\dim\mathcal P_N$? H6 provides only the lower bound
$\dim\mathcal P_N\gtrsim N$. If $\gamma_k\asymp k$ for the entire
sequence, then $\dim\mathcal P_N\asymp N$.
The full Gaussian example in dimension $n$ has
$\dim\mathcal P_N\asymp N^n$.
\end{openquestion}

\begin{openquestion}\label{q:gaussianweight}
Under H1--H4, what additional conditions imply
\[
  -\log\bigl(\vol(\Sigma_r)e^{-r^\alpha}\bigr)\asymp r^2?
\]
\Cref{thm:betarigidity} provides the corresponding radial conclusion
when the effective drift has a regular power-law exponent.
Can a comparable conclusion be obtained without an invariant radial
sector?
\end{openquestion}

\begin{openquestion}\label{prob:hermiterealize}
Does there exist $(M,f)$ satisfying H1--H4 and TS1 whose
$G$-invariant sector satisfies exact radial polynomiality of type
$\beta=1$, \Cref{def:exactpoly}? \Cref{rem:hermiteopen} provides a
family of type $\beta=2$. No example of type $\beta=1$ is known to
us.
\end{openquestion}

\subsection{Further geometric hypotheses}

H1--H3 bound $\Delta b$, but not the full Hessian of $b$ or the
Bakry--\'Emery Ricci tensor. Arguments using these quantities need
additional hypotheses. The angular contribution in
\Cref{prop:incompatibility} is a separate obstruction to a uniform
upper spectral comparison.

For semiclassical Schr\"odinger operators in one dimension on an
interval, Laurent and L\'eautaud \cite{LaurentLeautaud} prove uniform
upper and lower bounds for eigenfunction densities. They assume
Dirichlet boundary conditions and a potential with a single well.
Their hypotheses differ from the general setting here. Their result is
not used to infer localization under H1--H4.

\appendix
\section{The frequency function: further identities and a Caccioppoli bound}\label{sec:appendix}

This appendix contains further identities for the frequency function
of \S\ref{sec:freq}. The main rigidity proofs do not use them.
The Caccioppoli estimate below motivates the condition in
\Cref{q:sharpened}. We also bound the number of level sets contained
in a nodal set and explain why maximizing the mass density does not
by itself determine $r_C(k)$.

\subsection{Derivative formulas and the logarithmic identity}

\begin{lemma}\label{lem:Iprime}
For almost every regular $r > R_0$:
\begin{equation}\label{eq:Iprime}
  I'(r) \;=\; 2D(r) - V_r(r)\,I(r)
    + \frac{1}{\vol(\Sigma_r)}\int_{\Sigma_r}
      u^2\,\frac{\Lap b}{\abs{\grad b}}\,d\sigma,
\end{equation}
where $V_r(r) := (\log\vol(\Sigma_r))'$ is the logarithmic
derivative of the level-set volume.
\end{lemma}

\begin{proof}
Let $J(r) = \int_{\Sigma_r} u^2\abs{\grad b}\,d\sigma$,
so $I = J/\vol(\Sigma_r)$. Since $b$ is smooth and has no critical
points on $\{b\ge R_0\}$, apply the divergence theorem to the slab
$\{R_0<b<r\}$. The inner boundary contributes a term independent of
$r$, and hence
\[
  J(r)-J(R_0)
  =\int_{\{R_0<b<r\}}\mathrm{div}(u^2\grad b)\,dV.
\]
Since $\mathrm{div}(u^2\grad b)=2u\ip{\grad u}{\grad b}+u^2\Lap b$,
differentiating in $r$ with Lemma~\ref{lem:coarea} gives
\[
  J'(r) = \int_{\Sigma_r}
    \frac{2u\ip{\grad u}{\grad b} + u^2\Lap b}{\abs{\grad b}}\,d\sigma
  = 2\int_{\Sigma_r} u\,\bdy u\,d\sigma
    + \int_{\Sigma_r} u^2\frac{\Lap b}{\abs{\grad b}}\,d\sigma.
\]
By definition~\eqref{eq:D}, $\int_{\Sigma_r} u\bdy u\,d\sigma =
\vol(\Sigma_r)\,D(r)$, so $J'(r) = 2\vol(\Sigma_r)D(r)
+ \int_{\Sigma_r}u^2\Lap b/\abs{\grad b}\,d\sigma$. Since $I=J/\vol(\Sigma_r)$,
the quotient rule gives
\[
  I'(r) = \frac{J'(r)}{\vol(\Sigma_r)} - \frac{J(r)\,\vol(\Sigma_r)'}{\vol(\Sigma_r)^2}
  = \frac{J'(r)}{\vol(\Sigma_r)} - I(r)\,V_r(r),
\]
where we used the definition $V_r(r)=\vol(\Sigma_r)'/\vol(\Sigma_r)$
and $J/\vol(\Sigma_r)=I$ in the second term. Substituting $J'(r)$ from
above and division by $\vol(\Sigma_r)$ gives~\eqref{eq:Iprime}.
\end{proof}

\begin{lemma}\label{lem:logI}
For almost every $r>R_0$ with $I(r)>0$:
\begin{equation}\label{eq:logI}
  (\log I)'(r)
  = \frac{2U(r)}{r} + \mathrm{Err}(r) - V_r(r),
\end{equation}
where
\[
  \mathrm{Err}(r):=\frac{1}{\vol(\Sigma_r)I(r)}
  \int_{\Sigma_r}u^2\,\frac{\Delta b}{|\nabla b|}\,d\sigma.
\]
By \Cref{lem:Errexpectation}, this is an average of
$\Delta b/|\nabla b|^2$ over $\Sigma_r$.
Under the lower bound on $\Delta f$ in H3,
\begin{equation}\label{eq:Errlb}
  \mathrm{Err}(r)\ge -C_{\mathrm{err}},
  \qquad C_{\mathrm{err}}=C_b/c_0^2.
\end{equation}
\end{lemma}
\begin{proof}
Divide~\eqref{eq:Iprime} by $I(r)$ and use $D/I=U/r$. This
gives~\eqref{eq:logI}. For~\eqref{eq:Errlb}, Lemma~\ref{lem:deltab}
gives $\Delta b\ge-C_b$ on $\Sigma_r$, and $|\nabla b|\ge c_0$, so
$\Delta b/|\nabla b|\ge -C_b /c_0$. Using $|\nabla b|\ge c_0$ again,
$\int_{\Sigma_r}u^2\,d\sigma\le c_0^{-1}\int_{\Sigma_r}u^2|\nabla b|\,d\sigma
=c_0^{-1}\vol(\Sigma_r)I(r)$.
It follows that $\mathrm{Err}(r)\ge -(C_b /c_0)\cdot c_0^{-1}=-C_b /c_0^2$.
\end{proof}

For the lower bounds, we use this form of~\eqref{eq:logI}. Combining
$\log I$ and $\log\vol(\Sigma_r)$ cancels the volume term $V_r$
algebraically.

\begin{corollary}\label{cor:Jidentity}
With $J(r):=\vol(\Sigma_r)I(r)=\int_{\Sigma_r}u^2|\nabla b|\,d\sigma$, for
almost every $r>R_0$ with $J(r)>0$,
\begin{equation}\label{eq:Jidentity}
  (\log J)'(r)\;=\;\frac{2U(r)}{r}+\mathrm{Err}(r).
\end{equation}
Equivalently $(\log I)'(r)=\tfrac{2U(r)}{r}+\mathrm{Err}(r)-V_r(r)$ is
recovered by subtracting $V_r=(\log\vol(\Sigma_r))'$.
\end{corollary}

\begin{proof}
Since $\log J=\log I+\log\vol(\Sigma_r)$, we have
$(\log J)'=(\log I)'+V_r$. Substituting~\eqref{eq:logI}, the two
$V_r$ terms cancel.
\end{proof}

\begin{remark}
\Cref{lem:Ivanish} shows that $I(r_0)=0$ only when $\Sigma_{r_0}$
lies in the nodal set of $u$. Whether $I(r)$ stays bounded away from
zero in general is not settled by H1--H2 and H4. At such a zero,
$I(r_0)=I'(r_0)=0$ by \Cref{lem:Iprime}, which produces no contradiction
by itself. The Colding--Minicozzi argument for the analogous question
uses the identity $L_ff=\tfrac n2-f$ of the Gaussian case, which is
not available here. There are only finitely many such level sets. Courant's nodal domain
theorem bounds their number by the spectral counting function.
\end{remark}

\begin{lemma}\label{lem:nodallevelsets}
For a fixed eigenfunction $u_k$, let
\[
  N(\lambda_k):=\#\{j\ge0:\lambda_j\le\lambda_k\},
\]
where eigenvalues are counted with multiplicity. Then at most
$N(\lambda_k)-1$ radii $r_0\ge R_0$ have $\Sigma_{r_0}$ contained in
the nodal set $\{u_k=0\}$. Equivalently, $I_k(r_0)=0$ for at most
$N(\lambda_k)-1$ such radii.
\end{lemma}
\begin{proof}
Suppose $r_0^{(1)}<\cdots<r_0^{(m)}$ are such radii. Set
\[
  \Omega_0:=\{b<r_0^{(1)}\},\qquad
  \Omega_i:=\{r_0^{(i)}<b<r_0^{(i+1)}\}\quad(1\le i<m),
\]
and $\Omega_m:=\{b>r_0^{(m)}\}$. These are pairwise disjoint open
sets separated by nodal level sets.

The eigenfunction cannot vanish identically on any nonempty
$\Omega_i$. Otherwise unique continuation \cite{AronszajnInequality} for
$\Delta u_k=\langle\nabla f,\nabla u_k\rangle-\lambda_ku_k$
would give $u_k\equiv0$ on $M$. Each $\Omega_i$ contains a point
where $u_k\ne0$, and hence contains a nodal domain. The $m+1$ regions
contain $m+1$ distinct nodal domains.

Courant's theorem \cite{CourantHilbert1953}, allowing for multiplicity,
bounds the number
of nodal domains of any eigenfunction with eigenvalue $\lambda_k$ by
$N(\lambda_k)$. Hence $m+1\le N(\lambda_k)$, which proves the claim.
\end{proof}

\begin{remark}\label{rem:firstordersecondorder}
A single coarea differentiation of the domain integral underlying
\[
  J(r)=\int_{\Sigma_r}u^2|\nabla b|\,d\sigma
\]
uses only $\mathrm{div}(\nabla b)=\Delta b$. \Cref{lem:Iprime} is the
basic example. The boundary-flux identity of \Cref{lem:PhiUm} is likewise
first order. By contrast, differentiating the normal field itself, or
differentiating boundary quantities such as $D(r)$ a second time,
introduces $\nabla(\nabla b)=\mathrm{Hess}(b)$. The same obstruction
appears in Bochner, Mourre and Rellich--Pohozaev identities with
multipliers built from $\nabla b$. H1--H3 control $\Delta b$ but not
the full Hessian, which is why the arguments below remain first order
in the level-set geometry.
\end{remark}

\subsection{Bounds on the frequency from the logarithmic identity}

\begin{lemma}\label{lem:ErrUpper}
Under H1--H3, at every $r\ge R_0$ for which $I(r)>0$:
\begin{equation}\label{eq:Errub}
  \mathrm{Err}(r) \;\le\; \frac{C_b^+(r)}{c_0^2},
  \qquad
  C_b^+(r) :=
  \begin{cases}
  \dfrac{(1-\alpha)c_3}{\alpha^2 R_0} + \dfrac{\eta c_3}{\alpha}R_0^{\alpha-1},
    & 0<\alpha\le1,\\[3mm]
  \dfrac{\eta c_3}{\alpha}\,r^{\alpha-1}, & \alpha>1,
  \end{cases}
\end{equation}
as given by Lemma~\ref{lem:deltabUpper}. In particular, for
$0<\alpha\le1$ the bound $C_b^+(r)\equiv C_b^+$ is independent of $r$.
For $\alpha>1$ it grows like $r^{\alpha-1}$.
\end{lemma}

\begin{proof}
By Lemma~\ref{lem:deltabUpper}, $\Delta b\le C_b^+(r)$ on $\Sigma_r$,
and $|\nabla b|\ge c_0>0$. Dividing preserves the inequality,
\[
  \frac{\Delta b}{|\nabla b|} \;\le\; \frac{C_b^+(r)}{|\nabla b|}
  \;\le\; \frac{C_b^+(r)}{c_0},
\]
the last step using $C_b^+(r)\ge0$ and $|\nabla b|\ge c_0$. Hence
\[
  \int_{\Sigma_r} u^2\,\frac{\Delta b}{|\nabla b|}\,d\sigma
  \;\le\; \frac{C_b^+(r)}{c_0}\int_{\Sigma_r}u^2\,d\sigma.
\]
Using $|\nabla b|\ge c_0$ again,
\[
  \int_{\Sigma_r}u^2\,d\sigma
  \;\le\; c_0^{-1}\int_{\Sigma_r}u^2|\nabla b|\,d\sigma
  \;=\; c_0^{-1}\,\vol(\Sigma_r)\,I(r).
\]
Combining the two displays,
\[
  \int_{\Sigma_r}u^2\,\frac{\Delta b}{|\nabla b|}\,d\sigma
  \;\le\; \frac{C_b^+(r)}{c_0^2}\,\vol(\Sigma_r)\,I(r),
\]
and division by $\vol(\Sigma_r)I(r)$ gives~\eqref{eq:Errub}.
\end{proof}

The bounds of Lemmas~\ref{lem:logI} and~\ref{lem:ErrUpper} are the two extreme cases of a single averaging identity. Introduce, on each
level set $\Sigma_r$, the probability measure
\begin{equation}\label{eq:nur}
  d\nu_r \;:=\; \frac{u^2|\nabla b|\,d\sigma}{\int_{\Sigma_r}u^2|\nabla b|\,d\sigma},
\end{equation}
the unweighted level-set average
$A(r):=\fint_{\Sigma_r}\frac{\Delta b}{|\nabla b|^2}\,d\sigma$, its
deviation $\Xi:=\frac{\Delta b}{|\nabla b|^2}-A(r)$, the normalized
unweighted level-set measure $d\bar\sigma_r:=d\sigma/\vol(\Sigma_r)$,
and the density $\rho_r:=d\nu_r/d\bar\sigma_r$.

\begin{lemma}\label{lem:Errexpectation}
At every level with $I(r)>0$, and with the notation above,
\begin{equation}\label{eq:Errexpectation}
  \mathrm{Err}(r)\;=\;\int_{\Sigma_r}\frac{\Delta b}{|\nabla b|^2}\,d\nu_r,
\end{equation}
that is, $\mathrm{Err}(r)$ is the $\nu_r$-expectation of the
geometric scalar $\Delta b/|\nabla b|^2$. Equivalently,
\begin{equation}\label{eq:ErrACov}
  \mathrm{Err}(r)\;=\;A(r)\;+\;\mathrm{Cov}(r),
  \qquad
  \mathrm{Cov}(r):=\langle\Xi,\rho_r-1\rangle_{L^2(d\bar\sigma_r)}
       =\operatorname{Cov}_{\bar\sigma_r}\!\Bigl(\tfrac{\Delta b}{|\nabla b|^2},\,\rho_r\Bigr).
\end{equation}
\end{lemma}
\begin{proof}
Equation~\eqref{eq:Errexpectation} follows from
$J(r)=\int_{\Sigma_r}u^2|\nabla b|\,d\sigma=\vol(\Sigma_r)I(r)$ and
the definition of $\mathrm{Err}$. Its numerator is
$\int u^2\,\Delta b/|\nabla b|=\int(\Delta b/|\nabla b|^2)\,u^2|\nabla b|$,
and division by $\int u^2|\nabla b|$ gives~\eqref{eq:Errexpectation}.
For~\eqref{eq:ErrACov}, we have $\int\rho_r\,d\bar\sigma_r=
\int d\nu_r=1$, and $\int\Xi\,d\bar\sigma_r=0$ because $\Xi$ is the
deviation from its own $d\bar\sigma_r$-mean $A(r)$. Hence
\[
  \mathrm{Err}(r) = \int(A(r)+\Xi)\,d\nu_r
  = A(r)\int\rho_r\,d\bar\sigma_r + \int\Xi\rho_r\,d\bar\sigma_r
  = A(r) + \langle\Xi,\rho_r-1\rangle_{L^2(d\bar\sigma_r)}.
\]
Here
\[
  \int\Xi\rho_r\,d\bar\sigma_r
  =\langle\Xi,\rho_r-1\rangle+\int\Xi\,d\bar\sigma_r
  =\langle\Xi,\rho_r-1\rangle.
\]
A direct expansion of
$\mathrm{Cov}_{\bar\sigma_r}(\Delta b/|\nabla b|^2,\rho_r)$ gives the
same quantity. This is the final equality.
\end{proof}

\begin{remark}
Lemmas~\ref{lem:deltab} and~\ref{lem:deltabUpper} amount to replacing
Lemma~\ref{lem:Errexpectation}'s expectation by the infimum or
supremum of the integrand, using nothing about $u$ beyond positivity.
The dependence of $\mathrm{Err}$ on the eigenfunction is
contained in the covariance $\mathrm{Cov}(r)$. It vanishes when $\Delta b/|\nabla b|^2$ is constant on $\Sigma_r$,
as in an exactly radial level-set geometry. It also vanishes when
$\rho_r$ is constant, meaning that the boundary mass is angularly uniform. In the radial case
$u=h\circ b$ the density $\rho_r=|\nabla b|/\fint|\nabla b|$ is
independent of $k$, so $\mathrm{Cov}(r)$ is then geometric.
\end{remark}

\begin{lemma}\label{lem:Ulb}
Under H1--H3, let
\[
  \widetilde J(r) \;:=\; \int_{\Sigma_r} u^2\,d\sigma
\]
denote the (unweighted, unnormalized) boundary $L^2$-mass on $\Sigma_r$.
Then for almost every $r>R_0$ with $J(r)>0$:
\begin{equation}\label{eq:Ulb}
  U(r) \;\ge\; \frac{r}{2}\Bigl[(\log J)'(r) - \frac{C_b^+(r)}{c_0^2}\Bigr],
  \qquad J(r):=\vol(\Sigma_r)\,I(r),
\end{equation}
with $C_b^+(r)$ as in Lemma~\ref{lem:ErrUpper}, and no hypothesis on the
monotonicity of $\vol(\Sigma_r)$ is required. If $R\ge R_0$ and $J(s)>0$ for every
$s\in[R,2R]$, then
\begin{equation}\label{eq:Ulbavg}
  \int_R^{2R}\frac{2U(s)}{s}\,ds
  \;\ge\;
  \log\frac{\widetilde J(2R)}{\widetilde J(R)}
  \;+\;\log\frac{c_0}{C_0}
  \;-\; \frac{1}{c_0^2}\int_R^{2R}C_b^+(s)\,ds,
\end{equation}
and explicitly, using~\eqref{eq:Errub}:
\begin{equation}\label{eq:Ulbavgexplicit}
  \int_R^{2R}\frac{2U(s)}{s}\,ds
  \;\ge\;
  \log\frac{\widetilde J(2R)}{\widetilde J(R)}
  \;+\;\log\frac{c_0}{C_0}
  \;-\;
  \begin{cases}
  \dfrac{C_b^+\,R}{c_0^2}, & 0<\alpha\le1,\\[3mm]
  \dfrac{\eta c_3\,(2^\alpha-1)}{\alpha^2\,c_0^2}\,R^\alpha, & \alpha>1.
  \end{cases}
\end{equation}
\end{lemma}

\begin{proof}
\emph{Step 1.}
By \Cref{cor:Jidentity},
\begin{equation}\label{eq:logJexact}
  (\log J)'(r) \;=\; \frac{2U(r)}{r} + \mathrm{Err}(r).
\end{equation}
The unknown sign of $(\log\vol(\Sigma_r))'$ disappears after
combining $\log I$ and $\log\vol(\Sigma_r)$ into $\log J$.

\emph{Step 2.}
Rearrange~\eqref{eq:logJexact} and apply Lemma~\ref{lem:ErrUpper},
which gives $\mathrm{Err}(r)\le C_b^+(r)/c_0^2$. Hence
\[
  \frac{2U(r)}{r} \;=\; (\log J)'(r) - \mathrm{Err}(r)
  \;\ge\; (\log J)'(r) - \frac{C_b^+(r)}{c_0^2},
\]
which is~\eqref{eq:Ulb} after multiplying by $r/2$. This holds
unconditionally for a.e.\ regular $r> R_0$, the common domain of
Lemma~\ref{lem:logI} and Lemma~\ref{lem:ErrUpper}.

\emph{Step 3.}
Integrating~\eqref{eq:Ulb} in the form
$\tfrac{2U(s)}{s}\ge(\log J)'(s)-C_b^+(s)/c_0^2$ over $s\in[R,2R]$
gives
\begin{equation}\label{eq:JintermediateBound}
  \int_R^{2R}\frac{2U(s)}{s}\,ds
  \;\ge\; \log\frac{J(2R)}{J(R)} \;-\; \frac{1}{c_0^2}\int_R^{2R}C_b^+(s)\,ds.
\end{equation}

\emph{Step 4.}
By definition, $J(r)=\int_{\Sigma_r}u^2|\nabla b|\,d\sigma$ and
$\widetilde J(r)=\int_{\Sigma_r}u^2\,d\sigma$. Since
$c_0\le|\nabla b|\le C_0$ pointwise on $\Sigma_r$ by
Remark~\ref{rem:gradb},
\[
  c_0\,\widetilde J(r) \;\le\; J(r) \;\le\; C_0\,\widetilde J(r)
  \qquad\text{for every }r.
\]
Apply the lower bound at $r=2R$ and the upper bound at $r=R$. Then
\[
  \frac{J(2R)}{J(R)} \;\ge\; \frac{c_0\,\widetilde J(2R)}{C_0\,\widetilde J(R)},
\]
so
\[
  \log\frac{J(2R)}{J(R)} \;\ge\; \log\frac{\widetilde J(2R)}{\widetilde J(R)}
  + \log\frac{c_0}{C_0}.
\]
Substitution into~\eqref{eq:JintermediateBound} gives~\eqref{eq:Ulbavg}.
This step uses only the two-sided pointwise bounds on $|\nabla b|$
at $R,2R$. No comparison of derivatives is used. The
correction $\log(c_0/C_0)\le0$ is a fixed constant, independent of
$R$ and $k$.

\emph{Step 5.}
If $0<\alpha\le1$, $C_b^+(s)\equiv C_b^+$
is constant, so $\int_R^{2R}C_b^+(s)\,ds=C_b^+\cdot R$. If $\alpha>1$,
$C_b^+(s)=\frac{\eta c_3}{\alpha}s^{\alpha-1}$, so
\[
  \int_R^{2R}C_b^+(s)\,ds
  = \frac{\eta c_3}{\alpha}\int_R^{2R}s^{\alpha-1}\,ds
  = \frac{\eta c_3(2^\alpha-1)}{\alpha^2}\,R^\alpha.
\]
Substitution into~\eqref{eq:Ulbavg} gives~\eqref{eq:Ulbavgexplicit}.
\end{proof}

\begin{remark}
Combining $\log I$ and $\log\vol(\Sigma_r)$ as
$\log J=\log(I\cdot\vol(\Sigma_r))$ cancels the $V_r$ term before
any inequality is used. The volume
factor reappears only at the last step, in converting $J$ back to
$\widetilde J(r)=\int_{\Sigma_r}u^2\,d\sigma$. This step uses only the fixed constants $c_0,C_0$ from H2.
It assumes neither a sign for $V_r$ nor monotonicity of $\vol(\Sigma_r)$.
It adds the fixed constant $\log(c_0/C_0)\le0$ to
\eqref{eq:Ulbavg}--\eqref{eq:Ulbavgexplicit}.
It does not grow with $R$ or $k$, so it does not affect the scales
of interest, where $R\sim r_{\mathrm{Ag}}(k)\to\infty$.
\end{remark}

\begin{remark}\label{rem:Ulbstandalone}
Lemma~\ref{lem:Ulb} is a direct consequence of the logarithmic
identity and the geometric bound on $\Delta b$. The averaged form is
included for comparison. The upper estimate below and the
concentration-radius bounds of \S\ref{sec:twofiltrations} are obtained
independently.
\end{remark}

\subsection{A Caccioppoli upper bound on the frequency}

The lemmas above use only the pointwise behavior of $\Delta b$ and
$|\nabla b|$. A Caccioppoli inequality on an annulus gives a different estimate
relating $U$ to $\lambda_k$. The resulting local mass ratio is
the quantity used in \Cref{q:sharpened}.

\begin{definition}
For $r > 2R_0$, define the annuli
\[
  A(r) = \{r \le b \le 2r\},
  \qquad
  A^*(r) = \{r/2 \le b \le 3r\}.
\]
\end{definition}

\begin{lemma}\label{lem:cacc}
Under Assumptions~\ref{ass:H1}--\ref{ass:H3}, for each eigenfunction
$u_k$ with $\Lf u_k = -\lambda_k u_k$ and for all $r \ge 2R_0$:
\begin{equation}\label{eq:cacc}
  \int_{A(r)} \abs{\grad u_k}^2\,\ef\dV
  \;\le\; \Bigl(\frac{4C_\phi^2}{r^2} + 2\lambda_k\Bigr)
    \int_{A^*(r)} u_k^2\,\ef\dV,
\end{equation}
where $C_\phi > 0$ depends only on $C_0$.
\end{lemma}

\begin{proof}
Choose a one-variable cutoff $\chi$ supported in $[1/2,3]$, equal
to $1$ on $[1,2]$, and with $|\chi'|\le2$. Set
$\phi(x):=\chi(b(x)/r)$. Then $\phi\equiv1$ on $A(r)$,
$\phi\equiv0$ outside $A^*(r)$, and H2 implies
$|\nabla\phi|\le 2C_0/r$. One may take $C_\phi=2C_0$.

Multiply $\Lf u_k = -\lambda_k u_k$ by $\phi^2 u_k$ and integrate
against $\ef\dV$.
By self-adjointness of $\Lf$ in $\fdom$,
\[
  \int_M \phi^2 u_k\,\Lf u_k\,\ef\dV
  = -\int_M \ip{\grad(\phi^2 u_k)}{\grad u_k}\,\ef\dV.
\]
Expanding $\grad(\phi^2 u_k) = 2\phi u_k\grad\phi + \phi^2\grad u_k$
gives
\[
  -\int_M(2\phi u_k\ip{\grad\phi}{\grad u_k} + \phi^2\abs{\grad u_k}^2)\ef\dV
  = -\lambda_k\int_M \phi^2 u_k^2\,\ef\dV.
\]
Rearranging,
\[
  \int_M\phi^2\abs{\grad u_k}^2\ef\dV
  = -2\int_M\phi u_k\ip{\grad\phi}{\grad u_k}\ef\dV
    + \lambda_k\int_M\phi^2 u_k^2\ef\dV.
\]
Young's inequality
$\abs{2\phi u_k\ip{\grad\phi}{\grad u_k}}
  \le \frac{1}{2}\phi^2\abs{\grad u_k}^2 + 2u_k^2\abs{\grad\phi}^2$
gives
\[
  \int_M\phi^2\abs{\grad u_k}^2\ef\dV
  \le \frac{1}{2}\int_M\phi^2\abs{\grad u_k}^2\ef\dV
    + 2\int_M u_k^2\abs{\grad\phi}^2\ef\dV
    + \lambda_k\int_M\phi^2 u_k^2\ef\dV.
\]
Absorbing the $\frac{1}{2}$ term and using $\abs{\grad\phi}\le C_\phi/r$
and $\mathrm{supp}(\phi)\subset A^*(r)$, we get
\[
  \int_{A(r)}\abs{\grad u_k}^2\ef\dV
  \le \Bigl(\frac{4C_\phi^2}{r^2}+2\lambda_k\Bigr)\int_{A^*(r)}u_k^2\ef\dV.
  \qquad\qedhere
\]
\end{proof}

\begin{proposition}\label{prop:Ucacc}
Under H1--H4, for every $R\ge2R_0$ there is a regular value
$R^\dagger\in[R,2R]$ such that
\begin{equation}\label{eq:Ucacc}
  |U(R^\dagger)| \;\le\; C\bigl(R\sqrt{\lambda_k}+1\bigr)\cdot
  \frac{\displaystyle\sup_{s\in[R/2,3R]}m(s)}{m(R^\dagger)},
\end{equation}
where $m(s):=e^{-s^\alpha}\vol(\Sigma_s)I(s)$ is the weighted mass
density of \Cref{def:massdensity}, and $C=C(c_0,C_0)$ is independent
of $k$ and $R$.
\end{proposition}

\begin{proof}
Apply the Caccioppoli inequality, \Cref{lem:cacc}, at radius $R$,
\begin{equation}\label{eq:caccagain}
  \int_{A(R)}|\nabla u|^2e^{-f}\dV \;\le\; \Bigl(2\lambda_k+\frac{4C_\phi^2}{R^2}\Bigr)
  \int_{A^*(R)}u^2e^{-f}\dV, \qquad A(R)=\{R\le b\le2R\}.
\end{equation}
By the coarea formula, the left side equals
\[
  \int_R^{2R} h(s)\,ds,\qquad
  h(s):=e^{-s^\alpha}\int_{\Sigma_s}
       \frac{|\nabla u|^2}{|\nabla b|}\,d\sigma.
\]
The set of nodal levels $I(s)=0$ is finite by \Cref{lem:nodallevelsets}.
Hence one may choose a regular value $R^\dagger\in[R,2R]$ with
$I(R^\dagger)>0$ and $h(R^\dagger)$ no larger than the average of $h$
on $[R,2R]$. Hence
\[
  \int_{A(R)}|\nabla u|^2e^{-f}\dV
  \ge R\,e^{-(R^\dagger)^\alpha}
  \int_{\Sigma_{R^\dagger}}\frac{|\nabla u|^2}{|\nabla b|}\,d\sigma
  \ge \frac{R}{C_0}\,e^{-(R^\dagger)^\alpha}
  \vol(\Sigma_{R^\dagger})B(R^\dagger),
\]
where $B(r):=\vol(\Sigma_r)^{-1}\int_{\Sigma_r}|\nabla u|^2\,d\sigma$
and the last inequality uses $|\nabla b|\le C_0$.

For the right side of \eqref{eq:caccagain}, use the coarea formula
and $u^2/|\nabla b|=u^2|\nabla b|/|\nabla b|^2\le u^2|\nabla b|/c_0^2$,
which uses $|\nabla b|\ge c_0$ twice. Hence
\[
  \int_{A^*(R)}u^2e^{-f}\dV \;=\; \int_{R/2}^{3R}e^{-s^\alpha}
  \Bigl(\int_{\Sigma_s}\frac{u^2}{|\nabla b|}\,d\sigma\Bigr)ds
  \;\le\; \frac1{c_0^2}\int_{R/2}^{3R}m(s)\,ds
  \;\le\; \frac{5R}{2c_0^2}\sup_{s\in[R/2,3R]}m(s).
\]
Combining, and using $C_\phi=2C_0$ from \Cref{lem:cacc},
\[
  \frac{R}{C_0}\,e^{-(R^\dagger)^\alpha}\vol(\Sigma_{R^\dagger})B(R^\dagger)
  \;\le\; \Bigl(2\lambda_k+\frac{16C_0^2}{R^2}\Bigr)\cdot\frac{5R}{2c_0^2}
  \sup_{[R/2,3R]}m,
\]
so, since $m(R^\dagger)=e^{-(R^\dagger)^\alpha}\vol(\Sigma_{R^\dagger})I(R^\dagger)$,
\begin{equation}\label{eq:Bbound}
  B(R^\dagger) \;\le\; C\Bigl(\lambda_k+\frac1{R^2}\Bigr)\cdot
  \frac{\sup_{[R/2,3R]}m}{m(R^\dagger)/I(R^\dagger)}
  \;=\; C\Bigl(\lambda_k+\frac1{R^2}\Bigr)I(R^\dagger)\cdot
  \frac{\sup_{[R/2,3R]}m}{m(R^\dagger)}.
\end{equation}
By Cauchy--Schwarz on $\Sigma_{R^\dagger}$, using $|\nabla b|\ge c_0$,
\begin{equation}\label{eq:CSforD}
  |D(R^\dagger)| \;\le\; \sqrt{c_0^{-1}I(R^\dagger)\,B(R^\dagger)}.
\end{equation}
This bounds $D(R^\dagger)$ from both sides,
$-\sqrt{c_0^{-1}I(R^\dagger)B(R^\dagger)}\le D(R^\dagger)\le
\sqrt{c_0^{-1}I(R^\dagger)B(R^\dagger)}$. Since $U(R^\dagger)=
R^\dagger D(R^\dagger)/I(R^\dagger)$ and $R^\dagger\le2R$, dividing
\eqref{eq:CSforD} by $I(R^\dagger)$ and multiplying by $R^\dagger$
bounds $|U(R^\dagger)|$,
\[
  |U(R^\dagger)| \;\le\; 2R\sqrt{\frac{B(R^\dagger)}{c_0 I(R^\dagger)}}
  \;\le\; C'R\sqrt{\Bigl(\lambda_k+\frac1{R^2}\Bigr)Y}, \qquad
  Y:=\frac{\sup_{[R/2,3R]}m}{m(R^\dagger)}.
\]
Since $R\sqrt{\lambda_k+1/R^2}=\sqrt{R^2\lambda_k+1}$, and $Y\ge1$
gives $\sqrt Y\le Y$, this is $\le C'\sqrt{R^2\lambda_k+1}\cdot Y$.
Finally, $\sqrt{R^2\lambda_k+1}\le\sqrt{R^2\lambda_k}+\sqrt1=
R\sqrt{\lambda_k}+1$, by subadditivity of the square root. Hence
$|U(R^\dagger)|\le C'(R\sqrt{\lambda_k}+1)\cdot Y$, which is
\eqref{eq:Ucacc} with constants renamed.
\end{proof}

\begin{remark}\label{rem:doubling}
\Cref{prop:Ucacc} reduces the estimate of $|U(R^\dagger)|$ to control of
one quantity, the ratio $\sup_{[R/2,3R]}m/m(R^\dagger)$. This is a local doubling-type ratio for the weighted mass density $m$
over a fixed range of radii. It measures how much the boundary
mass can vary, in relative terms, across nearby radii. If there is $C_D\ge1$, independent of $k$ and $R$, with
\begin{equation}\label{eq:localdoubling}
  \sup_{s\in[R/2,3R]}m(s) \;\le\; C_D\,m(R^\dagger)
  \qquad\text{for the } R^\dagger \text{ given by \Cref{prop:Ucacc}},
\end{equation}
then \eqref{eq:Ucacc} immediately yields
\begin{equation}\label{eq:bothsides}
  |U(R^\dagger)| \;\le\; C\,C_D\,\bigl(R\sqrt{\lambda_k}+1\bigr) \;\sim\; CC_D\,R\sqrt{\lambda_k}.
\end{equation}

The bound does not compare $\lambda_k$ with $\gamma_k$ without
further information on the frequency. The radial results use the
ODE instead. We ask which conditions imply \eqref{eq:localdoubling} in
\Cref{q:sharpened}.
\end{remark}

\subsection{Critical points of the mass density}

\begin{proposition}\label{prop:exactcrit}
At any interior maximizer $r$ of $m_k$ at which $m_k(r)>0$ and the
quantities are differentiable, $J_k'(r)/J_k(r)=f'(r)$.
\end{proposition}
\begin{proof}
We have $\log m_k=\log J_k-f$. Differentiate and set to $0$.
\end{proof}

\begin{remark}\label{rem:whynotargmax}
To obtain an explicit value from Proposition~\ref{prop:exactcrit},
we need an asymptotic estimate for $J_k'/J_k$. H4 controls $I_k$, not
$J_k=\vol(\Sigma_r)I_k$. Even for $I_k$, an asymptotic size bound
cannot be differentiated. For instance,
$I(r)=r^{2\gamma}(2+\sin r^2)$ satisfies $\log I(r)=2\gamma\log
r+O(1)$ yet $(\log I)'(r)=2\gamma/r+O(r)$. The condition $(\log J_k)'(r)=2\gamma_k/r+o(1/r)$
would be substantially stronger than H4. We use instead the
median-type radius $r_C(k)$, for which we have an unconditional
one-sided bound.
\end{remark}

\end{document}